\documentclass[a4paper, 12pt]{amsart}
\usepackage{amsmath,amsthm,amssymb}
\usepackage[T1]{fontenc}
\usepackage{url}
\usepackage{amsmath,amsthm,amssymb}
\usepackage{verbatim}
\usepackage{mathrsfs}
\usepackage{graphics}
\usepackage{graphicx}
\usepackage{subfigure}
\usepackage{hyperref}
\usepackage{mathtools}
\usepackage{color}
\usepackage{tikz}
\usepackage{tikz-cd}
\allowdisplaybreaks[1]

\newtheorem{theorem}{Theorem}[section]

\newtheorem{lemma}[theorem]{Lemma}
\newtheorem{corollary}[theorem]{Corollary}

\theoremstyle{remark}
\newtheorem{remark}[theorem]{Remark}
\newtheorem{definition}{Definition}

\numberwithin{equation}{section}

\newcommand{\R}{{\mathbb{R}}}

\newcommand{\vep}{\varepsilon}

\newcommand{\C}{{\mathbb{C}}}
\newcommand{\Z}{{\mathbb{Z}}}
\newcommand{\N}{{\mathbb{N}}}

\newcommand{\sgn}{\operatorname{sgn}}

\newcommand{\SL}{\operatorname{SL}(\psi_{\R})}

\newcommand{\Arg}{\operatorname{Arg}}

\newcommand{\norm}[1]{\left\|#1\right\|} 

\newcommand{\pa}{\operatorname{P_a}}
\newcommand{\pb}{\operatorname{P_b}}

\newcommand{\pag}{\operatorname{P_aG}}

\newcommand{\Int}{\operatorname{Int}}

\makeatletter
\newcommand{\customlabel}[2]{\protected@write\@auxout{}{\string\newlabel{#1}{{#2}{\thepage}{#2}{#1}{}}}\hypertarget{#1}{#2}}
\makeatother

\begin{document}

\title[Solving cohomological equation - part I. local obstructions]
{Solving the cohomological equation for locally hamiltonian flows, part I - local obstructions}

\author[K.\ Fr\k{a}czek]{Krzysztof Fr\k{a}czek}
\address{Faculty of Mathematics and Computer Science, Nicolaus
Copernicus University, ul. Chopina 12/18, 87-100 Toru\'n, Poland}
\email{fraczek@mat.umk.pl}

\author[M. Kim]{Minsung Kim}

\address{Centro di Ricerca Matematica Ennio De Giorgi, Scuola Normale Superiore, Piazza dei Cavalieri 3, 56126 Pisa, Italy}
\email{minsung.kim@sns.it}

\address{}
\email{}
\date{\today}

\thanks{}
\maketitle
\begin{abstract}
We study the cohomological equation $Xu=f$ for smooth locally Hamiltonian flows on compact surfaces.
The main novelty of the proposed approach is that it is used to study the regularity of the solution $u$ when the flow has saddle loops, which has not been systematically studied before. Then we need to limit the flow to its minimal components.
We show the existence and (optimal) regularity of solutions regarding the relations with the associated cohomological equations for interval exchange transformations (IETs).
Our main theorems state that the regularity of solutions depends not only on the vanishing of the so-called Forni's distributions (cf.\ \cite{Fo1,Fo3}), but also on the vanishing of families of new invariant distributions (local obstructions) reflecting the behavior of $f$ around the saddles.  Our main results provide some key ingredient for the complete solution to the regularity problem of solutions (in cohomological equations) for a.a.\ locally Hamiltonian flows (with or without saddle loops) to be shown in \cite{Fr-Ki3}.

The main contribution of this article is to define the aforementioned new families of invariant distributions $\mathfrak{d}^k_{\sigma,j}$, $\mathfrak{C}^k_{\sigma,l}$ and analyze their effect on the regularity of $u$ and on the regularity of the associated cohomological equations for IETs. To prove this new phenomenon, we further develop local analysis of $f$ near degenerate singularities inspired by tools from \cite{Fr-Ki} and \cite{Fr-Ul2}. We develop new tools of handling functions whose higher derivatives have polynomial singularities over IETs.
\end{abstract}


\section{Introduction}

Let $M$ be a smooth compact connected orientable surface of genus $g\geq 1$.
We  deal with smooth flows $\psi_\R = (\psi_t)_{t\in\R}$ on $M$ (associated to a vector field $X:M\to TM$) preserving a smooth positive measure $\mu$, i.e.\ such that for any (orientable)
choice of local coordinates $(x,y)$ we have $d\mu=V(x,y)dx\wedge dy$ with $V$ positive and smooth.
These flows are called \emph{locally Hamiltonian flows}. Indeed,  for any (orientable)
choice of local coordinates $(x,y)$ such that $d\mu=V(x,y)dx\wedge dy$,  the flow $\psi_\R$ is a local solution to the Hamiltonian equation
\[
\frac{dx}{dt} = \frac{\frac{\partial H}{\partial y}(x,y)}{V(x,y)},\quad
\frac{dy}{dt} = -\frac{\frac{\partial H}{\partial x}(x,y)}{V(x,y)}
\]
for a smooth real-valued function $H$, or equivalently $\frac{dz}{dt}=-2\iota\frac{\frac{\partial H}{\partial \overline{z}}(z,\overline{z})}{V(z,\overline{z})}$.
For general introduction to locally Hamiltonian flows, we refer readers to  \cite{Fr-Ul2, Fr-Ki, Rav, Ul:ICM}.

For any smooth observable $f:M\to\C$ we are interested in understanding the smoothness of the solution $u:M\to\C$ of the cohomological equation
\begin{equation}\label{eq:coheqfl}
u(\psi_tx)-u(x)=\int_0^tf(\psi_s x)\,ds\text{ for all }x\in M,\ t\in\R,
\end{equation}
or equivalently $Xu=f$, where $Xu(x)=\frac{d}{dt}u(\psi_tx)|_{t=0}$.

{
Cohomological equations are an important area of study in dynamical systems since they are related to smooth conjugacy problems via Kolmogorov-Arnold-Moser techniques. Here it is worth to mention the pioneering work of Marmi-Moussa-Yoccoz \cite{Ma-Mo-Yo3} on the conjugacy problem for small perturbations of translation flows.
The problem of solving cohomological equations for other classes of smooth dynamical systems of parabolic nature and the regularity of solutions using invariant distributions has a rich literature and has been studied, among other places, in \cite{Av-Fa-Ko,Av-Ko,Fo2,Fl-Fo03,Fl-Fo07,Fo4,Fo-Ma-Ma,Gu-Li,Kat,Ta,Wa}.}

\medskip

We always assume that all fixed points of the locally Hamiltonian flow $\psi_\R$ are isolated, so the set of fixed points of $\psi_\R$, denoted  by $\mathrm{Fix}(\psi_\R)$, is finite. For $g \geq 2$, $\mathrm{Fix}(\psi_\R)$ is non-empty. As $\psi_\R$ is area-preserving, fixed points are either centers, simple saddles or multi-saddles (saddles with $2k$ prongs with $k \geq 2$). We will deal only with \emph{perfect} saddles defined as follows:
a fixed point  $\sigma\in \mathrm{Fix}(\psi_\R)$ is a (perfect) saddle of multiplicity $m=m_\sigma\geq 2$ if there exists a chart $(x,y)$ (called \emph{a singular chart}) in a neighborhood $U_\sigma$ of $\sigma$ such that
$d\mu=V(x,y)dx\wedge dy$ and $H(x,y)=\Im (x+\iota y)^m$ ($(0,0)$ are coordinates of $\sigma$). Then the corresponding local Hamiltonian equation in $U_\sigma$ is of the form
\[\frac{dx}{dt}=\frac{\frac{\partial H}{\partial y}(x,y)}{V(x,y)}=\frac{m\Re(x+\iota y)^{m-1}}{V(x,y)}, \quad
\frac{dy}{dt}=-\frac{\frac{\partial H}{\partial x}(x,y)}{V(x,y)}=-\frac{m\Im(x+\iota y)^{m-1}}{V(x,y)},
\]
or equivalently
$\frac{dz}{dt}=\frac{m\overline{z}^{m-1}}{V(z,\overline{z})}$.
The set of perfect saddles of  $\psi_\R$  we denote by $\mathrm{Sd}(\psi_\R)$.

We call a \emph{saddle connection} an orbit of $\psi_\R$ running from a saddle to a saddle. A \emph{saddle loop} is a saddle connection
joining the same saddle.
We will deal only with flows such that all their saddle connections are loops.
The set consisting of all saddle loops of the flow we denote by $\SL$.

Recall that if every fixed point in $\mathrm{Fix}(\psi_\R)$ is isolated, $M$ splits into a finite number of $\psi_\R$-invariant surfaces (with boundary) so that every such surface is a \emph{minimal component} of  $\psi_\R$ (every orbit, except of fixed points and saddle loops, is dense in the component) or is a periodic component (filled by periodic orbits, fixed points and saddle loops). The boundary of each component consists of saddle loops and fixed points.

{The interest in the study of locally Hamiltonians flows in higher genus and, in particular, in their ergodic properties, was highlighted
by Novikov \cite{No} in connection with problems arising in solid-state
physics as well as in pseudo-periodic topology (see e.g.\ the survey \cite{Zo:how} by A.~Zorich).
The simplest examples of locally Hamiltonian flows  with fixed points on the torus, flows with one center and one simple saddle, motivated Arnold in \cite{Arn} to ask about mixing properties on minimal components.
We should mention that mixing properties of  locally Hamiltonian flows with simple saddles restricted to their minimal components
are fully described in a series of papers \cite{Ch-Wr,Fa-Ka-Ze,Rav,Ul:mix,Ul:wea,Ul:abs,Ul:ICM}. When perfect multi-saddles appear, mixing was proved already by Kochergin in \cite{Ko:mix}.}

\medskip

The problem of existence and regularity of solutions for the cohomological equation \eqref{eq:coheqfl} was essentially solved in two seminal articles \cite{Fo1,Fo3} by Forni.
Forni considered the case when the flow $\psi_\R$ is minimal over the whole surface $M$ and the function $f$ belongs to a certain weighted Sobolev space.
More precisely, choose a non-negative smooth function $W:M\to\R_{\geq 0}$ (with zeros at $\mathrm{Sd}(\psi_\R)$) and an Abelian $1$-form $\omega$ on $M$ (with zeros at $\mathrm{Sd}(\psi_\R)$) such that $X=WS$ and $S$ is the unit horizontal vector field on the translation surface $(M,\omega)$.
In singular local coordinates around any $\sigma\in\mathrm{Sd}(\psi_\R)$ we have $W(z,\overline{z})=|z|^{2(m_\sigma-1)}/V(z,\overline{z})$. Then for any $s>0$, $f\in H_W^s(M)$ iff $W^{-1}f\in H_\omega^s(M)$, where
$H_\omega^s(M)$ is the fractional weighted Sobolev space associated to the Abelian form $\omega$ and the related area form. For a formal definition of $H_\omega^s(M)$ and useful characterization of its smooth elements we refer the reader to Section~2 in  \cite{Fo3}.

In \cite{Fo1,Fo3}, for a.e.\ flow, Forni proved the existence of fundamental invariant distributions on $H_W^s(M)$ which are responsible for the degree of smoothness of the solution of $\eqref{eq:coheqfl}$ for $f\in H_W^s(M)$.
Roughly speaking, Forni's distributions are related to the Lyapunov exponents of the Kontsevich-Zorich cocycle on the absolute $1$-cohomological bundle. If all Forni's distributions at $f\in H_W^s(M)$ are zero then the solution $u\in H_\omega^{s'}(M)$ for some $s'<s$ with $s'$ not too far away from $s$. Forni's beautiful approach is based on a very deep analysis of the Kontsevich-Zorich cocycle acting on various kinds of abstract objects  related to translation surfaces. An alternative approach to constructing invariant distributions was also presented by Bufetov in \cite{Bu}. A different approach, based on moving to a special representation and studying renormalization behavior for piecewise smooth functions over interval exchange translations, was initiated by Marmi-Moussa-Yoccoz in \cite{Ma-Mo-Yo} and later developed in \cite{Ma-Yo,Fr-Ul2,Fr-Ki}.

The main goal of this article (and the subsequent one \cite{Fr-Ki3}) is to go beyond the case of a minimal flow on the whole surface $M$ and beyond the case of functions $f$ belonging to a weighted Sobolev space.
We deal with locally Hamiltonian flows restricted to any minimal component and $f:M\to\C$ is any smooth function.
The study of locally Hamiltonian flows in such a context gives a rise to new invariant distributions, which, unlike Forni's distributions, are local in nature. The first two new families of such invariant distributions, defined in Section~\ref{sec:defdist}, read local behaviour of functions around saddle points. The last family, which is a counterpart of Forni's distributions, is defined in \cite{Fr-Ki3} using renomalization techniques inspired by the approach developed in \cite{Ma-Mo-Yo,Ma-Yo,Fr-Ul2,Fr-Ki}.

All three families of invariant distributions affect the degree of smoothness of the solution of the cohomological equation. However, in the present article we focus only on the first two families and the main results of the paper are contained in Theorems~\ref{thm1}, \ref{thm2} and \ref{thm3}. The methods for studying their effect on the degree of smoothness are purely analytical, in contrast to the dynamical arguments left to \cite{Fr-Ki3}, where the last family play a central role.

\subsection{Special representation and IETs}
Locally Hamiltonian flows restricted to their minimal components are represented as special flows over interval exchange transformations.
Let us consider a restriction of a locally Hamiltonian flow $\psi_\mathbb{R}$ on $M$ to its minimal component $M'\subset M$.
Let $I\subset M'$ be any transversal smooth curve with its standard parametrization $\gamma:[0,|I|]\to I$, i.e.\ $\int_{0}^{\gamma(s)}\eta=s$
for $s\in [0,|I|]$, where $\eta$ is
the closed $1$-form given by $\eta=\frac{\partial H}{\partial x}dx+\frac{\partial H}{\partial y}dy$ in local coordinates. By minimality, $I$ is a global transversal and the first return map $T:I\to I$ is an interval exchange transformation (IET) in standard coordinates on $I$. We will denote by $I_\alpha$, $\alpha\in \mathcal A$ the subintervals translated by $T$.
In order to minimize the number of exchanged intervals, we will always assume that each end of $I$ is the first meeting point of a separatrix (that is not a saddle connection) emanating by a fixed point (incoming or outgoing) with the set $I$.

Let $\tau:I\to\R_{>0}\cup\{+\infty\}$ be the first return time map.
Then each point in $M'\setminus (\mathrm{Sd}(\psi_\R)\cup \SL)$ is uniquely represented as $\psi_tx$ for some $x\in I$ and $0\leq t<\tau (x)$.
The function $\tau: I \rightarrow \R_{>0}\cup\{+\infty\}$ is smooth on the interior of any exchanged interval and  has
\emph{singularities} at discontinuities of $T$. Each such discontinuity is the first hitting point (forward or backward) of a separatrix emanated by a saddle with the curve (interval) $I$. Moreover, degenerate saddles ($m_\sigma>2$) of $\psi_\R$ are responsible
for the appearance of singularities of \emph{polynomial type} and
simple saddles ($m_\sigma=2$) are responsible for the appearance of \emph{logarithmic type} singularities.

\subsection{Two crucial operators and two cohomological equations}\label{sec:twoop}
For any smooth observable $f:M\to\C$ we deal with the corresponding map $\varphi_f:I\to\C\cup\{\infty\}$ given by
\[\varphi_f(x)=\int_0^{\tau(x)}f(\psi_t x)dt.\]
The function $\varphi_f$ is smooth on the interior of any interval $I_\alpha$ and  can have polynomial or  logarithmic type singularities at discontinuities of $T$ depending on the vanishing of some invariant distributions on $f$ defined in \cite{Fr-Ki} and based on partial derivatives of $f$ at saddles in $M'$. One of the aim of this paper is a deeper understanding of the operator $f\mapsto\varphi_f$ on the kernel of all invariant distributions coming from \cite{Fr-Ki}. Then $\varphi_f$ has no singularities, but its derivatives can have. In this paper we define an infinite sequence of new (a little bit more sophisticated) invariant
distributions (based on partial derivatives at saddles) which are responsible for understanding the  regularity of $\varphi_f$.

For solving the cohomological equation \eqref{eq:coheqfl} we also need to study another operator $g\mapsto u_{g,f}$. Suppose that $g:I\to\C$ is a smooth solution (at least continuous) of the another cohomological equation
\begin{equation}\label{eq:coheqT}
g(Tx)-g(x)=\varphi_f(x) \text{ on }I.
\end{equation}
This is an obvious necessary condition for the existence of a smooth solution of the equation \eqref{eq:coheqfl}. Indeed, if $u$ is smooth and satisfies \eqref{eq:coheqfl}, then the map $g:I\to\C$ defined as the restriction of $u$ to $I$ is smooth and satisfies \eqref{eq:coheqT}. A natural problem is: when is this also a sufficient condition?

Suppose that $g:I\to\C$ is a smooth solution of \eqref{eq:coheqT}. Then the corresponding solution $u_{g,f}:M'\setminus (\mathrm{Sd}(\psi_\R)\cup \SL)\to\C$ is defined as follows.
If $\psi_tx\in I$ for some
$t\in\R$ then
\[u_{g,f}(x):=g(\psi_tx)-\int_{0}^tf(\psi_sx)\,ds.\]
By the proof of Lemma~6.3 in \cite{Fr-Ul}, the function $u_{g,f}$ is well
defined on $M'\setminus (\mathrm{Sd}(\psi_\R)\cup \SL)$. Moreover, if $M$ is a $C^\infty$-surface,  $\psi_\R$ is a $C^\infty$-flow and $f$ is a $C^\infty$-observable, then
$u_{g,f}$ is as regular as $g$. Indeed, by the absence of saddle connections joining different saddles, for every $x_0\in M'\setminus (\mathrm{Sd}(\psi_\R)\cup \SL)$
there exists $t_0\in\R$ such that $\psi_{t_0}x_0\in \Int I$. For simplicity, assume that $t_0\leq 0$. Then choose $\vep>0$ such that $[\psi_{t_0}x_0-\vep,\psi_{t_0}x_0+\vep]\subset \Int I$ and let
\begin{equation}\label{not:R}
R(x_0,t_0,\vep):=\bigcup_{-\vep\leq t\leq -t_0+\vep}\psi_t [\psi_{t_0}x_0-\vep,\psi_{t_0}x_0+\vep].
\end{equation}
If $\vep>0$ is small enough then $\nu:[-\vep,-t_0+\vep]\times[\psi_{t_0}x_0-\vep,\psi_{t_0}x_0+\vep]\to R(x_0,t_0,\vep)$ given by $\nu(t,x)=\psi_tx$ is a $C^\infty$-diffeomorphism. Moreover,
\[u_{g,f}\circ\nu(t,x)=g(x)-\int_0^tf\circ\nu(s-t,x)ds=g(x)+\int_0^tf\circ\nu(s,x)ds.\]
It follows that the regularity of $u_{g,f}$ restricted to $R(x_0,t_0,\vep)$ coincides with the regularity of $g$ on $[\psi_{t_0}x_0-\vep,\psi_{t_0}x_0+\vep]$. Since $x_0\in \Int R(x_0,t_0,\vep)$, we obtain our claim.

\medskip

However, the solution $u_{g,f}$ of the cohomological equation  is not fully satisfactory because it is defined only on an open (dense) subset of the minimal component, without fixed points and saddle loops.
Our main goal is to find necessary and sufficient conditions for the existence of a smooth solution (of the cohomological equation) defined over all of $M'$. More precisely, instead of $M'$ we will study  smooth solutions defined on the end compactification $M_e'$ of $M'\setminus \mathrm{Sd}(\psi_\R)$. Roughly speaking, if a saddle $\sigma$ emanates $l\geq 2$ loops, then $\sigma$ is the $l$-fold end of the set $M'\setminus \mathrm{Sd}(\psi_\R)$. For this reason, $\sigma$ splits in $M'_e$ into $l$ different end points $\sigma_1,\ldots,\sigma_l$, see Figure~\ref{fig:comp}.
 \begin{figure}[h!]
 \includegraphics[width=0.8\textwidth]{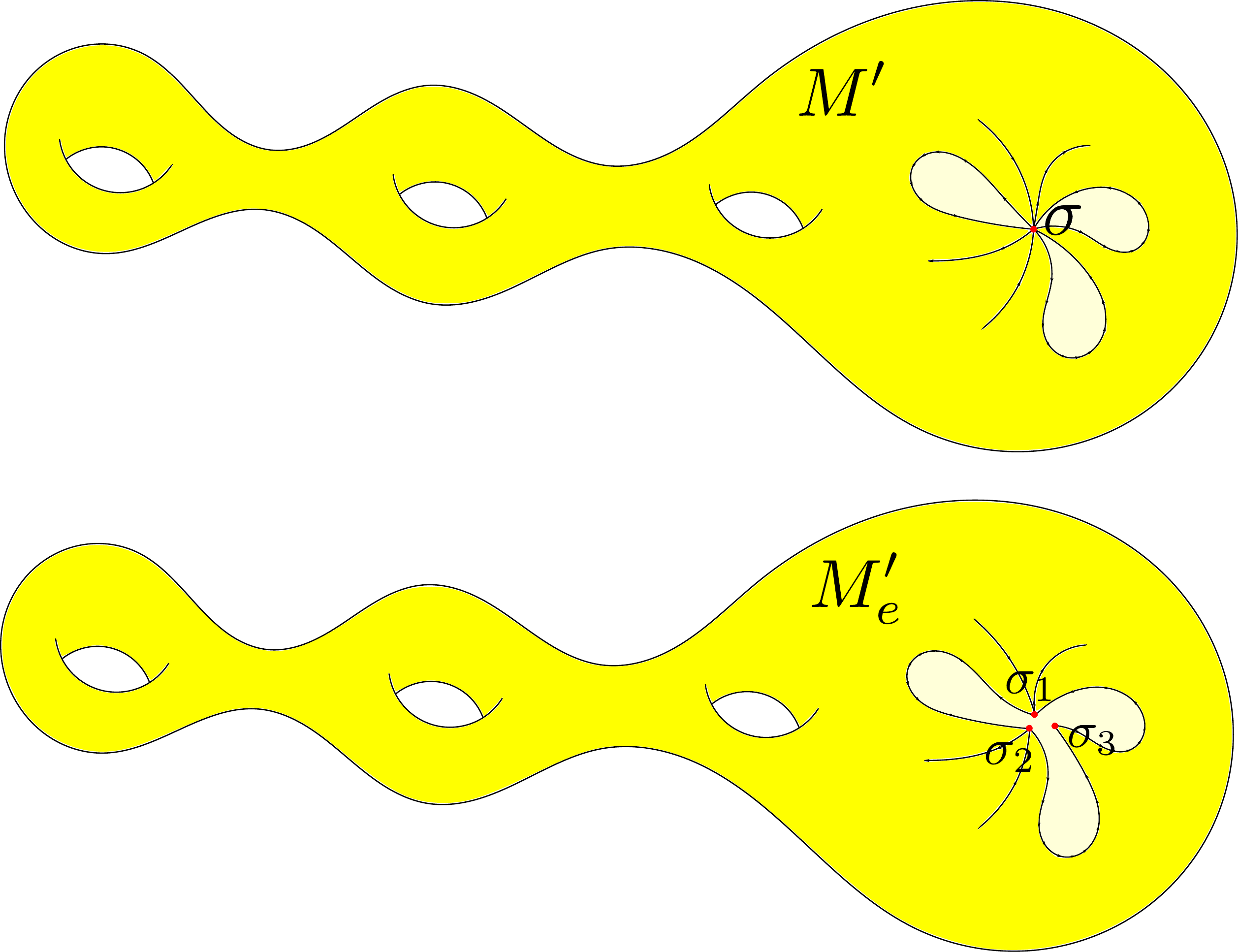}
 \caption{The minimal component $M'$ before and after separation procedure. \label{fig:comp}}
\end{figure}
We will look for smooth solutions $u:M'_e\to\C$ of \eqref{eq:coheqfl}. If a smooth solution $u:M'_e\to\C$ exists then it is smooth in a neighborhood (in $M'_e$) of any version $\sigma_i$ of the saddle point $\sigma$, but it does not even have to be continuous at $\sigma$, whenever the limits of $u$ at $\sigma$  with respect to different neighborhood sectors (connected components) are different. Of course, if each saddle emanates at most one saddle loop then $M'_e$ coincides with $M'$ and the problem of regularity of $u:M'_e\to\C$ and $u:M'\to\C$ are equivalent.

\subsection{Grading of smoothness}\label{sec:grad}
Let $M$ be a $C^\infty$-manifold with a boundary.
For any $n\in\Z_{\geq 0}$ and $0<a<1$ denote by $C^{n+a}(M)$ the space of $C^n$-functions on $M$ such that their $n$-th derivative is $a$-H\"older.
Let $\eta:\R_{\geq 0}\to\R_{\geq 0}$ be given by $\eta(x)=-x\log x$ for $x\in[0,e^{-1}]$ and $\eta(x)=e^{-1}$ for $x\geq e^{-1}$.
For any $n\in\Z_{\geq 0}$ denote by $C^{n+\eta}(M)$ the space of $C^n$-functions on $M$ such that their $n$-th derivative is continuous so that a positive multiple of $\eta$ is its modulus of continuity.
For every non-natural real $r>0$ we will write $C^r$ for $C^{\lfloor r\rfloor+\{r\}}$.

Let $\R_{\eta}:=(\R_{>-1}\setminus \Z)\cup (\Z_{\geq -1}+\{\eta\})$ and let $v:\R_{\eta}\to\R$ be given by $v(r)=r$ if $r\in (\R_{>-1}\setminus \N)$ and $v(n+\eta)=n+1$. Then $0\leq v(r)\leq v(r')$ iff $C^r\subset C^{r'}$.

\subsection{Invariant distributions}\label{sec:defdist}
To solve our main problem, in the present paper we introduce a family of invariant distributions $f\mapsto \mathfrak{d}^k_{\sigma,j}(f)$ for all $\sigma\in \mathrm{Sd}(\psi_\R)$, $k\geq 0$ and $0\leq j\leq k\wedge(m_\sigma-2)$. Throughout the article we use the notation $x\vee y=\max\{x,y\}$ and $x\wedge y=\min\{x,y\}$ for any pair of real numbers $x,y$.
 Recall that a linear bounded functional $f\mapsto \mathfrak{D}(f)$ is an \emph{invariant distribution} if $\mathfrak{D}(Xu)=0$ for any $u\in C^{\infty}(M)$. The distributions are defined locally around saddles and are obstructions to the existence of smooth solutions to the cohomological equation. The  invariant distributions $\mathfrak{d}^k_{\sigma,j}$ are defined based on the higher-order partial derivatives of the function $f$ in saddles or they are linear combinations of partial derivatives (if $k>m_\sigma-2$). We also introduce alternative versions of such invariant distributions, i.e.\ $f\mapsto \mathfrak{C}^k_{\sigma,l}(f)$ for $0\leq l< 2m_\sigma$, which have a more geometric interpretation, and generate the same space of invariant distributions as $\mathfrak{d}^k_{\sigma,j}$.

Suppose that $\sigma\in \mathrm{Sd}(\psi_\R)$ is a saddle of multiplicity $m_\sigma\geq 2$.  Fix a singular chart $(x,y)$ in a neighborhood $U_\sigma$ of $\sigma$.
Then the local Hamiltonian is of the form $H(x,y)=\Im (x+\iota y)^{m_\sigma}$ and the $\psi_\R$-invariant area-measure is $d\mu=V(x,y)dx\wedge dy$, where $V$ is  positive and  smooth.
Then for every $k\geq 0$ and $0\leq j\leq k\wedge(m_\sigma-2)$ with  $j\neq k-(m_\sigma-1)\operatorname{mod} m_\sigma$ we define the functional $\mathfrak{d}^k_{\sigma,j}:C^{k}(M)\to\C$ as follows:
\begin{equation}\label{def:gothd}
\mathfrak{d}^k_{\sigma,j}(f)=\sum_{0\leq n\leq \frac{k-j}{m_\sigma}}\frac{\binom{k}{j+nm_\sigma}\binom{\frac{(m_\sigma-1)-j}{m_\sigma}-1}{n}}{\binom{\frac{(k-j)-(m_\sigma-1)}{m_\sigma}}{n}}
\frac{\partial^k(f\cdot V)}{\partial z^{j+nm_\sigma} \partial \overline{z}^{k-j-nm_\sigma}}(0,0).
\end{equation}
Note that for $k\leq m_\sigma-2$ we have $\mathfrak{d}^k_{\sigma,j}(f)=\binom{k}{j}\frac{\partial^k(f\cdot V)}{\partial z^{j} \partial \overline{z}^{k-j}}(0,0)$, so $\mathfrak{d}^k_{\sigma,j}$ are essentially distributions defined already in \cite{Fr-Ki} to study {the} deviation spectrum of Birkhoff integrals of $f$. Let us mention that {the non-vanishing of} any of these distributions is an obstacle to the existence of any solution of the cohomological equation (even measurable).
The distributions $\mathfrak{d}^k_{\sigma,j}$ for $k\geq m_\sigma-1$ are responsible for determining {the} regularity of the solution if we already know that equation \eqref{eq:coheqfl} has a smooth solution.
To explain this relation in {a} better way, we need to introduce another family of distributions  $ \mathfrak{C}^k_{\sigma,l}:C^k(M)\to\C$ for $0\leq l< 2m_\sigma$,
\begin{equation}\label{def:gothc}
\mathfrak{C}^k_{\sigma,l}(f):=\!\sum_{\substack{0\leq i\leq k\\i\neq m_\sigma-1\operatorname{mod} m_\sigma\\i\neq k-(m_\sigma-1)\operatorname{mod} m_\sigma}}\!
\theta_\sigma^{l(2i-k)}\binom{k}{i}\mathfrak{B}(\tfrac{(m_\sigma-1)-i}{m_\sigma},\tfrac{(m_\sigma-1)-k+i}{m_\sigma})\frac{\partial^{k}(f\cdot V)}{\partial z^i\partial\overline{z}^{k-i}}(0,0),
\end{equation}
where $\theta_\sigma$ is the principal $2m_\sigma$-th root of unity and the (beta-like) function $\mathfrak{B}(x,y)$ is defined for any pair $x,y$ of real numbers such that  $x,y\notin \Z$  as follows
\[\mathfrak{B}(x,y)=\frac{\pi e^{\iota\frac{\pi}{2}(y-x)}}{2^{x+y-2}}\frac{\Gamma(x+y-1)}{\Gamma(x)\Gamma(y)},\]
where we adopt the convention $\Gamma(0)=1$ and $\Gamma(-n)=1/(-1)^n n!$.

{For any saddle $\sigma\in\mathrm{Sd}(\psi_\R)$ its neighbourhood splits into $2m_\sigma$ angular sectors (see \eqref{def:angsec}) bounded by separatrices emanated from $\sigma$.
As we will see in Corollary~\ref{cor:Hold1}, the values of $\mathfrak{C}^k_{\sigma,l}(f)$ for $k\geq 0$, roughly speaking, measure {how the appearance of orbits (of $\psi_\R$) in the $l$-th angular sector affects} on the value of ergodic integrals of the function $f$. Unfortunately, the geometric interpretation of the value of the distributions $\mathfrak{d}^k_{\sigma,j}$ is difficult to understand.}

The functionals $ \mathfrak{C}^k_{\sigma,l}$ for $0\leq l< 2m_\sigma$ are not linearly independent, in contrast to the family of functionals $\mathfrak{d}^k_{\sigma,j}$. Indeed,
\begin{gather}
\label{eq:symdist}
\mathfrak{C}^k_{\sigma,l+m_\sigma}=(-1)^k\mathfrak{C}^k_{\sigma,l}\text{ for }0\leq l< m_\sigma\text{ and}\\
\sum_{0\leq l<2m_\sigma}\theta_\sigma^{(k-2j)l}\mathfrak{C}^k_{\sigma,l}=0\text{ if }j=m_\sigma-1\text{ or }j=k-(m_\sigma-1). \nonumber
\end{gather}
 The element of $\R_\eta$ given by
\[\mathfrak{e}(\mathfrak{d}^k_{\sigma,j})=\mathfrak{e}(\mathfrak{C}^k_{\sigma,l})=\mathfrak{e}(\sigma,k)=\left\{\begin{array}{cl}\frac{k-(m_\sigma-2)}{m_\sigma}&\text{ if }\frac{k-(m_\sigma-2)}{m_\sigma}\notin \Z,\\
\frac{k-2(m_\sigma-1)}{m_\sigma}+\eta &\text{ if }\frac{k-(m_\sigma-2)}{m_\sigma}\in \Z.\end{array}\right.\]
is called the \emph{exponent} of  $\mathfrak{d}^k_{\sigma,j}$ or $\mathfrak{C}^k_{\sigma,l}$.
Then
\[\mathfrak{o}(\mathfrak{d}^k_{\sigma,j})=\mathfrak{o}(\mathfrak{C}^k_{\sigma,l})=\mathfrak{o}(\sigma,k):=v(\mathfrak{e}(\sigma,k))=\frac{k-(m_\sigma-2)}{m_\sigma}\]
is called the \emph{order} of  $\mathfrak{d}^k_{\sigma,j}$ or $\mathfrak{C}^k_{\sigma,l}$.
Finally, let
\begin{align*}
\widehat{\mathfrak{e}}(\mathfrak{d}^k_{\sigma,j})&=\widehat{\mathfrak{e}}(\sigma,k)=k-(m_\sigma-1)+\eta \text{ and} \\
\widehat{\mathfrak{o}}(\mathfrak{d}^k_{\sigma,j})&=\widehat{\mathfrak{o}}(\sigma,k)=v(\widehat{\mathfrak{e}}(\mathfrak{d}^k_{\sigma,j}))=k-(m_\sigma-2).
\end{align*}
{Referring to Section~\ref{sec:grad}, $\mathfrak{e}$ and $\widehat{\mathfrak{e}}$ should be regarded as the degree of (modified H\"older) regularity in the scale introduced in  Section~\ref{sec:grad}, while  $\mathfrak{o}$ and $\widehat{\mathfrak{o}}$ should be regarded as the numerical values of the degrees of regularity that are easy to compare.}

For any saddle $\sigma\in\mathrm{Sd}(\psi_\R)$ its (singular) neighbourhood $U_\sigma$ splits into $2m_\sigma$ (angular) sectors bounded by separatrices emanated from $\sigma$. In singular coordinates $z=(x,y)$ they are of the form
\begin{equation}\label{def:angsec}
U_{\sigma,l}:=\{z\in U_\sigma: \Arg z\in(\tfrac{\pi l}{m_\sigma},\tfrac{\pi (l+1)}{m_\sigma})\}\text{ for }0\leq l<2m_\sigma.
\end{equation}
Each such sector is either included in a minimal component $M'$ of $\psi_\R$ or is disjoint from $M'$. In the problem of studying the regularity of the solutions of the cohomological equation, only non-zero values of invariant distributions $\mathfrak{C}^k_{\sigma,l}(f)$ such that $U_{\sigma,l}\cap M'\neq \emptyset$ turn out to be relevant.

\subsection{Main results}
The first main theorem describes the smoothness of the function $\varphi_f$ depending on the values of the functionals described in Section~\ref{sec:defdist}.
To precisely describe the regularity of $\varphi_f$, in Section~\ref{sec;polycocycle}, for any $n\in\Z_{\geq 0}$ and $0\leq a<1$ we introduce the space $C^{n+\pa}(\sqcup_{\alpha \in \mathcal{A}}I_\alpha )$ (and its geometric version $C^{n+\pag}$)
of functions whose $n$-th derivative has polynomial singularities of order at most $-a$ at the ends of the intervals translated by the IET $T$.
We should mention that for any $n\in\N$ we have $C^{n+\pa}\subset C^{(n-1)+(1-a)}$ if $0<a<1$ and $C^{n+\mathrm{P_0}}\subset C^{(n-1)+\eta}$.

Recall that we always assume that $M$ is a compact connected orientable $C^\infty$-surface and $\psi_\R$ is a locally Hamiltonian $C^\infty$-flow on $M$ with isolated fixed points and such that all its saddles are perfect and all saddle connections are loops. Let $M'\subset M$ be a minimal component of the flow and let $I\subset M'$ be a transversal curve. The corresponding IET $T:I\to I$ exchanges the intervals $I_\alpha$, $\alpha \in \mathcal{A}$.

For any $r\geq -\frac{m-2}{m}$, where $m$ is the maximal multiplicity of saddles in $\mathrm{Sd}(\psi_\R)\cap M'$,  let
\begin{equation*}
k_r=\left\{
\begin{array}{cl}
\lceil mr+(m-1)\rceil &\text{if }-\frac{m-2}{m}\leq r\leq -\frac{m-3}{m},\\
\lceil mr+(m-2)\rceil &\text{if }-\frac{m-3}{m}< r.
\end{array}
\right.
\end{equation*}
Note that
\begin{equation}\label{eq:expkr}
\max\{k\geq 0:\exists_{\sigma\in \mathrm{Sd}(\psi_\R)\cap M'}\mathfrak{o}(\sigma,k)< r\}+1=\lceil mr+(m-2)\rceil.
\end{equation}

Denote by $\mathscr{TD}$ the set of triples $(\sigma,k,j)\in (\mathrm{Sd}(\psi_\R)\cap M')\times\Z_{\geq 0}\times\Z_{\geq 0}$ such that $0\leq j\leq k\wedge(m_\sigma-2)$ and $j\neq k-(m_\sigma-1)\operatorname{mod} m_\sigma$
and by $\mathscr{TC}$ the set of triples $(\sigma,k,l)\in (\mathrm{Sd}(\psi_\R)\cap M')\times\Z_{\geq 0}\times\Z_{\geq 0}$ such that $0\leq l<2m_\sigma$ and $U_{\sigma,l}\cap M'\neq \emptyset$.

\begin{theorem}\label{thm1}
Fix $r\geq-\frac{m-2}{m}$. Suppose that $f\in C^{{k}_r}(M)$ is such that $\mathfrak{C}^k_{\sigma,l}(f)=0$ for all $(\sigma,k,l)\in\mathscr{TC}$ such that $\mathfrak{o}(\mathfrak{C}^k_{\sigma,l})<r$.
Then $\varphi_f\in C^{n+\pag}(\sqcup_{\alpha \in \mathcal{A}}I_\alpha)$ with $n=\lceil r\rceil$ and $a=\lceil r\rceil-r$.
Moreover,
the operator
\[C^{{k}_r}(M)\cap\bigcap_{\substack{(\sigma,k,l)\in\mathscr{TC}\\ \mathfrak{o}(\mathfrak{C}^k_{\sigma,l})<r}}\ker(\mathfrak{C}^k_{\sigma,l})\ni f\mapsto \varphi_f\in C^{n+\pag}(\sqcup_{\alpha \in \mathcal{A}}I_\alpha)\]
is bounded.
\end{theorem}

This result provides a descending filtration of the space $\Phi^k:=\{\varphi_f:f\in C^k(M)\}$, $k\in\N\cup\{\infty\}$ that is the basis for proving a spectral theorem (in \cite{Fr-Ki3}) for the so-called Kontsevich-Zorich cocycle on $\Phi^k$. Using renormalization techniques, the aforementioned spectral result allows understanding the regularity of the solution of the cohomological equation  \eqref{eq:coheqT} (see also \cite{Fr-Ki3}) for a.e.\ IET $T$.

The second main theorem solves the problem regarding {the} regularity of the solutions of \eqref{eq:coheqfl}, provided we know the degree of smoothness for  the solution of \eqref{eq:coheqT}. This result is another ingredient in the proof of the final theorem on the regularity of the solution of the cohomological equation \eqref{eq:coheqfl} presented in \cite{Fr-Ki3}.

\begin{theorem}\label{thm2}
Fix $r\in\R_\eta$ so that $v(r)>0$. Assume that $f\in C^{{k}_{v(r)}}(M)$ is such that
\begin{itemize}
 \item $\mathfrak{d}^k_{\sigma,j}(f)=0$ for all $(\sigma,k,j)\in\mathscr{TD}$ with $\widehat{\mathfrak{o}}(\mathfrak{d}^k_{\sigma,j})<v(r)$;
 \item $\mathfrak{C}^k_{\sigma,l}(f)=0$ for all $(\sigma,k,l)\in\mathscr{TC}$ with ${\mathfrak{o}}(\mathfrak{C}^k_{\sigma,l})<v(r)$.
\end{itemize}
Suppose that $g\in C^{r}(I)$ is a solution of the cohomological equation $\varphi_f=g\circ T-g$. Then there exists $u_{g,f}\in C^r(M'_e)$ satisfying  $Xu_{g,f}=f$ on $M'_e$. Moreover, there exists a constant $C_r>0$ such that
\[\|u_{g,f}\|_{C^r(M'_e)}\leq C_r(\|g\|_{C^{r}(I)}+\|f\|_{C^{{k}_{v(r)}}(M)}).\]
\end{theorem}
{The following theorem states that the regularity of the solution obtained in Theorem \ref{thm2} is optimal.}

\begin{theorem}[optimal regularity]\label{thm3}
Let $r\in\R_\eta$ with $v(r)>0$ and let $f\in C^{{k}_{v(r)}}(M)$. If there exists  $u\in C^{r}(M'_e)$ such that $Xu=f$ on $M'_e$ then
\begin{itemize}
 \item $\mathfrak{d}^k_{\sigma,j}(f)=0$ for all $(\sigma,k,j)\in\mathscr{TD}$ with $\widehat{\mathfrak{o}}(\mathfrak{d}^k_{\sigma,j})<v(r)$;
 \item $\mathfrak{C}^k_{\sigma,l}(f)=0$ for all $(\sigma,k,l)\in\mathscr{TC}$ with ${\mathfrak{o}}(\mathfrak{C}^k_{\sigma,l})<v(r)$.
\end{itemize}
\end{theorem}

In summary, all three main results provide an analytical background necessary to fully solve the regularity problem of solving the cohomological equation for locally Hamiltonian flows. The dynamical component, using mainly renormalization techniques, the authors left to \cite{Fr-Ki3}.

{\subsubsection*{Discussion on the minimal case}
Assume that the locally Hamiltonian flow $\psi_\R$ has no saddle connection, so it is minimal on the whole surface $M$. Then  $\mathscr{TC}$ is the set of all triples $(\sigma,k,l)\in \mathrm{Sd}(\psi_\R)\times\Z_{\geq 0}\times\Z_{\geq 0}$ such that $0\leq l<2m_\sigma$, {that is,} there are no restrictions on $l$ coming from the excluded sectors.} It follows that for any $k\geq 0$ the functionals $\mathfrak{C}^k_{\sigma,l}$ and $\mathfrak{d}^k_{\sigma,j}$ generate the same space of invariant distributions. In general, the former space is a subspace of the latter.  Since ${\mathfrak{o}}(\sigma,k)<\widehat{\mathfrak{o}}(\sigma,k)$, the conditions involving the functionals  $\mathfrak{d}^k_{\sigma,j}$ ({in Theorems~\ref{thm2}~and~\ref{thm3}) can be removed.
Then only sectorial distributions $\mathfrak{C}^k_{\sigma,l}$ matter and our main results (Theorems~\ref{thm2}~and~\ref{thm3}) have the following simplified form.}
\begin{corollary}\label{main:cor}
{Suppose $\psi_\mathbb{R}$ is a locally Hamiltonian flow on $M$ that has no saddle connections.}
Fix $r\in\R_\eta$ so that $v(r)>0$. Assume that $f\in C^{{k}_{v(r)}}(M)$  and $g\in C^{r}(I)$ is a solution of the cohomological equation $\varphi_f=g\circ T-g$. Then the existence of $u_{g,f}\in C^r(M)$ satisfying  $Xu_{g,f}=f$ is equivalent to
 \[\mathfrak{C}^k_{\sigma,l}(f)=0\text{ for all }\sigma\in {\rm{Sd}}(\psi_\R),\ 0\leq l<m_\sigma\text{ and }k<m_\sigma v(r)+(m_\sigma-2).\]
\end{corollary}
{Note that the above simplified version of the main results remains true also in the presence of some saddle loops. But not in every case, only when the excluded sectors do not affect the size of the space generated by the sectorial distributions. This is due to the absence of linear independence of the distribution $\mathfrak{C}^k_{\sigma,l}$. For example, if for each pair $U_{l,\sigma}$, $U_{l+m_\sigma,\sigma}$, $0\leq l<m_\sigma$ of sectors at most one is excluded then, by \eqref{eq:symdist},  for any $k\geq 0$ the functionals $\mathfrak{C}^k_{\sigma,l}$ and $\mathfrak{d}^k_{\sigma,j}$ generate the same space of invariant distributions and Corollary~\ref{main:cor} remains still true.}

Let us mention that local $C^\infty$-solutions of cohomological equations for  flows without saddle loops around saddles were studied by Roussarie in \cite{Ro}. We should emphasize that our results are new (even for flows without saddle loops) because they involve solutions with finite differentiability, which causes significant technical complications.
In this case, Forni has suggested us an alternative strategy that potentially simplifies the complex techniques used in this article.

However, the main advantage and novelty of {the} local tools introduced in this article is the ability to study solutions in closed angular sectors (so-called semi-solutions), which makes it possible to apply {our tools to flows} that have saddle loops. These types of problems has not been systematically studied before.
Under an assumption that some saddles have (many) loops, for every $k$ large enough the functionals $\mathfrak{C}^k_{\sigma,l}$ generate {smaller} space than that generated by $\mathfrak{d}^k_{\sigma,j}$. Then some functionals $\mathfrak{d}^k_{\sigma,j}$ begin to have an independent effect on the regularity of solutions, but their influence has less intensity than the functionals $\mathfrak{C}^k_{\sigma,l}$, even though both types of functionals  (for fixed $k$) have the same order of regularity. This seems to be a completely new phenomenon, not previously observed in the study of the regularity of solutions to cohomological equations in parabolic dynamics.

\subsection{Structure of the paper}
The paper is organized as follows. In Section \ref{sec;polycocycle}, we define one-parameter family of Banach spaces of functions whose (higher order) derivatives have polynomial singularities at the ends of intervals exchanged by an IET. We establish their basic properties necessary in next sections of the article.
In Section \ref{sec;localanal}, for any continuous function $f$ defined around a saddle, we define three types of functions: $\varphi_{f,l}$,  $\mathscr{F}_{f,l}$ and $F_f$.
The map $\varphi_{f,l}$ is a local version of the function $\varphi_f$ defined in Section~\ref{sec:twoop} and is necessary to study the local behavior of $\varphi_f$ near the ends of intervals exchanged by an IET.
The map  $F_{f}$ is (in a sense) a local solution to the cohomological equation $Xu=f$ in open angular sectors $U_{\sigma,l}$ around the saddle. The map $\mathscr{F}_{f,l}$ is a covering of $F_{f}$ and is a technical tool for showing basic properties of the other two. In Section~\ref{sec;localanal}, we prove basic properties of $\mathscr{F}_{f,l}$, which are used to understand the behavior of $F_f$ on open angular sectors $U_{\sigma,l}$.
In Section \ref{sec;laphi}, using the tools introduced in Section~\ref{sec;localanal}, we determine precisely the form of $\varphi_{f,l}$ and $\mathscr{F}_{f,l}$ on some angular sectors.
Both of these results are then used to prove that $F_f$ has a smooth extension to closed angular sectors $\overline{U_{\sigma,l}}$ and to establish necessary and sufficient conditions (expressed in the language of local invariant distributions) for such an extension.
Finally, in Section \ref{sec;GP}, we use the contents of all previous sections to prove Theorem \ref{thm1}, \ref{thm2} and \ref{thm3}.

\section{Functions whose (higher order) derivatives have polynomial singularities}\label{sec;polycocycle}
In this section we introduce one-parameter family of Banach spaces of functions whose (higher order) derivatives have polynomial singularities at the ends of intervals exchanged by an IET. 
The  new spaces simply generalize Banach spaces $\pa$ studied in \cite{Fr-Ki}.

\subsection{Space $C^{n+\pa}$}\label{sec;npa}
Fix $0\leq a<1$ and an IET $T:I\to I$ satisfying so called Keane's condition. Denote by  $I_\alpha=[l_\alpha,r_\alpha)$, $\alpha\in \mathcal A$ all subintervals exchanged by $T$. The IET is determined by a pair
$(\pi,\lambda)$, where $\lambda=(\lambda_\alpha)_{\alpha\in\mathcal{A}}\in
\R_{>0}^{\mathcal{A}}$ is the vector of lengths of exchanged intervals, i.e.\ $\lambda_\alpha=r_\alpha-l_\alpha$, and $\pi=(\pi_0,\pi_1)$ is the pair of
bijections $\pi_\vep:\mathcal{A}\to\{1,\ldots,d\}$ for $\vep=0,1$ ($d=|\mathcal A|$ is the number of exchanged intervals) such that $\pi_0(\alpha)$ is the item of $I_\alpha$ before the translation and $\pi_1(\alpha)$ after the translation.

For every $\alpha\in\mathcal{A}$, denote by $m_\alpha$ the middle point of  $I_\alpha$, i.e.\ $m_\alpha = (l_\alpha + r_\alpha)/2$. For  every $\varphi\in C^{1}(\sqcup_{\alpha \in \mathcal{A}}\Int I_\alpha, \C)$ let us consider
\begin{align*}
p_{a}(\varphi):
=& \max_{\alpha\in\mathcal{A}}\Big\{\sup_{x\in(l_\alpha,m_\alpha]}|D\varphi(x)(x-l_\alpha)^{1+a}|,\sup_{x\in[m_\alpha,r_\alpha)}|D\varphi(x)(r_\alpha-x)^{1+a}|\Big\}.
\end{align*}

\begin{definition}\label{def;pa}
For every integer $n \geq 0$, we denote by $C^{n+\pa}(\sqcup_{\alpha \in \mathcal{A}}I_\alpha )$ the space of functions $\varphi\in C^{n+1}(\sqcup_{\alpha \in \mathcal{A}}\Int I_\alpha, \C)$ such that
$p_a(D^n\varphi)<+\infty$ and for every $\alpha\in\mathcal{A}$ the limits
\begin{align*}
C_{\alpha,n}^{a,+}(\varphi) &=(-1)^{n}C_{\alpha}^+(D^n\varphi):=(-1)^{n+1}\lim_{x\searrow l_\alpha}D^{n+1}\varphi(x)(x-l_\alpha)^{1+a},\\
C_{\alpha,n}^{a,-}(\varphi) &=C_{\alpha}^-(D^n\varphi):=\lim_{x\nearrow r_\alpha}D^{n+1} \varphi(x)(r_\alpha-x)^{1+a}
\end{align*}
exist.
We denote by $C^{n+\pag}(\sqcup_{\alpha \in \mathcal{A}}I_\alpha )\subset C^{n+\pa}(\sqcup_{\alpha \in \mathcal{A}}I_\alpha )$ the subset (the union of subspaces) of functions $\varphi\in C^{n+\pa}(\sqcup_{\alpha \in \mathcal{A}}I_\alpha )$ of \emph{geometric type}, i.e.\ such that
\[C_{\pi_0^{-1}(d),n}^{a,-}(\varphi) \cdot C_{\pi_1^{-1}(d),n}^{a,-}(\varphi)=0\quad\text{and}\quad C_{\pi_0^{-1}(1),n}^{a,+}(\varphi) \cdot C_{\pi_1^{-1}(1),n}^{a,+}(\varphi)=0.\]
\end{definition}
%
%
%

For every $0\leq a<1$ and every integer $n \geq 0$, by Lemma~4.3 in \cite{Fr-Ki}, if $\varphi\in C^{n+\pa}(\sqcup_{\alpha \in \mathcal{A}}I_\alpha)$ then $D^{n}\varphi\in L^1(I)$. Let us consider the norm on $C^{n+\pa}(\sqcup_{\alpha \in \mathcal{A}}I_\alpha)$ given by
\begin{equation}\label{eqn;pan}
\|\varphi\|_{C^{n+\pa}}:= \sum_{k=0}^n\|D^{k}\varphi\|_{L^1(I)}+p_a(D^{n}\varphi).
\end{equation}
Recall that, by Lemma~4.2 in \cite{Fr-Ki}, for $n=0$ the space
$C^{n+\pa}(\sqcup_{\alpha \in \mathcal{A}}I_\alpha)$ equipped with the norm $\|\,\cdot\,\|_{C^{n+\pa}}$ is Banach. This gives Banach's condition also for all $n\geq 1$. Moreover, $C^{n+\pag}(\sqcup_{\alpha \in \mathcal{A}}I_\alpha)$ is  a closed subspace of $C^{n+\pa}(\sqcup_{\alpha \in \mathcal{A}}I_\alpha)$ for any $n\geq0$.


Let $\eta:\R_{\geq 0}\to\R_{\geq 0}$ be given by $\eta(x)=-x\log x$ for $x\in[0,e^{-1}]$ and $\eta(x)=e^{-1}$ for $x\geq e^{-1}$. Denote by $C^{\eta}(\sqcup_{\alpha \in \mathcal{A}}I_\alpha)$
the space of functions $f:I\to \C$ such that
\[|f|_{C^\eta}:=\max_{\alpha\in\mathcal A}\sup\left\{\frac{|f(x)-f(y)|}{\eta(|x-y|)}:x,y\in\Int I_\alpha,x\neq y\right\}<+\infty.\]
Then $C^{\eta}(\sqcup_{\alpha \in \mathcal{A}}I_\alpha)$ equipped with the norm  $\|f\|_{C^\eta}=\|f\|_{L^1}+|f|_{C^\eta}$ is a Banach space.
 For every $0<a<1$  denote by $C^{a}(\sqcup_{\alpha \in \mathcal{A}}I_\alpha)$
the space of piecewise $a$-H\"older continuous functions, i.e.\ such that
\[|f|_{C^a}:=\max_{\alpha\in\mathcal A}\sup\left\{\frac{|f(x)-f(y)|}{|x-y|^a}:x,y\in\Int I_\alpha,x\neq y\right\}<+\infty,\]
equipped with the Banach norm $\|f\|_{C^a}=\|f\|_{L^1}+|f|_{C^a}$.

For every $n\geq 0$ we also deal with the Banach spaces $C^{n+\eta}(\sqcup_{\alpha \in \mathcal{A}}I_\alpha)$, $C^{n+a}(\sqcup_{\alpha \in \mathcal{A}}I_\alpha)$
equipped with the norms
\[\|\varphi\|_{C^{n+\eta}}=\sum_{k=0}^n\|D^{k}\varphi\|_{L^1}+|D^{n}\varphi|_{C^\eta},\quad \|\varphi\|_{C^{n+a}}=\sum_{k=0}^n\|D^{k}\varphi\|_{L^1}+|D^{n}\varphi|_{C^a}, \text{ resp.}\]
For every non-natural real number $r>0$ we will write $C^r$ for $C^{\lfloor r\rfloor+\{r\}}$.
\begin{remark}\label{rmk:Hold}
In view of Lemma~4.5 in \cite{Fr-Ki}, for every $\varphi \in C^{0+\pa}(\sqcup_{\alpha \in \mathcal{A}}I_\alpha)$ and $x\in \Int I_\alpha$,
\begin{align*}
\begin{split}
|\varphi(x)|& \leq \frac{\norm{\varphi}_{L^1}}{|I|} +p_a(\varphi)\Big(\frac{1}{a\min\{x-l_\alpha, r_\alpha-x\}^{a}}+
\frac{2^{a+2}}{a(1-a)|I_\alpha|^{a}}\Big)\text{ if }0<a<1,\\
|\varphi(x)| &\leq \frac{\norm{\varphi}_{L^1}}{|I|} + p_a(\varphi)\Big(\log\frac{|I_\alpha|}{2\min\{x-l_\alpha, r_\alpha-x\}}+2 \Big)\text{ if }a=0.
\end{split}
\end{align*}
It follows that if $\varphi \in C^{n+\pa}$ for some $n\geq 1$, then
\begin{align*}
\varphi \in C^{(n-1)+(1-a)}&\text{ with }\|\varphi\|_{C^{(n-1)+(1-a)}}\leq \frac{2^{2+a}\max_{\alpha\in\mathcal A}|I_\alpha|^{1-2a}}{a(1-a)}\|\varphi\|_{C^{n+\pa}}\text{ if } 0<a<1,\\
\varphi \in C^{(n-1)+\eta}&\text{ with }\|\varphi\|_{C^{(n-1)+\eta}}\leq (|I|^{-1}+3)\|\varphi\|_{C^{n+\pa}}\text{ if } a=0.
\end{align*}
\end{remark}

\begin{remark}\label{rmk:filpa}
For any $0\leq a<1$ and any interval $J\subset I_\alpha$ let
\[p_a(\varphi,J):=\sup\{(\min\{x-l_\alpha,r_\alpha-x\})^{1+a}|\varphi'(x)|:x\in J\}.\]
Moreover, for any $n\geq 0$ let
\[\|\varphi\|_{C^{n+\pa}(J)}:= \sum_{k=0}^n\|D^{k}\varphi\|_{L^1(J)}+p_a(D^{n}\varphi,J).\]
In view of Lemma~4.3 in \cite{Fr-Ki}, if $J=(l_\alpha,l_\alpha+\vep]$ or $J=[r_\alpha-\vep,r_\alpha)$ with $\vep\leq |I_\alpha|/2$, then for every $0\leq b<1$ we have
\begin{equation}\label{neq:L1pa}
p_b(\varphi,J)\leq \|\varphi'\|_{L^1(J)}+\frac{p_a(\varphi',J)}{1-a}.
\end{equation}
Let $n,n'\geq 0$ and $0\leq a,a'<1$ such that $n-a\leq n'-a'$. In view of \eqref{neq:L1pa}, $C^{n'+\mathrm{P_{a'}}}\subset C^{n+\pa}$ and
$\|\varphi\|_{C^{n+\pa}(J)}\leq \frac{1}{1-a'}\|\varphi\|_{C^{n'+\mathrm{P_{a'}}}(J)}$.

If $J\subset[l_\alpha+\vep,r_\alpha-\vep]$ for some $\vep>0$ then for any $n\geq 0$ and $0\leq a<1$,
$\|\varphi\|_{C^{n+\pa}(J)}\leq \|\varphi\|_{C^{n+1}(J)}$.
\end{remark}


\section{Local analysis around saddles}\label{sec;localanal}

{In this section, for any continuous function $f$ defined around a saddle $\sigma$, we introduce three types of functions: $\varphi_{f,l}$,  $\mathscr{F}_{f,l}$ for $0\leq l<m_\sigma$ (related to the behavior of the $f$ function on  angular sectors bounded by outgoing separatrices) and $F_f$.
The map $\varphi_{f,l}$ is a local version of the function $\varphi_f$ defined in Section~\ref{sec:twoop} and it is responsible for the behavior of the ergodic integrals for $f$ along the flow orbits as they pass around the saddle and through the $l$-th angular sector.
The map  $F_{f}$ is an ersatz of a local solution to the cohomological equation $Xu=f$ around the saddle $\sigma$, which is defined only on the union of open angular sectors. The map $\mathscr{F}_{f,l}$ is a covering of $F_{f}$ and is a technical tool for showing basic properties of the other two. In Section~\ref{sec;localanal}, we prove basic properties of $\mathscr{F}_{f,l}$, which are used to understand the behavior of $F_f$ on  $U_{\sigma,l}$.
In Section \ref{sec;laphi}, using the tools introduced in Section~\ref{sec;localanal}, we determine precisely the form of $\varphi_{f,l}$ and $\mathscr{F}_{f,l}$ on some angular sectors.

To analyse these three classes of functions we will use a change of variables around $\sigma$ by the covering map $z\mapsto z^{m_\sigma}$. This is a very useful strategy because in the new coordinates the trajectories of the locally Hamiltonian flow are straight lines. Furthermore, we will use the Taylor expansion of the function $f$ at the point $\sigma$ with respect to the complex variables $z$ and $\bar{z}$. Since the operators $f\mapsto \varphi_{f,l}$, $f\mapsto \mathscr{F}_{f,l}$ and $f\mapsto F_{f}$ are linear, our analysis boils down to the study of functions $f$ that are polynomials  in the variables $z$, $\bar{z}$. Moreover, such polynomials are well-behaved for composition with the map $z\mapsto z^{m_\sigma}$ and its inverse branches.}

\medskip

Let $m\geq 2$ be the multiplicity of a saddle point.
Let $G_0:\C\to\C$ be the principal branch of the $m$-th root $G_0(re^{\iota t})=r^{1/m}e^{\iota t/m}$ if $t\in[0,2\pi$) and let $\theta$ and $\theta_0$ be the principal $m$-th and $2m$-th root of unity respectively. Then $G=G_l:\C\to\C$ given by $G_l=\theta^lG_0$ for $0\leq l< m$ is the $l$-th branch of the $m$-th root. 

Let $f:\mathcal{D}\to\C$ be a bounded Borel map where $\mathcal{D}=\mathcal{D}_m$ is the pre-image of the square $[-1,1]\times[-1,1]$
by the map $\C\ni\omega\mapsto \omega^m\in\C$. We will usually  treat $f$ as a function depending on a pair of complex variables $(\omega,\bar{\omega})$. The purpose of this and the next section is to understand the properties of two types of functions $\varphi_{f,l}:[-1,0)\cup(0,1]\to\C$ and  $\mathscr{F}_{f,l}:[-1,1]^2\setminus ([0,1]\times\{0\})\to\C$ for  $0\leq l<m$ associated with $f$, which are crucial in proving the main results of this article. They are given by
\begin{equation}\label{eqn;fgus}
\varphi_{f,l}(s)=\int_{-1}^1\frac{f(G_l(u,s))}{(u^2+s^2)^{\frac{m-1}{m}}}\,du, \quad
\mathscr{F}_{f,l}(u,s) = \int_{-1}^u\frac{f(G_l(v,s))}{(v^2+s^2)^{\frac{m-1}{m}}}\,dv.
\end{equation}
Then $\varphi_{f,l}(s)=\mathscr{F}_{f,l}(1,s)$ for $s\neq 0$. We will usually  treat $\mathscr{F}_{f,l}$ as a function depending on a pair of complex variables $(z,\bar{z})$, where $z=u+\iota s$.

For any $0\leq \alpha<\beta\leq 1$ let $\mathcal{D}(\alpha,\beta):=\{\omega\in \mathcal{D}\setminus\{0\}:\Arg(\omega)\in(2\pi \alpha, 2\pi \beta)\}$. We denote its closure by  $\overline{\mathcal{D}}(\alpha,\beta)$.
For any $A\subset \C$ denote by $A^{1/m}$ the pre-image of $A$ for the map $\omega\mapsto\omega^m$.
We will also need third type of associated function $F_f:\mathcal{D}\setminus([0,1]\times \{0\})^{1/m}\to\C$. As $\mathcal{D}\setminus([0,1]\times \{0\})^{1/m}=\bigcup_{0\leq l<m}\mathcal{D}(\frac{l}{m},\frac{l+1}{m})$, the map $F_f$ is defined by
\[F_f(\omega, \overline \omega) : = \mathscr{F}_{f,l}(\omega^m, {\overline \omega}^m)\text{ on }\mathcal{D}(\tfrac{l}{m},\tfrac{l+1}{m}).
\]
Note that $F_{f\cdot V}$ is (in a sense) a local solution to the cohomological equation $Xu=f$ in any angular sector $\mathcal{D}(\tfrac{l}{m},\tfrac{l+1}{m})$. Indeed, since $d\omega/dt=m\bar{\omega}^{m-1}/V$, we have
\[Xu=m\Big(\overline{\omega}^{m-1}\frac{\partial u}{\partial \omega}+{\omega}^{m-1}\frac{\partial u}{\partial \overline{\omega}}\Big)/V.\]
By definition,
\begin{align*}
\frac{\partial\mathscr{F}_{f\cdot V,l}(z,\bar z)}{\partial z}+\frac{\partial\mathscr{F}_{f\cdot V,l}(z,\bar z)}{\partial\bar z}&=\frac{\partial\mathscr{F}_{f\cdot V,l}(u,s)}{\partial u}\\
&=\frac{(f\cdot V)(G_l(u,s))}{(u^2+s^2)^{\frac{m-1}{m}}}=\frac{(f\cdot V)(G_l(z, \overline z))}{|z|^{2\frac{m-1}{m}}}.
\end{align*}
Then for any $\omega\in \mathcal{D}(\tfrac{l}{m},\tfrac{l+1}{m})$,
\begin{align*}
XF_{f\cdot V}(\omega, \overline \omega)&=\frac{m(\overline{\omega}^{m-1}\frac{\partial \mathscr{F}_{f\cdot V,l}(\omega^m, {\overline \omega}^m)}{\partial \omega}+{\omega}^{m-1}\frac{\partial \mathscr{F}_{f\cdot V,l}(\omega^m, {\overline \omega}^m)}{\partial \overline{\omega}})}{V(\omega, {\overline \omega})}\\
&=\frac{m^2|\omega|^{2(m-1)}(\frac{\partial \mathscr{F}_{f\cdot V,l}(\omega^m, {\overline \omega}^m)}{\partial z}+\frac{\partial \mathscr{F}_{f\cdot V,l}(\omega^m, {\overline \omega}^m)}{\partial \overline{z}})}{V(\omega, {\overline \omega})}=m^2 f(\omega,\overline \omega).
\end{align*}
The map $F_{f}$ is well defined and smooth on every open angular sector $\mathcal{D}(\tfrac{l}{m},\tfrac{l+1}{m})$.
One of the most important technical challenges of this article is to answer the question of when and how the map $F_{f}$ extends smoothly into the closure $\overline{\mathcal{D}}(\tfrac{l}{m},\tfrac{l+1}{m})$.

Some key properties of the three functions are taken in Theorems~\ref{thm:Cketa},~\ref{cor:Hold}~and~\ref{thm;ext}. Since their proofs are very technical, long and intertwined, we precede them with a long list of auxiliary results, which should be regarded as intermediate steps in the proof of the main theorems.

\subsection{Preliminary calculations}\label{subsec:perm}
For every $r>0$ let $\mathscr{S}(r)$ be the circular sector
\[\mathscr{S}(r)=\{(u,s)\neq (0,0): u\leq r|s|\}.\]
For any $0< s\leq 1$ and any $a\in \R$ let
\[\langle s\rangle^a=\left\{
\begin{array}{ccl}
\frac{s^{a}-1}{-a}+1&\text{if}& a<0,\\
1-\log s&\text{if}& a=0,\\
1&\text{if}& a>0.
\end{array}
\right.\]
\begin{remark}\label{rmk:power}
Note that
\begin{itemize}
\item for any $0< s\leq 1$ and any pair of real number $a\leq b$ we have $\langle s\rangle^a\geq \langle s\rangle^b$;
\item for any $a\geq 1$ we have $s^{-a}/a\leq \langle s\rangle^{-a}\leq s^{-a}$;
\item for any $0< a\leq 1$ we have $s^{-a}\leq \langle s\rangle^{-a}\leq s^{-a}/a$;
\item for any $m\geq 1$ and $a\in \R$ we have $\langle s^m\rangle^a\leq m\langle s\rangle^{am}$.
\end{itemize}
\end{remark}

\begin{lemma}\label{lem:estF}
For every $a\in \R$ and $r>0$ there exist $C_a, C_{a,r}>0$ such that
\begin{equation}\label{eq:estF1}
\int_{-1}^u\frac{1}{(v^2+s^2)^a}\,dv \leq C_a \langle|s|\rangle^{1-2a} \text{ for all }s\in[-1,1]\setminus\{ 0\}
\end{equation}
and
\begin{equation}\label{eq:estF2}
\int_{-1}^u\frac{1}{(v^2+s^2)^a}\,dv \leq C_{a,r} \langle\sqrt{\tfrac{u^2+s^2}{2}}\rangle^{1-2a} \text{ for all }(u,s)\in [-1,1]^2\cap\mathscr{S}(r).
\end{equation}
If $f:[-1,1]^2\to\C$ is continuous at $(0,0)$, $f(0,0)=0$ and $a\geq 1/2$ then
\begin{equation}\label{eq:estF3}
\int_{-1}^u\frac{f(v,s)}{(v^2+s^2)^a}\,dv = o(\langle|s|\rangle^{1-2a}).
\end{equation}
\end{lemma}
\begin{proof}

\textbf{Case 1.} Suppose that $a>1/2$.
If $s\neq 0$ then
\[\int_{-1}^u\frac{1}{(v^2+s^2)^a}\,dv=|s|^{1-2a}\int_{-1/|s|}^{u/|s|}\frac{1}{(t^2+1)^a}\,dt\leq |s|^{1-2a}\int_{-\infty}^{+\infty}\frac{1}{(t^2+1)^a}\,dt,\]
which gives \eqref{eq:estF1}.

If $s=0$ and $u<0$ then
\begin{equation}\label{eq:s=0}
\int_{-1}^{u}\frac{1}{(v^2+s^2)^a}\,dv =\int_{|u|}^{1}v^{-2a}\,dv=\frac{1}{2a-1}(|u|^{1-2a}-1)\leq \frac{1}{2a-1}|u|^{1-2a}.
\end{equation}
Let us consider the function $\nu:(-\infty,+\infty)\to\R_+$ given by $\nu(x):=\int_{-\infty}^x\frac{1}{(t^2+1)^a}\,dt$.
If $s\neq 0$ then $\int_{-1}^u(v^2+s^2)^{-a}\,dv\leq |s|^{1-2a}\nu(u/|s|)$.
As
\[\lim_{x\to-\infty}\frac{\nu'(x)}{\frac{d}{dx}(x^2+1)^{1/2-a}}=\lim_{x\to-\infty}\frac{(x^2+1)^{-a}}{(1-2a)x(x^2+1)^{-a-1/2}}=\frac{1}{2a-1},\]
we have $\nu(x)/(x^2+1)^{1/2-a}\to 1/(2a-1)$ as $x\to-\infty$. Therefore there exists $C_{a,r}>0$ such that
$\nu(x)\leq C_{a,r}(\frac{x^2+1}{2})^{1/2-a}$ for $x\leq r$. It follows that for every $(u,s)\in \mathscr{S}(r)$ with $s\neq 0$,
\[\int_{-1}^u\frac{dv}{(v^2+s^2)^a}\leq |s|^{1-2a}\nu(u/|s|)\leq C_{a,r} |s|^{1-2a}\big(\tfrac{(u/s)^2+1}{2}\big)^{1/2-a}=C_{a,r} \big(\tfrac{u^2+s^2}{2}\big)^{1/2-a},\]
which (together with \eqref{eq:s=0}) gives \eqref{eq:estF2}.
\medskip

\textbf{Case 2.} Suppose that $a=1/2$.
If $s\neq 0$ then
 \[\int_{-1}^u(v^2+s^2)^{-1/2}\,dv=\log\frac{u+\sqrt{u^2+s^2}}{-1+\sqrt{1+s^2}}\leq -2\log\frac{s}{3},\]
which gives \eqref{eq:estF1}.
If $s=0$ and $u<0$ then
\begin{equation}\label{eq:s=01}
\int_{-1}^u(v^2+s^2)^{-1/2}\,dv=-\log|u|.
\end{equation}
Moreover,  for any $(u,s)\in \mathscr{S}(r)$ with $s\neq 0$,
\[\int_{-1}^u(v^2+s^2)^{-1/2}\,dv=\log\frac{1+\sqrt{1+s^2}}{-u+\sqrt{u^2+s^2}}\leq \log\frac{6(r^2+1)}{\sqrt{u^2+s^2}},\]
which (together with \eqref{eq:s=01}) gives \eqref{eq:estF2}.
\medskip

\textbf{Case 3.} Suppose that $a<1/2$.
If $0<a<1/2$ then
 \[\int_{-1}^u(v^2+s^2)^{-a}\,dv\leq 2\int_0^1v^{-2a}\,dv=\frac{2}{1-2a}.\]
If $a\leq 0$ then
\[\int_{-1}^u(v^2+s^2)^{-a}\,dv\leq 2^{1-a},\]
which gives \eqref{eq:estF1} and \eqref{eq:estF2}.
\medskip

\textbf{Last claim.}
Suppose that $f:[-1,1]^2\to\C$ is continuous at $(0,0)$, $f(0,0)=0$ and $a\geq 1/2$. For any $\vep>0$ choose $\delta>0$
such that $|f(v,s)|\leq \vep$ if $|v|,|s|<\delta$. It follows that if $|s|<\delta$ then
\begin{align*}
\Big|\int_{-1}^u\frac{f(v,s)}{(v^2+s^2)^a}\,dv \Big|&\leq
\int_{-\delta}^\delta\frac{|f(v,s)|}{(v^2+s^2)^a}\,dv +2\int_{\delta}^1\frac{\|f\|_{\sup}}{(v^2+s^2)^a}\,dv\\
&\leq
\vep\int_{-1}^1\frac{1}{(v^2+s^2)^a}\,dv +2\int_{\delta}^1\frac{\|f\|_{\sup}}{v^{2a}}\,dv\\
&\leq\vep C_a \langle|s|\rangle^{1-2a}+2\|f\|_{\sup}\langle|\delta|\rangle^{1-2a}.
\end{align*}
This gives \eqref{eq:estF3}.
\end{proof}

\begin{remark}\label{rem;zus} For ${z} = u+\iota s$, the followings hold:
\begin{align}
\frac{\partial}{\partial u}  = \left(\frac{\partial}{\partial {z}} + \frac{\partial}{\partial \overline {z}}\right), \quad \frac{\partial}{\partial s} &= \iota\left(\frac{\partial}{\partial {z}} - \frac{\partial}{\partial \overline {z}}\right), \label{der;us}\\
\frac{\partial}{\partial {z}}  = \frac{1}{2}\left(\frac{\partial}{\partial u} - \iota\frac{\partial}{\partial s}\right), \quad \frac{\partial}{\partial \overline {z}} &= \frac{1}{2}\left(\frac{\partial}{\partial u} + \iota\frac{\partial}{\partial s}\right) \label{der;us2}.
\end{align}
\end{remark}
For any $n_1,n_2,a_1, a_2\in\Z_{\geq 0}$ and any $f\in C^n(\mathcal{D})$ ($n=n_1+n_2$), we will deal with some auxiliary functions $F_{n_1,n_2,a_1,a_2},G_{n_1,n_2,a_1,a_2}:[-1,1]^2\setminus([0,1]\times\{0\})$  given by
\begin{align*}
F_{n_1,n_2,a_1,a_2}({z},\overline {z}) &= F_{n_1,n_2,a_1,a_2}(u,s)  = \frac{\partial^{n}f}{\partial {\omega}^{n_1}\partial \overline{\omega}^{n_2}}(G(u,s))\cdot {G(u,s)}^{-a_1}\overline{G(u,s)}^{-a_2},\\
G_{n_1,n_2,a_1,a_2}({z},\overline {z}) &= G_{n_1,n_2,a_1,a_2}(u,s) =  \int_{-1}^uF_{n_1,n_2,a_1,a_2}(v,s) dv.
\end{align*}
The functions $F_{n_1,n_2,a_1,a_2}$ and $G_{n_1,n_2,a_1,a_2}$ will be called \emph{$F$-type} and \emph{$G$-type} functions.\\

{The main use of $F,G$-type functions is to estimate higher-order derivatives of the functions $\mathscr{F}_{f,l}$ and ${F}_{f}$ in a convenient way.
This will then be used to study the effect that the remainder term in the Taylor expansion of $f$ has on the behavior of the functions $\mathscr{F}_{f,l}$ and ${F}_{f}$ around zero (see Lemma~\ref{eqn;F(z)bound}~and~\ref{lem:estDF}). }
\begin{lemma}\label{lem;fd}
For any $f\in C^{n+1}(\mathcal{D})$ we have
\begin{align}
\frac{\partial F_{n_1,n_2,a_1,a_2}}{\partial {z}} &= \frac{1}{m}F_{n_1+1,n_2,a_1+m-1,a_2} -\frac{a_1}{m}F_{n_1,n_2,a_1+m,a_2},\label{eqn;derF1}\\
\frac{\partial F_{n_1,n_2,a_1,a_2}}{\partial \overline {z}} &= \frac{1}{m}F_{n_1,n_2+1,a_1,a_2+m-1} -\frac{a_2}{m}F_{n_1,n_2,a_1,a_2+m},\label{eqn;derF2}\\
\frac{\partial G_{n_1,n_2,a_1,a_2}}{\partial {z}} & = \frac{1}{2} F_{n_1,n_2,a_1,a_2} + \frac{1}{2m} (G_{n_1+1,n_2,a_1+m-1,a_2} - a_1G_{n_1,n_2,a_1+m,a_2}) \label{eqn;derG1} \\
& \quad - \frac{1}{2m} (G_{n_1,n_2+1,a_1,a_2+m-1} - a_2G_{n_1,n_2,a_1,a_2+m}),\nonumber \\
\frac{\partial G_{n_1,n_2,a_1,a_2}}{\partial \overline  {z}} & =  \frac{1}{2}F_{n_1,n_2,a_1,a_2}  - \frac{1}{2m} (G_{n_1+1,n_2,a_1+m-1,a_2} - a_1G_{n_1,n_2,a_1+m,a_2})\label{eqn;derG2}\\
&  \quad + \frac{1}{2m} (G_{n_1,n_2+1,a_1,a_2+m-1} - a_2G_{n_1,n_2,a_1,a_2+m})\nonumber .
\end{align}
\end{lemma}
\begin{proof}
Since $\frac{\partial G}{\partial z}=\frac{1}{m}G^{1-m}$, $\frac{\partial G}{\partial \overline{z}}=0$,
$\frac{\partial \overline{G}}{\partial \overline{z}}=\frac{1}{m}\overline{G}^{1-m}$ and $\frac{\partial \overline{G}}{\partial z}=0$, we obtain
\begin{align*}
\frac{\partial F_{n_1,n_2,a_1,a_2}}{\partial {z}}  &=  \frac{1}{m}\frac{\partial^{n_1+n_2+1}}{\partial {\omega}^{n_1+1}\partial \overline{\omega}^{n_2}}f(G,\overline {G})\cdot{G}^{-a_1+1-m}\overline {G}^{-a_2} \\
&\quad - \frac{a_1}{m}\frac{\partial^{n_1+n_2}}{\partial {\omega}^{n_1}\partial \overline{\omega}^{n_2}}f(G,\overline {G})\cdot{G}^{-a_1-m}\overline {G}^{-a_2} \\
& = \frac{1}{m}F_{n_1+1,n_2,a_1+m-1,a_2} - \frac{a_1}{m}F_{n_1,n_2,a_1+m,a_2}.
\end{align*}
We also verify  \eqref{eqn;derF2} in the same manner.

To obtain \eqref{eqn;derG1}, in view of  \eqref{der;us} and \eqref{der;us2}, we get
\begin{align*}
\frac{\partial G_{n_1,n_2,a_1,a_2}}{\partial  {z}} &= \frac{1}{2}\left(\frac{\partial}{\partial u} - \iota\frac{\partial}{\partial s}\right)
\int_{-1}^uF_{n_1,n_2,a_1,a_2}(v,s) \, dv\\
& = \frac{1}{2}F_{n_1,n_2,a_1,a_2} - \frac{\iota}{2} \frac{\partial}{\partial s}\int_{-1}^u F_{n_1,n_2,a_1,a_2}(v,s)\,dv\\
& = \frac{1}{2}F_{n_1,n_2,a_1,a_2} + \frac{1}{2} \int_{-1}^u \left(\frac{\partial}{\partial {z}} - \frac{\partial}{\partial \overline {z}}\right) F_{n_1,n_2,a_1,a_2}(v,s)\,dv.
\end{align*}
Therefore, in view of \eqref{eqn;derF1} and \eqref{eqn;derF2}, this gives \eqref{eqn;derG1}. Likewise, we repeat the same for \eqref{eqn;derG2}.
\end{proof}

The quantities
\[d(F_{n_1,n_2,a_1,a_2})=\frac{n_1+n_2+a_1+a_2}{m},\ d(G_{n_1,n_2,a_1,a_2})=\frac{n_1+n_2+a_1+a_2}{m}-1\]
we call the \emph{degrees} of the functions $F_{n_1,n_2,a_1,a_2}$ and $G_{n_1,n_2,a_1,a_2}$.
In view of \eqref{eqn;derF1}-\eqref{eqn;derG2}, we have the following conclusion.
\begin{corollary}\label{cor:FGtype}
For any $(l_1,l_2)\in \Z^2_{\geq 0}$ the partial derivative
$\frac{\partial^{l}G_{n_1,n_2,a_1,a_2}}{\partial {z}^{l_1}\partial \overline{z}^{l_2}}$
is a linear combination of $F$-type and $G$-type functions of degree $d(G_{n_1,n_2,a_1,a_2})+l$, where $l=l_1+l_2$. Moreover, each component of the linear combination is of the form $F_{n'_1,n'_2,a'_1,a'_2}$ and
$G_{n'_1,n'_2,a'_1,a'_2}$ such that $n\leq n'_1+n'_2\leq n+l$.
\end{corollary}

\begin{lemma}\label{lem;FGbound}
Let $n,\in\Z_{\geq 0}$.
Assume that $f\in C^{k\vee n}(\mathcal{D})$ and  $D^jf(0,0)=0$ for $0\leq j<k$. Then
\[|F_{n_1,n_2,a_1,a_2}(z,\overline z)|  \leq \|f\|_{C^{k\vee n}}|z|^{-d(F_{n_1,n_2,a_1,a_2}) + \frac{k\vee n}{m}}.\]
Moreover, there exists $C=C_{a_1,a_2,n,k}>0$ such that
\begin{align}\label{eq:Gna1}
|G_{n_1,n_2,a_1,a_2}(z,\overline z)| \leq
C\|f\|_{C^{k\vee n}}\langle|\Im z|\rangle^{-d(G_{n_1,n_2,a_1,a_2}) + \frac{k\vee n}{m}}
\end{align}
for any $z\in [-1,1]^2\setminus ([-1,1]\times\{0\})$.
For every $r>0$  there exists $C_r=C_{a_1,a_2,n,k,r}>0$ such that
\begin{align}\label{eq:Gna2}
|G_{n_1,n_2,a_1,a_2}(z,\overline z)|\leq
C_r\|f\|_{C^{k\vee n}}\langle|z|/\sqrt{2}\rangle^{-d(G_{n_1,n_2,a_1,a_2}) + \frac{k\vee n}{m}}
\end{align}
on $[-1,1]^2\cap \mathscr{S}(r)$.

If additionally $n\leq k$ and $D^kf(0,0)=0$ then
\begin{align}\label{eq:FGbound}
\begin{split}
|F_{n_1,n_2,a_1,a_2}(z,\overline z)|&=o(|\Im z|^{-d(F_{n_1,n_2,a_1,a_2}) + \frac{k}{m}})\text{ if }\tfrac{k}{m}\leq d(F_{n_1,n_2,a_1,a_2})\text{ and}\\
|G_{n_1,n_2,a_1,a_2}(z,\overline z)|&=o(\langle|\Im z|\rangle^{-d(G_{n_1,n_2,a_1,a_2}) + \frac{k}{m}})\text{ if }\tfrac{k}{m}\leq d(G_{n_1,n_2,a_1,a_2}).
\end{split}
\end{align}
\end{lemma}
\begin{proof}
As $D^jf(0,0)=0$ for $0\leq j<k$,
\[\left|\frac{\partial^n f}{\partial \omega^{n_1}\partial {\bar \omega}^{n_2}}(\omega,\bar{\omega})\right|\leq \|D^{k\vee n}f\|_{C^0}|\omega|^{(k\vee n)-n}.
%
\]
Hence
\[|F_{n_1,n_2,a_1,a_2}(z,\overline z)|  \leq \|f\|_{C^{k\vee n}}|z|^{-\frac{a_1+a_2+n}{m}+\frac{k\vee n}{m}} = \|f\|_{C^{k\vee n}}|z|^{-d(F_{n_1,n_2,a_1,a_2}) +\frac{k\vee n}{m}}.\]
If $d_F=d(F_{n_1,n_2,a_1,a_2})=\frac{n+a_1+a_2}{m}$ then
\[|G_{n_1,n_2,a_1,a_2}({z},\overline {z})| \leq  \int_{-1}^u|F_{n_1,n_2,a_1,a_2}(v,s)| dv=\|f\|_{C^{k\vee n}}\int_{-1}^u(v^2+s^2)^{\frac{-d_F+\frac{k\vee n}{m}}{2}} dv.\]
As $d_G=d(G_{n_1,n_2,a_1,a_2})=d_F-1$, the inequalities \eqref{eq:Gna1} and \eqref{eq:Gna2} follow directly from Lemma~\ref{lem:estF}.

\medskip

Suppose that $0\leq n\leq k$, $f\in C^{k}(\mathcal{D})$ and $D^jf(0,0)=0$ for $0\leq j\leq k$. Then
$\left|\frac{\partial^n f}{\partial \omega^{n_1}\partial {\bar \omega}^{n_2}}(\omega,\bar{\omega})\right|=o(|\omega|^{k-n})$. Hence
\[|F_{n_1,n_2,a_1,a_2}(z,\overline z)| =o(|z|^{-\frac{a_1+a_2+n}{m}+\frac{k}{m}}) = o(|\Im z|^{-d_F +\frac{k}{m}}) \text{ if }\frac{k}{m}\leq d_F.\]
 Moreover,
\[|G_{n_1,n_2,a_1,a_2}({z},\overline {z})| \leq  \int_{-1}^u|F_{n_1,n_2,a_1,a_2}(v,s)| dv=
\int_{-1}^u\frac{\xi(v,s)}{(v^2+s^2)^{\frac{d_G+1-\frac{k}{m}}{2}}} dv,
\]
where $\lim_{(v,s)\to(0,0)}\xi(v,s)=0$. If $d_G\geq \frac{k}{m}$ then  the second line of \eqref{eq:FGbound} follows directly from \eqref{eq:estF3}.
\end{proof}

\subsection{Higher derivatives of functions $\mathscr{F}$ and $F$}
In this section,   using the results proved in Section~\ref{subsec:perm}, we study the behaviour  around zero of the higher order partial derivatives for the functions $\mathscr{F}_{f,l}$ and $F_f$.
For any $a\in\Z_{\geq 0}$ and any bounded Borel map $f:\mathcal{D}\to\C$
let us consider
$\mathscr{F}=\mathscr{F}_l=\mathscr{F}_{f,l}:[-1,1]^2\setminus ([0,1]\times\{0\})\to\C$ given by
\[ \mathscr{F}_{f,l}(u,s) = \int_{-1}^u\frac{f(G_l(v,s))}{(v^2+s^2)^{\frac{a}{m}}}\,dv.\]

\begin{lemma}\label{eqn;F(z)bound}
Assume that $f\in C^{k\vee n}(\mathcal{D})$ and $D^jf(0,0)=0$ for $0\leq j<k$. Then there exists $C=C_{a,n,k}>0$ such that for every $(n_1,n_2)\in\Z^2_{\geq 0}$ with $n_1+n_2=n$ we have
\begin{align}\label{eqn;F(z)bound1}
\left|\frac{\partial^{n} {\mathscr{F}}(z,{\overline z})}{\partial z^{n_1}\partial {\overline z}^{n_2}}\right|\leq
C\|f\|_{C^{k\vee n}} \langle|\Im z|\rangle^{-(\frac{2a}{m}+(n-1)-\frac{k}{m})}\text{ if }\Im z\neq 0.
\end{align}
For every $r>0$ there exists $C_r=C_{a,n,k,r}>0$ such that
\begin{align}\label{eqn;F(z)bound2}
\left|\frac{\partial^{n} {\mathscr{F}}(z,{\overline z})}{\partial z^{n_1}\partial {\overline z}^{n_2}}\right|\leq
C_r\|f\|_{C^{k\vee n}}\langle|z|/\sqrt{2}\rangle^{-(\frac{2a}{m}+(n-1)-\frac{k}{m})}\text{ on }\mathscr{S}(r).
\end{align}
If additionally $0\leq n\leq k$ and $D^kf(0,0)=0$ then
\begin{align}\label{eqn;F(z)bound3}
\left|\frac{\partial^{n} {\mathscr{F}}(z,{\overline z})}{\partial z^{n_1}\partial {\overline z}^{n_2}}\right|=o(\langle|\Im z|\rangle^{-(\frac{2a}{m}+(n-1)-\frac{k}{m})})\text{ if }\frac{2a}{m}+(n-1)\geq \frac{k}{m}.
\end{align}
\end{lemma}
\begin{proof}
By definition, $\mathscr{F}=G_{0,0,a,a}$. In view of Corollary~\ref{cor:FGtype}, the partial derivative
$\frac{\partial^{n}G_{0,0,a,a}}{\partial {z}^{n_1}\partial \overline{z}^{n_2}}$
is a linear combination of $F$-type and $G$-type functions of the form $F_{n'_1,n'_2,a'_1,a'_2}$ and $G_{n'_1,n'_2,a'_1,a'_2}$ such that their degree is $2a/m+n-1$ and $0\leq n':=n'_1+n'_2\leq n$.

Suppose that $n\leq k$. Then \eqref{eqn;F(z)bound1} and  \eqref{eqn;F(z)bound2}  follow directly from Lemma~\ref{lem;FGbound}.
The same arguments combined with \eqref{eq:FGbound} yield \eqref{eqn;F(z)bound3}.

Suppose that $n>k$. Then $k\leq n'\vee k<n$. Therefore, $\|f\|_{C^{n'\vee k}}\leq \|f\|_{C^n}$ and for any $0\leq s\leq 1$ and $d\in \R$ we have $\langle s\rangle^{-d+\frac{n'\vee k}{m}}\leq \langle s\rangle^{-d+\frac{k}{m}}$.
In view of Lemma~\ref{lem;FGbound} this gives \eqref{eqn;F(z)bound1} and  \eqref{eqn;F(z)bound2}.
\end{proof}
By change of coordinates, we obtain the bound of higher derivatives of the map $F=F_f:\mathcal{D}\setminus([0,1]\times \{0\})^{1/m}\to\C$ given by
$F_f(\omega, \overline \omega) = \mathscr{F}_{f,l}(\omega^m, {\overline \omega}^m)$ on $\mathcal{D}(\frac{l}{m},\frac{l+1}{m})$.

\begin{lemma} \label{lem:estDF}
Assume that $f\in C^{k\vee n}(\mathcal{D})$ and $D^jf(0,0)=0$ for $0\leq j<k$. Then
for any $r>0$ there exists $C_{r,n}>0$ such that for every $(n_1,n_2)\in\Z^2_{\geq 0}$ with $n_1+n_2=n$,
\begin{align} \label{neq:estDF}
\left|\frac{\partial^{n} {F}(\omega,{\overline \omega})}{\partial \omega^{n_1}\partial {\overline \omega}^{n_2}}\right|
\leq C_{r,n}\|f\|_{C^{k\vee n}}(1+|\log|\omega||)|\omega|^{(-2a+m+k-n)\wedge 0}
\end{align}
for  $\omega\in \mathcal{D}\cap \mathscr{S}(r)^{1/m}$.
\end{lemma}
\begin{proof}
Recall that ${F}(\omega,{\overline \omega}) = \mathscr{F}(\omega^m, {\overline \omega}^m)$. By Fa\`a di Bruno's formula,
\begin{align*}
\begin{split}
\frac{\partial^{n} F(\omega,\overline \omega)}{\partial \omega^{n_1}\partial \overline \omega^{n_2}} & = \frac{d^{n}\mathscr{F}(\omega^m,\overline \omega^m)}{d\omega^{n_1}d\overline \omega^{n_2}}\\
 &= \sum_{\bar{p},\bar{q}} C_{\bar{p},\bar{q}} \frac{\partial^{|\bar{p}|+|\bar{q}|} \mathscr{F}}{\partial z^{|\bar{p}|}\partial \overline z^{|\bar{q}|} }(\omega^m,\overline \omega^m)\cdot  \prod_{j=1}^{n_1\wedge m}(\omega^{m-j})^{p_j}\prod_{j=1}^{n_2\wedge m}(\overline \omega^{m-j})^{q_j},
\end{split}
\end{align*}
where the sum is over all $n_1$-tuples $\bar{p}=(p_1,\ldots,p_{n_1})$ and $n_2$-tuples $\bar{q}=(q_1,\ldots,q_{n_2})$ of non-negative integers satisfying the constraints
\[
\sum_{j=1}^{n_1}jp_j=n_1, p_j=0 \text{ for } j>n_1\wedge m \text{ and } \sum_{j=1}^{n_2}jq_j=n_2, q_j=0 \text{ for } j>n_1\wedge m,
\]
and we use the notation $|\bar{p}|=\sum_{j=1}^{n_1}p_j$ and $|\bar{q}|=\sum_{j=1}^{n_2}q_j$. Let
\begin{gather*}
P:=\{|\bar{p}|: \sum_{j=1}^{n_1}jp_j=n_1, p_j=0\text{ for }j>n_1\wedge m\},\\
Q:=\{|\bar{q}|: \sum_{j=1}^{n_2}jq_j=n_2, q_j=0\text{ for }j>n_2\wedge m\}.
\end{gather*}
Then
\begin{align*}\frac{\partial^{n} F(\omega,\overline \omega)}{\partial \omega^{n_1}\partial \overline \omega^{n_2}}
&= \sum_{\bar{p},\bar{q}} C_{\bar{p},\bar{q}} \frac{\partial^{|\bar{p}|+|\bar{q}|} \mathscr{F}}{\partial z^{|\bar{p}|}\partial \overline z^{|\bar{q}|} }(\omega^m,\overline \omega^m)\cdot  \omega^{m|\bar{p}|-n_1}{\overline \omega}^{m|\bar{q}|-n_2}\\
&= \sum_{p\in P,q\in Q} C'_{p,q} \frac{\partial^{p+q} \mathscr{F}}{\partial z^{p}\partial \overline z^{q} }(\omega^m,\overline \omega^m)\cdot  \omega^{mp-n_1}{\overline \omega}^{mq-n_2}.
\end{align*}
In view of Lemma~\ref{eqn;F(z)bound} and Remark~\ref{rmk:power}, for every $\omega\in \mathcal{D} \cap \mathscr{S}(r)^{1/m}$,
\begin{align*}
\left|\frac{\partial^{p+q} \mathscr{F}}{\partial z^{p}\partial \overline z^{q} }(\omega^m,\overline \omega^m)\right| \leq
mC_r\|f\|_{C^{k\vee n}}\langle|\omega|/\sqrt[2m]{2}\rangle^{-(2a+(p+q-1)m-k)}.
\end{align*}
Therefore, taking $l=p+q\in P+Q$, for every $\omega\in \mathcal{D} \cap \mathscr{S}(r)^{1/m}$,
\begin{align}
\label{eq:Fomega}
\left|\frac{\partial^{n} {F}(\omega,{\overline \omega})}{\partial \omega^{n_1}\partial {\overline \omega}^{n_2}}\right| \leq \sum_{l\in P+Q}m C_r\|f\|_{C^{k\vee n}}
\langle|\omega/\sqrt[2m]{2}|\rangle^{-(2a+(l-1)m-k)}|\omega|^{ml-n}.
\end{align}
Moreover, for $l=p+q\in P+Q$ we have
\[ml-n=\sum_{j=1}^{n_1\wedge m}(m-j)p_j+\sum_{j=1}^{n_2\wedge m}(m-j)q_j\geq 0.\]
If $2a+(l-1)m-k>0$ then
\[\langle|\omega/\sqrt[2m]{2}|\rangle^{-(2a+(l-1)m-k)}|\omega|^{ml-n}=O(|\omega|^{-(2a+n-m-k)}).\]
If $2a+(l-1)m-k=0$ then
\[\langle|\omega/\sqrt[2m]{2}|\rangle^{-(2a+(l-1)m-k)}|\omega|^{ml-n}=O((1+|\log|\omega||)|\omega|^{-(2a+n-m-k)}).\]
If $2a+(l-1)m-k<0$ then 
\[\langle|\omega/\sqrt[2m]{2}|\rangle^{-(2a+(l-1)m-k)}|\omega|^{ml-n}=|\omega|^{ml-n}=O(1).\]
In view of \eqref{eq:Fomega}, this gives \eqref{neq:estDF}.
%
%
%
\end{proof}

\subsection{Preliminary results necessary to define invariant distributions}
{In this section we introduce new functions $\mathfrak G^l_{a_1,a_2}$ which are useful  for computing $\mathscr{F}_{f,l}$ when $f$ is a polynomial. Their recursive properties are crucial for studying new invariant distributions.}

For any pair of integers $(a_1,a_2)$ {and any $0\leq l<m$}, let $\mathfrak F_{a_1,a_2}=\mathfrak F_{a_1,a_2}^l:[-1,1]^2\setminus([0,1]\times\{0\})\to\C$
and $\mathfrak G_{a_1,a_2}=\mathfrak G_{a_1,a_2}^l:[-1,1]^2\setminus([0,1]\times\{0\})\to\C$ be given by
\begin{align}
\mathfrak F_{a_1,a_2}({z},\overline {z}) &= \mathfrak F_{a_1,a_2}(u,s)  =  {G_l(u,s)}^{-a_1}\overline{G_l(u,s)}^{-a_2},\label{form;F}\\
\mathfrak G_{a_1,a_2}({z},\overline {z}) &= \mathfrak G_{a_1,a_2}(u,s) =  \int_{-1}^u\mathfrak F_{a_1,a_2}(v,s) dv.\label{form;G}
\end{align}
Then $\mathfrak G^l_{a_1,a_2}=\theta^{l(a_2-a_1)}\mathfrak G^0_{a_1,a_2}$ for every $0\leq l<m$.
As
\[G_0(-u-\iota s)=\theta_0 G_0(u+\iota s)\text{ and }G_0(u-\iota s)=\theta^2_0 \overline{G_0(u+\iota s)}\text{ for }s>0,\] it follows that
\begin{equation}\label{eq:Gl}
\mathfrak G^l_{a_1,a_2}(1,s)=\left\{
\begin{array}{rl}
\theta_0^{2l(a_2-a_1)}\mathfrak G^0_{a_1,a_2}(1,|s|)&\text{ if}\quad s\in (0,1],\\
\theta_0^{(2l+1)(a_2-a_1)}\mathfrak G^0_{a_1,a_2}(1,|s|)&\text{ if}\quad s\in [-1,0)
\end{array}
\right.
\end{equation}
and
\begin{equation}\label{eq:Glu}
\mathfrak G^l_{a_1,a_2}(u,s)=\left\{
\begin{array}{rl}
\theta_0^{2l(a_2-a_1)}\mathfrak G^0_{a_1,a_2}(u,|s|)&\text{ if}\quad s\in (0,1],\\
\theta_0^{(2l+2)(a_2-a_1)}\overline{\mathfrak G^0_{a_1,a_2}(u,|s|)}&\text{ if}\quad s\in [-1,0).
\end{array}
\right.
\end{equation}

For any set $A\subset \R^d$ denote by $C^\omega(A)$ the space of complex-valued real-analytic maps on $A$ which have analytic extention to the closure $\bar{A}$. If $f,g:A\to\C$ are such that $f-g\in C^\omega(A)$ then we write $f=g+ C^\omega(A)$. We denote by $C^{\omega,m}_l$ the space of functions $f:[-1,1]^2\setminus([0,1]\times\{0\})\to\C$ such that $\omega\mapsto f(\omega^m,\overline{\omega}^m)$ has a real analytic extension on $\overline{\mathcal{D}}(\frac{l}{m},\frac{l+1}{m})$. For example, $\mathfrak F_{a_1,a_2}\in C^{\omega,m}_l$ if $a_1,a_2$ are non-positive.
\begin{lemma}
For any pair of integers $(a_1,a_2)$,
\begin{equation}\label{eq:anal1}
a_1\mathfrak G^l_{a_1+m,a_2}(1,s)+a_2\mathfrak G^l_{a_1,a_2+m}(1,s)\in C^\omega((0,1])\cap C^\omega([-1,0)).
\end{equation}
If $a_1,a_2$ are additionally both non-positive then
\begin{equation}\label{eq:anal2}
a_1\mathfrak G^l_{a_1+m,a_2}+a_2\mathfrak G^l_{a_1,a_2+m}\in C^{\omega,m}_l.
\end{equation}
\end{lemma}
\begin{proof}
Note that, by \eqref{der;us} and \eqref{form;G},
\begin{align*}
&\mathfrak F_{a_1,a_2}(z,\overline {z})-\mathfrak F_{a_1,a_2}(-1+\iota\Im z,-1-\iota\Im z)=\mathfrak F_{a_1,a_2}(u,s)-\mathfrak F_{a_1,a_2}(-1,s)\\
&=\int_{-1}^u\frac{\partial}{\partial v}\mathfrak F_{a_1,a_2}(v,s)\,dv=
\int_{-1}^u\left(\frac{\partial}{\partial z}+\frac{\partial}{\partial \overline z}\right)\mathfrak F_{a_1,a_2}(v,s)\,dv\\
&=-\int_{-1}^u\left(\frac{a_1}{m}\mathfrak F_{a_1+m,a_2}(v,s)+\frac{a_2}{m}\mathfrak F_{a_1,a_2+m}(v,s)\right)\,dv\\
&=-\frac{1}{m}({a_1}\mathfrak G_{a_1+m,a_2}(z,\overline {z})+{a_2}\mathfrak G_{a_1,a_2+m}(z,\overline {z})).
\end{align*}
{Taking $(z,\overline {z}) = (1,s)$, by \eqref{form;F}, it follows that}
\begin{gather*}
a_1\mathfrak G_{a_1+m,a_2}(1,s)+a_2\mathfrak G_{a_1,a_2+m}(1,s)\\
=m(G_l(-1,s)^{-a_1}\overline{G_l(-1,s)}^{-a_2}-G_l(1,s)^{-a_1}\overline{G_l(1,s)}^{-a_2}).
\end{gather*}
Since the maps $[-1,1]\ni s\mapsto G_l(-1,s)\in\C$, $[0,1]\ni s\mapsto G_l(1,s)\in\C$  and $[-1,0)\ni s\mapsto G_l(1,s)\in\C$ are analytic and the latter has an analytic extension to $[-1,0]$,
this gives \eqref{eq:anal1}. Moreover, for any $\omega\in \mathcal{D}(\frac{l}{m},\frac{l+1}{m})$,
\begin{gather*}
a_1\mathfrak G_{a_1+m,a_2}({\omega}^m,\overline {\omega}^m)+a_2\mathfrak G_{a_1,a_2+m}({\omega}^m,\overline {\omega}^m)\\
=m(G_l(-1+\iota\Im \omega^m)^{-a_1}\overline{G_l(-1+\iota\Im\omega^m)}^{-a_2}-G_l(\omega^m)^{-a_1}\overline{G_l(\omega^m)}^{-a_2})\\
=m(G_l(-1+\iota\Im\omega^m)^{-a_1}\overline{G_l(-1+\iota\Im\omega^m)}^{-a_2}-{\omega}^{-a_1}\overline {\omega}^{-a_2}).
\end{gather*}
Since $-a_1,-a_2$ are non-negative integers, all functions on the RHS are analytic which completes the proof.
\end{proof}

As a conclusion we obtain that for any integer $k\neq m$,
\begin{equation*}\label{eq:anal10}
\mathfrak G_{0,k}(1,s), \mathfrak G_{k,0}(1,s)\in C^\omega((0,1])\cap C^\omega([-1,0))
\end{equation*}
and  for any integer $k<m$,
\begin{equation}\label{eq:anm}
\mathfrak G_{0,k}, \mathfrak G_{k,0}\in C^{\omega,m}_l.
\end{equation}
Moreover,
\begin{align*}
\mathfrak G_{0,m}(1,s)&=\int_{-1}^{1}\frac{1}{v-\iota s}dv=\int_{-1}^{1}\frac{v+\iota s}{v^2+s^2}dv=\iota\sgn(s)\int_{-1/|s|}^{1/|s|}\frac{1}{x^2+1}dx\\
&=\iota\sgn(s)(\arctan(1/|s|)-\arctan(-1/|s|))\\&
=\iota(\operatorname{arccot}(s)-\operatorname{arccot}(-s))+\iota\sgn(s)\pi.
\end{align*}
Hence, $\mathfrak G_{0,m}(1,s), \mathfrak G_{m,0}(1,s)\in C^\omega((0,1])\cap C^\omega([-1,0))$. Using  \eqref{eq:anal1} again, we have
\begin{equation}\label{eq:anal11}
\mathfrak G_{k,-lm}(1,s), \mathfrak G_{-lm,k}(1,s)\in C^\omega((0,1])\cap C^\omega([-1,0))\text{ for all }k\in\Z,l\in \Z_{\geq 0}.
\end{equation}

\subsection{Invariant distributions $\partial^k_j$ and their effect on the regularity of $\mathscr{F}$ and $F$}
%
%
For every $m\geq 2$, $0\leq l<m$, $k\geq 0$ and $f\in C^k(\mathcal{D})$ we deal with three associated functions  $\mathscr{F}_{f,l}$, $\varphi_{f,l}$ and $F_f$. Recall that
$\mathscr{F}_{f,l}:[-1,1]^2\setminus ([0,1]\times\{0\})\to\C$ is given by
\[\mathscr{F}_{f,l}(z,\overline{z})=\mathscr{F}_{f,l}(u,s)=\int_{-1}^u\frac{f(G_l(v,s))}{(v^2+s^2)^{\frac{m-1}{m}}}dv,\]
 $\varphi_{f,l}:[-1,0)\cup(0,1]\to\C$ is given by $\varphi_{f,l}(s)=\mathscr{F}_{f,l}(1,s)$ and $F_f:\mathcal{D}\setminus([0,1]\times\{0\})^{1/m}\to\C$ is given by
$F_f(\omega, \overline \omega) = \mathscr{F}_{f,l}(\omega^m, {\overline \omega}^m)$ on $\mathcal{D}(\frac{l}{m},\frac{l+1}{m})$.

\medskip

For every $k\geq 0$ and let us consider functionals $\partial^k_{j}:C^k(\mathcal{D})\to\C$ for $0\leq j\leq k\wedge (m-1)$ given by
\begin{equation}\label{def:parl}
\partial^k_{j}(f)=\sum_{0\leq n\leq \frac{k-j}{m}}
\frac{\binom{k}{j+nm}\binom{\frac{(m-1)-j}{m}-1}{n}}{\binom{\frac{(k-j)-(m-1)}{m}}{n}}\frac{\partial^kf}{\partial \omega^{j+nm} \partial \overline{\omega}^{k-j-nm}}(0,0).
\end{equation}
Comparing with \eqref{def:gothd}, functionals $\partial^k_{j}$ will play a key role in understanding the meaning of distribution $\mathfrak{d}^k_{\sigma,j}$.
If $0\leq k\leq m-2$ then $\partial^k_j(f)=\binom{k}{j}\frac{\partial^kf}{\partial \omega^{j} \partial \overline{\omega}^{k-j}}(0,0)$.
If $k\geq m-1$ then as we will see in the following lemma, only $m-2$ functionals matter. More precisely,
$\partial^k_{j}$ is irrelevant
if $j=m-1$ or  $j=k-(m-1)\operatorname{mod} m$, {this is, its value does not affect the smoothness of $\varphi_{f,l}$ and $F_f$,  see also Remark~\ref{rem:irrel}}. Note that if $k=m-2\operatorname{mod} m$ then, in this exceptional case, we have $m-1$ relevant functionals.

\medskip

Recall that for any $0\leq \alpha<\beta\leq 1$ let $\mathcal{D}(\alpha,\beta):=\{\omega\in \mathcal{D}\setminus\{0\}:\Arg(\omega)\in(2\pi \alpha, 2\pi \beta)\}$. We denote its closure by  $\overline{\mathcal{D}}(\alpha,\beta)$.

\begin{lemma}\label{lem:analf}
Suppose that $f\in \C[\omega,\overline \omega]$ is a polynomial of degree at most $k$ such that
$\partial^j_{i}(f)=0$ for all $0\leq j\leq k$ and $0\leq i\leq j\wedge(m-2)$ with $i\neq j-(m-1)\operatorname{mod} m$. Then
\[F_f\in C^\omega(\mathcal{D}(\tfrac{l}{m},\tfrac{l+1}{m}))\ \text{ and }\ \varphi_{f,l}\in C^\omega([-1,0))\cap C^{\omega}((0,1])\text{ for }0\leq l<m.\]
Moreover, for any $n\geq 0$ there exists a constant $C^n_k>0$ such that
\begin{align}
\label{neq:FCn}
\|F_f\|_{C^n(\mathcal{D}(\frac{l}{m},\frac{l+1}{m}))}\leq C^n_k\|f\|_{C^{k}(\mathcal{D})}\text{ and }
\|\varphi_{f,l}\|_{C^n([-1,0)\cup(0,1])}\leq C^n_k\|f\|_{C^{k}(\mathcal{D})}.
\end{align}
\end{lemma}
\begin{proof}
First note that \eqref{neq:FCn} follows directly from the first part of the lemma. Indeed, $f\mapsto F_{f}\in C^\omega(\mathcal{D}(\frac{l}{m},\frac{l+1}{m}))$ and $f\mapsto \varphi_{f,l}\in C^\omega([-1,0))\cap C^{\omega}((0,1])$
are linear operators on a finite-dimensional space, so they are bounded.  This gives \eqref{neq:FCn}.

By assumption, $f=\sum_{0\leq j\leq k}f_j$ with
\[f_j(\omega,\overline \omega)=\frac{1}{j!}\sum_{0\leq i\leq j}\binom{j}{i}\frac{\partial^jf}{\partial \omega^{i} \partial \overline{\omega}^{j-i}}(0,0)\omega^i\overline{\omega}^{j-i}.\]
We will show that if $\partial^j_{i}(f)=0$ for all $0\leq i\leq j\wedge(m-2)$ with $i\neq j-(m-1)\operatorname{mod} m$, then
\begin{equation*}\label{eq:analfj}
F_{f_j}\in C^\omega(\mathcal D(\tfrac{l}{m},\tfrac{l+1}{m}))\ \text{ and }\ \varphi_{f_j,l}\in C^\omega([-1,0))\cap C^{\omega}((0,1])\text{ for }0\leq l<m.
\end{equation*}
This gives our claim.

Note that
\begin{align*}\label{eq:Ff}
\begin{split}
\mathscr{F}_{f_j,l}&=\frac{1}{j!}\sum_{0\leq i\leq j}\binom{j}{i}\frac{\partial^jf}{\partial \omega^{i} \partial \overline{\omega}^{j-i}}(0,0)\mathfrak G_{(m-1)-i,(m-1)-(j-i)}
=\frac{1}{j!}\sum_{0\leq i< m}\xi_i,
\end{split}
\end{align*}
where
\[\xi_i(z,\overline z)=\!\!\sum_{0\leq n\leq \frac{j-i}{m}}\!\!\binom{j}{i+nm}\frac{\partial^jf}{\partial \omega^{i+nm} \partial \overline{\omega}^{(j-i)-nm}}(0,0)\mathfrak G_{(m-1)-i-nm,(m-1)-(j-i)+nm}(z,\overline{z}).\]
For every $1\leq n\leq \frac{j-i}{m}$ we have
\[a_1=(m-1)-i-nm< 0 \text{ and } a_2=(m-1)-(j-i)+(n-1)m< 0.\]
 In view of \eqref{eq:anal2},
\begin{align*}
&((m-1)-i-nm)\mathfrak G_{(m-1)-i-(n-1)m,(m-1)-(j-i)+(n-1)m}\\
&\quad+((m-1)-(j-i)+(n-1)m)\mathfrak G_{(m-1)-i-nm,(m-1)-(j-i)+nm}\in C^{\omega,m}_l,
\end{align*}
so
\begin{align*}
&\mathfrak G_{(m-1)-i-nm,(m-1)-(j-i)+nm}\\
&\quad=\frac{(\frac{(m-1)-i}{m}-n)}{(\frac{(j-i)-(m-1)}{m}-(n-1))}\mathfrak G_{(m-1)-i-(n-1)m,(m-1)-(j-i)+(n-1)m}+ C^{\omega,m}_l,
\end{align*}
It follows that for every $0\leq n\leq \frac{j-i}{m}$,
\begin{equation}\label{eq:Gmulti}
\mathfrak G_{(m-1)-i-nm,(m-1)-(j-i)+nm}=\frac{\binom{\frac{(m-1)-i}{m}-1}{n}}{\binom{\frac{(j-i)-(m-1)}{m}}{n}}\mathfrak G_{(m-1)-i,(m-1)-(j-i)}
+C^{\omega,m}_l.
\end{equation}
It follows that for every $0\leq i\leq m-1$,
\begin{equation*}
\xi_i=\partial^j_{i}(f)\mathfrak G_{(m-1)-i,(m-1)-(j-i)}+C^{\omega,m}_l.
\end{equation*}
If $i=m-1$ then by \eqref{eq:anm}, $\mathfrak G_{(m-1)-i,(m-1)-(j-i)}\in C^{\omega,m}_l$ so $\xi_i\in C^{\omega,m}_l$.
If $i=j-(m-1)\operatorname{mod} m$, then $(m-1)-(j-i)+nm=0$ with $n=\lfloor\frac{j-i}{m}\rfloor=\frac{j-i-(m-1)}{m}$.
Again by \eqref{eq:anm}, $\mathfrak G_{(m-1)-i-nm,(m-1)-(j-i)+nm}\in C^{\omega,m}_l$. In view of \eqref{eq:Gmulti}, it follows that
$\mathfrak G_{(m-1)-i,(m-1)-(j-i)}\in C^{\omega,m}_l$ and again $\xi_i\in C^{\omega,m}_l$. Hence
\begin{equation}\label{eq:anxi1}
\mathscr{F}_{f_j,l}=\frac{1}{j!}\sum_{\substack{0\leq i\leq j\wedge(m-2)\\i\neq j-(m-1)\operatorname{mod} m}}\partial^j_{i}(f)\mathfrak G_{(m-1)-i,(m-1)-(j-i)}+C^{\omega,m}_l.
\end{equation}
As $\partial^j_i(f_j)=0$ for $i\neq m-1$ and $i\neq j-(m-1) \ \operatorname{mod} m$, this yields $\mathscr{F}_{f_j,l}\in C^{\omega,m}_l$
 and
$F_{f_j}\in C^\omega(\mathcal{D}(\frac{l}{m},\frac{l+1}{m}))$.

Using  \eqref{eq:anal1} instead of \eqref{eq:anal2}, the same arguments show $\varphi_{f_j,l}{(s)}=\mathscr{F}_{f_j,l}(1,s)\in C^\omega([-1,0))\cap C^{\omega}((0,1])$.
\end{proof}

{
\begin{remark}\label{rem:irrel}
Let $k\geq m-1$ and let us consider $P(\omega,\overline{\omega})=\omega^{m-1}\overline{\omega}^{k-(m-1)}$ and $Q(\omega,\overline{\omega})=\omega^{k-(m-1)}\overline{\omega}^{m-1}$. Then
\[\mathscr{F}_{P,l}=\mathfrak{G}^l_{0,2(m-1)-k}\text{ and }\mathscr{F}_{Q,l}=\mathfrak{G}^l_{2(m-1)-k,0}.\]
As $2(m-1)-k<m$, in view of \eqref{eq:anm}, $\mathscr{F}_{P,l}, \mathscr{F}_{Q,l}\in C^{\omega,m}_l$, so $F_P$ and $F_Q$ are analytic on each closed angular sector. On the other hand, {by \eqref{def:parl}}
\[\partial^k_{m-1}(P)=k!\text{ and } \partial^k_{j}(Q)=k!\binom{\frac{(m-1)-j}{m}-1}{\frac{(k-j)-(m-1)}{m}},\] where $0\leq j<m$ is such that $j=k-(m-1)\operatorname{mod} m$. It follows that the smoothness of $F_f$ is not related to the vanishing of distributions $\partial^k_{m-1}(f)$ and $\partial^k_{j}(f)$.
\end{remark}
}

We finish this section by showing a smooth extension of $F_f$ on angular sectors.
\begin{theorem}\label{thm:Cketa}
Let $k\geq m-1$.
Suppose that $f\in C^{k}(\mathcal D)$ and  $\partial^j_{i}(f)=0$ for all $0\leq j< k$ and $0\leq i\leq j\wedge(m-2)$ with $i\neq j-(m-1)\operatorname{mod} m$. Then for every $r>0$ the map $F_f$ on every angular sector of $\mathcal D\cap \mathscr{S}(r)^{1/m}$ has a $C^{\widehat{\mathfrak{e}}(\sigma,k)}$-extension at $(0,0)$ (recall that $\widehat{\mathfrak{e}}(\sigma,k)=k-(m-1)+\eta$). Moreover, there exists $C_r>0$ such that
$\|F_f\|_{C^{\widehat{\mathfrak{e}}(\sigma,k)}(\mathcal D\cap \mathscr{S}(r)^{1/m})}\leq C_r\|f\|_{C^{k}}$.
\end{theorem}
\begin{proof}
Let $\mathscr{A}$ be an angular sector of $\mathcal D\cap \mathscr{S}(r)^{1/m}$. Let us decompose $f=f_{<k}+e_f$ with
\[f_{<k}(\omega,\overline \omega)=\sum_{0\leq j< k}\frac{1}{j!}\sum_{0\leq i\leq j}\binom{j}{i}\frac{\partial^j f}{\partial\omega^i\partial\overline{\omega}^{j-i}}(0,0)\omega^i\overline{\omega}^{j-i}.\]
Note that the operator $C^{k}(\mathcal D)\ni f\mapsto f_{<k}\in C^{k}(\mathcal D)$ is bounded.
By Lemma~\ref{lem:analf}, $F_{f_{<k}}$ is analytic and for every $n\geq 0$ there exists $C^n_k>0$ such that
$\|F_{f_{<k}}\|_{C^n(\mathcal{D}(\frac{l}{m},\frac{l+1}{m}))}\leq C^n_k\|f\|_{C^{k}(\mathcal D)}$ for every $0\leq l<m$ . On the other hand, $D^j(e_f)=0$ for every $0\leq j< k$.
In view of Lemma~\ref{lem:estDF}, if $(n_1,n_2)\in\Z^2_{\geq 0}$ is such that $n_1+n_2=n= k-(m-2)$ then for every $\omega\in \mathscr A$,
\[\left|\frac{\partial^{n} {F_{e_f}}(\omega,{\overline \omega})}{\partial \omega^{n_1}\partial {\overline \omega}^{n_2}}\right| \leq
C_{r,n}\|e_f\|_{C^{k}(\mathcal{D})}(1+|\log|\omega||)|\omega|^{(k-n-(m-2))\wedge 0}.\]
As $\|D^{k-(m-2)}F_{e_f}(\omega,{\overline \omega})\|\leq C_{r,n}\|e_f\|_{C^{k}(\mathcal{D})}(1+|\log|\omega||)$,
$D^{k-(m-1)}F_{e_f}$ on $\mathscr{A}$ has a continuous extension on $\overline{\mathscr{A}}=\mathscr{A}\cup\{(0,0)\}$ with the modulus of continuity bounded by a multiplicity of $\eta$.
Therefore, $F_{e_f}$ can be extended to a $C^{k-(m-1)+\eta}$-function on $\overline{\mathscr{A}}$
 and  $\|F_{e_f}\|_{C^{k-(m-1)+\eta}(\overline{\mathscr{A}})}\leq C\|e_f\|_{C^{k}(\mathcal{D})}\leq C'\|f\|_{C^{k}(\mathcal{D})}$. As $F_f=F_{f_{<k}}+F_{e_f}$,
this gives our claim.
\end{proof}

\section{Local analysis of $\varphi_f$}\label{sec;laphi}
This section is devoted to computing a limiting behavior of higher derivatives of $\varphi_f$ related to singularities on  angular sectors of $\mathcal{D}$. We introduce a family of functionals $\mathscr{C}^k_l$ which are responsible
for the asymptotic behaviour of $\varphi_{f,l}$ around zero. The new result is inspired by the  approach for multi-saddles (related to polynomial singularities) in \cite{Fr-Ki}.
The main results of this section (Theorem~\ref{cor:Hold}) plays a central role in proving Theorem \ref{thm1} in \S \ref{sec;GP} as well as
is applied to extend the regularity of $F_f$ (obtained in Theorem~\ref{thm:Cketa}) to the closure of any sector $\mathcal{D}(\frac{l}{m},\frac{l+1}{m})$.

\subsection{Preliminary properties of $\mathfrak G_{a_1,a_2}(1,s)$}\label{sec;prelim}
Firstly we present limiting behaviour of $\frac{d^n}{ds^n}\mathfrak G_{a_1,a_2}(1,s)$ around zero. We show that for large enough higher derivatives  their asymptotic is polynomial with a weight factor established by the Beta-like function $\mathfrak{B}$. This is further used in evaluating {asymptotics} of $D^{n+1}\varphi_{f,l}$ in \S \ref{sec;limfact}. \medskip

Note that for any pair of integers $a_1,a_2$,
\begin{align}\label{eq:derGa}
\begin{split}
\frac{d}{ds}\mathfrak G_{a_1,a_2}(1,s)&=-\frac{2\iota a_1}{m}\mathfrak G_{a_1+m,a_2}(1,s)+ C^\omega((0,1])\cap C^\omega([-1,0))\\
&=\frac{2\iota a_2}{m}\mathfrak G_{a_1,a_2+m}(1,s)+ C^\omega((0,1])\cap C^\omega([-1,0)).
\end{split}
\end{align}
Indeed,
\begin{align*}
\frac{d}{ds}\mathfrak G_{a_1,a_2}(u,s)&=\int_{-1}^u\frac{d}{ds}\mathfrak F_{a_1,a_2}(v,s) dv=\int_{-1}^u \iota\left(\frac{\partial}{\partial z}-\frac{\partial}{\partial \overline z}\right)\mathfrak F_{a_1,a_2}(v,s) dv\\
&=\int_{-1}^u \iota\left(-\frac{a_1}{m}\mathfrak F_{a_1+m,a_2}(v,s)+\frac{a_2}{m}\mathfrak F_{a_1,a_2+m}(v,s)\right)\,dv\\
&=\frac{\iota}{m}(-{a_1}\mathfrak G_{a_1+m,a_2}(u,s)+{a_2}\mathfrak G_{a_1,a_2+m}(u,s)).
\end{align*}
In view of \eqref{eq:anal1}, this gives \eqref{eq:derGa}.
It follows that
for every $n\geq 1$,
\begin{equation}\label{eq:dnds}
\frac{d^n}{ds^n}\mathfrak G_{a_1,a_2}(1,s)=n!(-2\iota)^n \binom{-\frac{a_2}{m}}{n}\mathfrak G_{a_1,a_2+nm}(1,s)+ C^\omega((0,1])\cap C^\omega([-1,0))
\end{equation}
and
\begin{equation}\label{eq:dndsus}
\frac{d^n}{ds^n}\mathfrak G_{a_1,a_2}=\iota^n n!\sum_{0\leq j\leq n}(-1)^{n-j} \binom{-\frac{a_1}{m}}{j}\binom{-\frac{a_2}{m}}{n-j}\mathfrak G_{a_1+jm,a_2+(n-j)m}.
\end{equation}

\medskip

Suppose that $a_1,a_2$ are integers such that  $a_1+a_2>m$. Then for every  $s\in(0,1]$,
\begin{align*}
\mathfrak G^0_{a_1,a_2}(1,s)&=\int_{-1}^1G_0(v+\iota s)^{-a_1}\overline{G_0(v+\iota s)}^{-a_2} dv\\&=
s^{-\frac{a_1+a_2-m}{m}}\int_{-1/s}^{1/s}G_0(x+\iota)^{-a_1}\overline{G_0(x+\iota)}^{{-a_2}} dx.
\end{align*}
{Since $|G_0(x+\iota)^{-a_1}\overline{G_0(x+\iota)}^{{-a_2}}|=\sqrt{x^2+1}^{-\frac{a_1+a_2}{m}}$ with $\frac{a_1+a_2}{m}>1$, we have}
\begin{align*}
\lim_{s\to 0^+} s^{\frac{a_1+a_2-m}{m}}\mathfrak G^0_{a_1,a_2}(1,s)=\mathfrak{B}(\tfrac{a_1}{m},\tfrac{a_2}{m}):=\int_{\R}G_0(x+\iota)^{-a_1}\overline{G_0(x+\iota)}^{{-a_2}} dx.
\end{align*}
Note that, {making the substitution $x=\cot t$, we obtain}
\[\mathfrak{B}(\tfrac{a_1}{m},\tfrac{a_2}{m})=\int_0^{\pi}\frac{e^{\iota\frac{a_2-a_1}{m}t}}{\sin^{-\frac{a_1+a_2}{m}+2}t}dt\]
and for $\frac{a_1}{m},\frac{a_2}{m}\notin \Z_{\leq 0}$,
\[\int_0^{\pi}\frac{e^{\iota\frac{a_2-a_1}{m}t}}{\sin^{-\frac{a_1+a_2}{m}+2}t}dt=\frac{\pi e^{\iota\frac{a_2-a_1}{2m}\pi}}{2^{\frac{a_1+a_2-2m}{m}}\frac{a_1+a_2-m}{m}} \frac{\Gamma(\frac{a_1}{m}+\frac{a_2}{m}-1)}{\Gamma(\frac{a_1}{m})\Gamma(\frac{a_2}{m})}.\]
For any pair $x,y$ of real numbers such that  $x,y\notin \Z_{\leq 0}$ and $x+y\notin \Z_{\leq 1}$ let
\[\mathfrak{B}(x,y)=\frac{\pi e^{\iota\frac{\pi}{2}(y-x)}}{2^{x+y-2}(x+y-1)B(x,y)}=\frac{\pi e^{\iota\frac{\pi}{2}(y-x)}}{2^{x+y-2}}\frac{\Gamma(x+y-1)}{\Gamma(x)\Gamma(y)}\neq 0.\]
Note that
\begin{equation}\label{eq:conj}
\overline{\mathfrak{B}(x,y)}=\mathfrak{B}(y,x)= e^{-\iota{\pi}(y-x)}\mathfrak{B}(x,y).
\end{equation}
By \eqref{eq:Gl}, for any pair of integers $a_1,a_2$ such that $a_1+a_2>m$ and $\frac{a_1}{m},\frac{a_2}{m}\notin \Z_{\leq 0}$,
\begin{align}\label{eq:0pm}
\begin{split}
\lim_{s\to 0^+} |s|^{\frac{a_1+a_2-m}{m}}\mathfrak G^l_{a_1,a_2}(1,s)&=\theta_0^{2l(a_2-a_1)}\mathfrak{B}(\tfrac{a_1}{m},\tfrac{a_2}{m}),\\
\lim_{s\to 0^-} |s|^{\frac{a_1+a_2-m}{m}}\mathfrak G^l_{a_1,a_2}(1,s)&=\theta_0^{(2l+1)(a_2-a_1)}\mathfrak{B}(\tfrac{a_1}{m},\tfrac{a_2}{m}).
\end{split}
\end{align}
In view of \eqref{eq:anal11}, if $\frac{a_1}{m}\in \Z_{\leq 0}$ or $\frac{a_2}{m}\in \Z_{\leq 0}$ then the limit is zero.
For this reason, we extend the definition of the function $\mathfrak{B}$ by letting
\begin{equation}\label{def:Bint}
\mathfrak{B}(x,y)=0\text{ if  } x\in \Z_{\leq 0}\text{ or }y\in \Z_{\leq 0}.
\end{equation}

{In the next three lemmas, we present a crucial decomposition of the map $\mathfrak G^l_{a_1,a_2}$ depending on the value of $a=a_1+a_2$. We deal with three cases: $a>m$, $a=m$ and $a<m$.}
\begin{lemma}\label{lem:Bs^}
Suppose that $a=a_1+a_2>m$.
For every $0<r<1$ there exist $\rho^\pm,\varrho^\pm\in C^\omega([0,r])$ such that if $0<u\leq 1$ and $0<|s|\leq ru$ then
\begin{align}\label{eq:Bs^-}
\begin{split}
\mathfrak G^l_{a_1,a_2}(u,s)&\!=\!\theta_0^{2l(a_2-a_1)}\!\big(\mathfrak{B}(\tfrac{a_1}{m},\tfrac{a_2}{m})|s|^{-\frac{a-m}{m}}+\rho^+(|s|)+u^{-\frac{a-m}{m}}\varrho^{+}(\tfrac{|s|}{u})\big)
\text{ if }s>0,\\
\mathfrak G^l_{a_1,a_2}(u,s)&\!=\!\theta_0^{(2l+1)(a_2-a_1)}\!\big(\mathfrak{B}(\tfrac{a_1}{m},\tfrac{a_2}{m})|s|^{-\frac{a-m}{m}}\!+\!\rho^-(|s|)\!+\!u^{-\frac{a-m}{m}}\!\varrho^{-}(\tfrac{|s|}{u})\big)
\text{ if }s<0.
\end{split}
\end{align}
\end{lemma}

\begin{proof}
In view of \eqref{eq:Glu} and \eqref{eq:conj}, it suffices to show the first line of \eqref{eq:Bs^-} for $l=0$.

 By change of variables used twice, for every $s,u\in (0,1]$,
\begin{align*}
\mathfrak G^0_{a_1,a_2}(u,s)&=
s^{-\frac{a-m}{m}}\int_{-1/s}^{u/s}G_0(x+\iota)^{-a_1}\overline{G_0(x+\iota)}^{{-a_2}} dx\\
&=s^{-\frac{a-m}{m}}\Big(\int_{-\infty}^{+\infty}-\int_{-\infty}^{-1/s}-\int_{u/s}^{+\infty}\Big)G_0(x+\iota)^{-a_1}\overline{G_0(x+\iota)}^{{-a_2}} dx\\
&=s^{-\frac{a-m}{m}}\mathfrak{B}(\tfrac{a_1}{m},\tfrac{a_2}{m})+s^{1-\frac{a}{m}}\Big(\int_0^{s/u}\frac{\xi_+(t)}{t^{2-\frac{a}{m}}}dt+\int_0^s\frac{\xi_-(t)}{t^{2-\frac{a}{m}}}dt\Big),
\end{align*}
where $\xi_{\pm}:\R\to\C$, $\xi_\pm(t)=-G_0(\pm 1+\iota t)^{-a_1}\overline{G_0(\pm 1+\iota t)}^{{-a_2}}$
is an analytic map with the radius of convergence at $0$ equal to $1$. Then for every $0<r< 1$ let $\sum_{n\geq 0}|c^\pm_n|r^n<+\infty$ such that $\sum_{n\geq 0}c^\pm_nt^n$ tends to $\xi_\pm(t)$ uniformly on $[0,r]$.
As $\frac{a}{m}>1$,
\[\frac{1}{t^{2-\frac{a}{m}}}\sum_{n\geq 1}c^\pm_nt^n=\sum_{n\geq 1}c^\pm_nt^{(n-1)+\frac{a}{m}-1}\text{ tends on $[0,r]$ uniformly to }\frac{\xi_\pm(t)-c^\pm_0}{t^{2-\frac{a}{m}}}.\]
It follows that
\[s^{1-\frac{a}{m}}\int_0^s\frac{(\xi_\pm(t)-c^\pm_0)}{t^{2-\frac{a}{m}}}dt=s^{1-\frac{a}{m}}\sum_{n\geq 1}c^\pm_n\int_0^st^{(n-1)+\frac{a}{m}-1}dt=\sum_{n\geq 1}\frac{c^\pm_ns^n}{n+\frac{a}{m}-1}.\]
Since $\sum_{n\geq 0}|c^\pm_n|r^n<+\infty$, the map $s^{1-\frac{a}{m}}\int_0^s\frac{(\xi_\pm(t)-c^\pm_0)}{t^{2-\frac{a}{m}}}dt\in C^\omega([0,r])$. Moreover,
\[s^{1-\frac{a}{m}}\int_0^s\frac{\xi_\pm(t)}{t^{2-\frac{a}{m}}}dt =\frac{c^\pm_0}{\frac{a}{m}-1}+s^{1-\frac{a}{m}}\int_0^s\frac{(\xi_\pm(t)-c^\pm_0)}{t^{2-\frac{a}{m}}}dt,\]
so $\widetilde{\xi}_\pm(s)=s^{1-\frac{a}{m}}\int_0^s\frac{\xi_\pm(t)}{t^{2-\frac{a}{m}}}dt\in C^\omega([0,r])$. As
\[\mathfrak G^0_{a_1,a_2}(u,s)=s^{-\frac{a-m}{m}}\mathfrak{B}(\tfrac{a_1}{m},\tfrac{a_2}{m})+u^{-\frac{a-m}{m}}\widetilde{\xi}_+(s/u)+\widetilde{\xi}_-(s)\text{ if }0\leq s/u\leq r,\]
this completes the proof of \eqref{eq:Bs^-}.
\end{proof}

By definition, for every natural number $n$ if $x,y\notin\Z$ and $x+y\notin\Z_{\leq 1}$  then
\begin{equation}\label{eq:beta}
(2\iota)^n\tbinom{-y}{n}\mathfrak{B}(x,y+n)=\tbinom{x+y+n-2}{n}\mathfrak{B}(x,y)=(-2\iota)^n\tbinom{-x}{n}\mathfrak{B}(x+n,y).
\end{equation}
We can extend again the domain of the function $\mathfrak{B}$ by adding the pairs $(x,y)$ such that $x,y\notin\Z$ and $x+y\in \Z_{\leq 1}$. For every such pair we let
$\mathfrak{B}(x,y)=\frac{\pi e^{\iota\frac{\pi}{2}(y-x)}}{2^{x+y-2}}\frac{\Gamma(x+y-1)}{\Gamma(x)\Gamma(y)}$, where we adopt the convention
\[\Gamma(0):=\lim_{x\to 0}x\Gamma(x)=1 \text{ and } \Gamma(-n):=\frac{\Gamma(0)}{(-1)\cdots(-n)}=\frac{1}{(-1)\cdots(-n)}\] for any $n\in\N$.
Then we also have
\begin{equation}\label{eq:beta1}
\tbinom{-x}{n}\mathfrak{B}(x+n,y-n)=(-1)^n\tbinom{-y+n}{n}\mathfrak{B}(x,y).
\end{equation}
The extended $\Gamma$-function satisfies $\Gamma(x+1)=x\Gamma(x)$ for all $x\in\R\setminus \{0\}$ and $\Gamma(1)=\Gamma(0)=1$. It follows that
\eqref{eq:beta} holds even when $x+y+n\in\Z_{\leq 1}$.


Finally note that, if $x+y=1$ then we also have
\begin{equation}\label{eq:beta3}
\mathfrak{B}(x,y)=\frac{2\pi e^{\iota\pi(y-1/2)}}{\Gamma(1-y)\Gamma(y)}=-2\iota e^{\iota\pi y}\sin (\pi y)=1-e^{2\iota\pi y}=1+e^{\iota\pi(y-x)}.
\end{equation}

\begin{lemma}\label{lem:Blog}
Suppose that $a=a_1+a_2=m$. 
There exist $\rho^\pm,\varrho^\pm\in C^\omega([0,1])$ such that for all $0<u\leq 1$ and $0<|s|\leq u$,
\begin{align}\label{eq:Bs^-log}
\begin{split}
&\mathfrak G^l_{a_1,a_2}(u,s)\!=\!\theta_0^{2l(a_2-a_1)}\!\big(\!-\!\mathfrak{B}(\tfrac{a_1}{m},\tfrac{a_2}{m})\!\log|s|\!+\!\log u\!+\!\rho^+(|s|)\!+\!\varrho^{+}(\tfrac{|s|}{u})\big)\! \text{ if }s>0,\\
&\mathfrak G^l_{a_1,a_2}(u,s)\!=\!\theta_0^{(2l+1)(a_2-a_1)}\!\big(\!-\!\mathfrak{B}(\tfrac{a_1}{m},\tfrac{a_2}{m})\!\log|s|\!+\!\theta_0^{a_2-a_1}
\log u\!+\!\rho^-(|s|)\!+\!\varrho^{-}(\tfrac{|s|}{u})\big)\! \text{ if }s<0.
\end{split}
\end{align}
\end{lemma}
\begin{proof}
In view of \eqref{eq:Glu} and \eqref{eq:conj}, it suffices to show the first line of  \eqref{eq:Bs^-log} for $l=0$.

By change of variables, for every $u,s\in (0,1]$,
\begin{align*}
\psi(u,s)&:=\int_{0}^uG_0(v+\iota s)^{-a_1}\overline{G_0(v+\iota s)}^{-a_2}dv-\int_{0}^u\frac{1}{v-\iota s} dv\\
&=\int_{0}^u\left(\left(\frac{\overline{G_0(v+\iota s)}}{G_0(v+\iota s)}\right)^{a_1}-1\right)\frac{1}{v-\iota s} dv
=\int_{0}^{u/s}\frac{\left(\left(\frac{\overline{G_0(x+\iota)}}{G_0(x+\iota)}\right)^{a_1}-1\right)}{x-\iota} dx.
\end{align*}
It follows that $\psi(u,s)=\widetilde{\psi}(s/u)$, where
\[\widetilde{\psi}'(x)=-\frac{1}{x^2}\frac{\left(\frac{\overline{G_0(1/x+\iota)}}{G_0(1/x+\iota)}\right)^{a_1}-1}{1/x-\iota}
=\frac{\left(\frac{\overline{G_0(1+\iota x)}}{G_0(1+\iota x)}\right)^{a_1}-1}{x}\frac{1}{\iota x-1}.\]
As the map $x\mapsto \left(\frac{\overline{G_0(1+\iota x)}}{G_0(1+\iota x)}\right)^{a_1}-1$ is real analytic and vanishes at $0$, $\widetilde{\psi}'$ is also analytic. Hence $\widetilde{\psi}\in C^\omega(\R)$.
Moreover,
\begin{align*}
\int_{0}^u\frac{1}{v-\iota s} dv=\int_{0}^u\frac{v+\iota s}{{v^2+s^2}} dv= \log\sqrt{u^2+s^2}-\log s+\iota \operatorname{arccot} (s/u).
\end{align*}
Hence
\[\int_{0}^uG_0(v+\iota s)^{-a_1}\overline{G_0(v+\iota s)}^{-a_2}dv=-\log (s/u)+\varrho^+(s/u),\]
where
$\varrho^+(x)=\log\sqrt{1+x^2}+\widetilde{\psi}(x)+\iota \operatorname{arccot} (x)$ is analytic.
In particular,
\begin{equation}\label{eq:=m}
\int_{0}^1G_0(v+\iota s)^{-a_1}\overline{G_0(v+\iota s)}^{-a_2}dv=-\log s+\varrho^+(s).
\end{equation}
Since
\[\int_{-1}^0G_0(v+\iota s)^{-a_1}\overline{G_0(v+\iota s)}^{-a_2}dv=\int_{0}^1G_0(-v+\iota s)^{-a_1}\overline{G_0(-v+\iota s)}^{-a_2}dv\]
and $G_0(-v+\iota s)=\theta_0 \overline{G_0(v+\iota s)}$ if $s,v\in(0,1]$, we get
\[\int_{-1}^0G_0(v+\iota s)^{-a_1}\overline{G_0(v+\iota s)}^{-a_2}dv=\theta_0^{(a_2-a_1)}\int_{0}^1G_0(v+\iota s)^{-a_2}\overline{G_0(v+\iota s)}^{-a_1}dv.\]
In view of \eqref{eq:=m}, this gives
\[\mathfrak G^0_{a_1,a_2}(u,s)=-(1+\theta_0^{(a_2-a_1)})\log s+\log u + \theta_0^{(a_2-a_1)}\overline{\varrho^+(s)}+\varrho^+(s/u).\]
Since, by \eqref{eq:beta3},
$1+\theta_0^{(a_2-a_1)}=1+e^{\pi \iota(\frac{a_2}{m}-\frac{a_1}{m})}=\mathfrak{B}(\tfrac{a_1}{m},\tfrac{a_2}{m})$,
which gives the first line of \eqref{eq:Bs^-log}.
\end{proof}

\begin{lemma}\label{lem:Bmix}
Suppose that $a=a_1+a_2<m$  and $\frac{a_1}{m},\frac{a_2}{m}\notin\Z$.
If $\frac{a}{m}\notin\Z$ then for every $0<r<1$ there exist $\rho^\pm,\varrho^\pm\in C^\omega([0,r])$ such that if $0<u\leq 1$ and $0<|s|\leq ru$ then
\begin{align}\label{eq:Bs^-low}
\begin{split}
\mathfrak G^l_{a_1,a_2}(u,s)&\!=\!\theta_0^{2l(a_2-a_1)}\big(\mathfrak{B}(\tfrac{a_1}{m},\tfrac{a_2}{m})|s|^{\frac{m-a}{m}}
\!+\!\rho^+(|s|)\!+\!u^{\frac{m-a}{m}}\varrho^{+}(\tfrac{|s|}{u})\big)\! \text{ if }s>0,\\
\mathfrak G^l_{a_1,a_2}(u,s)&\!=\!\theta_0^{(2l+1)(a_2-a_1)}\big(\mathfrak{B}(\tfrac{a_1}{m},\tfrac{a_2}{m})|s|^{\frac{m-a}{m}}
\!+\!\rho^-(|s|)\!+\!u^{\frac{m-a}{m}}\varrho^{-}(\tfrac{|s|}{u})\big) \!\text{ if }s<0.
\end{split}
\end{align}
If $\frac{a}{m}\in\Z$ then there exist $\rho^\pm,\varrho^\pm\in C^\omega([0,1])$ and $c_\pm\in \C$  such that if $0<u\leq 1$ and $0<|s|\leq u$ then
\begin{align}\label{eq:Bs^-lowlog}
\begin{split}
\mathfrak G^l_{a_1,a_2}(u,s)&=\theta_0^{2l(a_2-a_1)}\Big(-\mathfrak{B}(\tfrac{a_1}{m},\tfrac{a_2}{m})|s|^{\frac{m-a}{m}}\log|s|\\
&\quad+c_+|s|^{\frac{m-a}{m}}\log u+\rho^+(|s|)+u^{\frac{m-a}{m}}\varrho^{+}(\tfrac{|s|}{u})\Big) \text{ if }s>0,\\
\mathfrak G^l_{a_1,a_2}(u,s)&=\theta_0^{(2l+1)(a_2-a_1)}\Big(-\mathfrak{B}(\tfrac{a_1}{m},\tfrac{a_2}{m})|s|^{\frac{m-a}{m}}\log|s|\\
&\quad+c_-|s|^{\frac{m-a}{m}}\log u+\rho^-(|s|)+u^{\frac{m-a}{m}}\varrho^{-}(\tfrac{|s|}{u})\Big) \text{ if }s<0.
\end{split}
\end{align}
\end{lemma}

\begin{proof}
In view of \eqref{eq:Glu} and \eqref{eq:conj}, it suffices to show the first line of \eqref{eq:Bs^-low} and \eqref{eq:Bs^-lowlog} for $l=0$. Let $n=\lceil \frac{m-a}{m}\rceil$.
By \eqref{eq:dndsus}, for every $k\geq 0$,
\[\frac{d^k}{ds^k}\mathfrak G^0_{a_1,a_2}=\iota^k\sum_{0\leq j\leq k}k!(-1)^{k-j} \binom{-\frac{a_1}{m}}{j}\binom{-\frac{a_2}{m}}{k-j}\mathfrak G^0_{a_1+jm,a_2+(k-j)m}.\]
A direct computation shows that if $a=a_1+a_2<m$ and $u\in[0,1]$ then
\[\mathfrak G^0_{a_1,a_2}(u,0)=\frac{m}{m-a}(u^{\frac{m-a}{m}}+\theta_0^{a_2-a_1}).\]
It follows that if $k<\frac{m-a}{m}$ (i.e.\ $k<n$) then there exist $c_{k,1},c_{k,0}\in\C$ such that
\begin{equation}\label{eq:derk}
\frac{d^k}{ds^k}\mathfrak G^0_{a_1,a_2}(u,0)=c_{k,1}u^{\frac{m-a}{m}-k}+c_{k,0}.
\end{equation}

 If $\frac{a}{m}\notin \Z$ then for any $0\leq j\leq n$ we have $a_1+jm+a_2+(n-j)m=a+nm>m$. Hence, by Lemma~\ref{lem:Bs^}, there exist $\rho_n^+,\varrho_n^+\in C^{\omega}([0,r])$ such that for all $0<u\leq 1$ and $0<s\leq ru$,
\begin{align*}
\frac{d^n}{ds^n}&\mathfrak G^0_{a_1,a_2}(u,s)=\rho_n^+(s)+u^{\frac{m-a}{m}-n}\varrho_n^{+}(\tfrac{s}{u})\\
&+\iota^n\sum_{0\leq j\leq n}n!(-1)^{n-j} \binom{-\frac{a_1}{m}}{j}\binom{-\frac{a_2}{m}}{n-j}
\mathfrak{B}(\tfrac{a_1}{m}+j,\tfrac{a_2}{m}+n-j)s^{\frac{m-a}{m}-n}.
\end{align*}

 If $\frac{a}{m}\in \Z$ then $n=\frac{m-a}{m}$ and $a_1+jm+a_2+(n-j)m=a+nm=m$. Hence, by Lemma~\ref{lem:Blog}, there exist $\rho_n^+,\varrho_n^+\in C^{\omega}([0,1])$ such that for all $0<u\leq 1$ and $0<s\leq u$,
\begin{align*}
\frac{d^n}{ds^n}&\mathfrak G^0_{a_1,a_2}(u,s)=\rho_n^+(s)+\varrho_n^{+}(\tfrac{s}{u})\\
&+\iota^n\sum_{0\leq j\leq n}n!(-1)^{n-j} \binom{-\frac{a_1}{m}}{j}\binom{-\frac{a_2}{m}}{n-j}
\big(-\mathfrak{B}(\tfrac{a_1}{m}+j,\tfrac{a_2}{m}+n-j)\log s +\log u).
\end{align*}

By \eqref{eq:beta} (and its extension in the integer case),
\begin{gather*}
\iota^n\sum_{0\leq j\leq n}n!(-1)^{n-j} \binom{-\frac{a_1}{m}}{j}\binom{-\frac{a_2}{m}}{n-j}
\mathfrak{B}(\tfrac{a_1}{m}+j,\tfrac{a_2}{m}+n-j)\\
=(-1/2)^n\sum_{0\leq j\leq n}n!\binom{n}{j}\binom{\frac{a}{m}+n-2}{n}
\mathfrak{B}(\tfrac{a_1}{m},\tfrac{a_2}{m})=(-1)^nn!\binom{\frac{a}{m}+n-2}{n}
\mathfrak{B}(\tfrac{a_1}{m},\tfrac{a_2}{m}).
\end{gather*}
Therefore, in the non-integer case, for all $0<u\leq 1$ and $0<s\leq ru$,
\[\frac{d^n}{ds^n}\mathfrak G^0_{a_1,a_2}(u,s)=\rho_n^+(s)+u^{\frac{m-a}{m}-n}\varrho_n^{+}(\tfrac{s}{u})+(-1)^nn!\tbinom{-(\frac{m-a}{m}-(n-1))}{n}
\mathfrak{B}(\tfrac{a_1}{m},\tfrac{a_2}{m})s^{\frac{m-a}{m}-n}.\]
In the integer case, for all $0<u\leq 1$ and $0<s\leq u$,
\[\frac{d^n}{ds^n}\mathfrak G^0_{a_1,a_2}(u,s)=\rho_n^+(s)+\varrho_n^{+}(\tfrac{s}{u})-n!
\mathfrak{B}(\tfrac{a_1}{m},\tfrac{a_2}{m})\log s +c_n\log u.\]
Since
\[\frac{d^k}{ds^k}\mathfrak G^0_{a_1,a_2}(u,s)=\frac{d^k}{ds^k}\mathfrak G^0_{a_1,a_2}(u,0)+\int_0^s\frac{d^{k+1}}{ds^{k+1}}\mathfrak G^0_{a_1,a_2}(u,t)dt\text{ for all }0\leq k<n,\]
using the  formulae for $\frac{d^n}{ds^n}\mathfrak G^0_{a_1,a_2}$ together with \eqref{eq:derk} and induction, we obtain \eqref{eq:Bs^-low} and \eqref{eq:Bs^-lowlog}.
\end{proof}

\begin{remark}
To summarize, by Lemmas~\ref{lem:Bs^}, \ref{lem:Blog} and \ref{lem:Bmix}, for any pair of integer numbers $a_1,a_2$ such that $\frac{a_1}{m},\frac{a_2}{m}\notin\Z$ if $\frac{m-a}{m}\notin \Z$ ($a=a_1+a_2$) or $\frac{m-a}{m}\in \Z_{<0}$ then
\begin{align}\label{eq:B1s}
\begin{split}
\mathfrak G^l_{a_1,a_2}(1,s)&=\theta_0^{2l(a_2-a_1)}\mathfrak{B}(\tfrac{a_1}{m},\tfrac{a_2}{m})|s|^{\frac{m-a}{m}}+C^\omega((0,1]),\\
\mathfrak G^l_{a_1,a_2}(1,s)&=\theta_0^{(2l+1)(a_2-a_1)}\mathfrak{B}(\tfrac{a_1}{m},\tfrac{a_2}{m})|s|^{\frac{m-a}{m}}+C^\omega([-1,0)).
\end{split}
\end{align}
If  $\frac{m-a}{m}\in \Z_{\geq 0}$ then
\begin{align}\label{eq:B1sln}
\begin{split}
\mathfrak G^l_{a_1,a_2}(1,s)&=-\theta_0^{2l(a_2-a_1)}\mathfrak{B}(\tfrac{a_1}{m},\tfrac{a_2}{m})|s|^{\frac{m-a}{m}}\log |s|+C^\omega((0,1]),\\
\mathfrak G^l_{a_1,a_2}(1,s)&=-\theta_0^{(2l+1)(a_2-a_1)}\mathfrak{B}(\tfrac{a_1}{m},\tfrac{a_2}{m})|s|^{\frac{m-a}{m}}\log |s|+C^\omega([-1,0)).
\end{split}
\end{align}
Indeed, in the non-integer case, we obtain the analyticity of the remainder only on intervals $[-r,0]$ and $[0,r]$ for any $0<r<1$. Nevertheless, for any choice of integer $a_1,a_2$, the function $\mathfrak G^l_{a_1,a_2}(1,s)$ is analytic on $[r,1]$ and $[-1,-r]$ for any $0<r<1$. This gives our claim.
\end{remark}

\subsection{Evaluation of asymptotic factors for $\varphi_{f,l}$}\label{sec;limfact}
The behaviour of higher derivatives of $\varphi_{f,l}$ at zero is evaluated by linear combinations of invariant distributions $\partial_j^{k}$.
For this reason, we define a list of new functionals $\mathscr{C}^k_l:C^{k}(\mathcal{D})\to\C$ for $k\geq 0$ and $0\leq l<2m$ given by
\begin{align*}
\mathscr{C}^k_{l}(f)=\sum_{\substack{0\leq j\leq  k\wedge(m-2)\\j\neq k-(m-1)\operatorname{mod} m}}
\theta_0^{l(2j-k)}\mathfrak{B}(\tfrac{(m-1)-j}{m},\tfrac{(m-1)-(k-j)}{m})\partial_j^{k}(f).
\end{align*}
{As we will see in the next theorem, the values of $\mathscr{C}^k_{l}(f)$ for $k\geq 0$ measure the effect that the appearance of orbits  in the $l$-th angular sector has on the value of ergodic integrals of the function $f$.}
Comparing with \eqref{def:gothc}, functionals $\mathscr{C}^k_{l}$  play a key role in understanding the meaning of distribution $\mathfrak{C}^k_{\sigma,l}$.

From now on, we adopt the convention $\tbinom{0}{n}:=\lim_{x\to 0}\tbinom{x}{n}/x=\frac{(-1)^{n-1}}{n}$.


\begin{theorem}\label{thm:thmC}
For any  $k\geq 0$ let $n=\lceil\frac{k-(m-2)}{m}\rceil$ and $b=n-\frac{k-(m-2)}{m}$.
Suppose that $f\in C^{k\vee(n+1)}(\mathcal D)$ is  such that  $\partial^j_{i}(f)=0$ for all $0\leq j<k$ and $0\leq i\leq j\wedge(m-2)$ with $i\neq j-(m-1)\operatorname{mod} m$.  Then $\varphi_{f,l}\in C^{n+\operatorname{P_b}}([-1,0)\cup(0,1])$ and there exists $C>0$ such that $\|\varphi_{f,l}\|_{C^{n+\operatorname{P_b}}([-1,0)\cup(0,1])}\leq C\|f\|_{C^{k\vee(n+1)}(\mathcal D)}$. Moreover,
for every $0\leq l<m$,
\begin{align}
\label{eq:b+1+}
\lim_{s\to 0^+}|s|^{b+1}D^{n+1}\varphi_{f,l}(s)&=(-1)^{n+1}\frac{(n+1)!}{k!}\binom{b}{n+1}\mathscr{C}^k_{2l}(f),\\
\label{eq:b+1-}
\lim_{s\to 0^-}|s|^{b+1}D^{n+1}\varphi_{f,l}(s)&=\frac{(n+1)!}{k!}\binom{b}{n+1}\mathscr{C}^k_{2l+1}(f).
\end{align}
\end{theorem}

\begin{remark}\label{rem:kwedgen}
Before the proof, let us note that
\[
k\vee(n+1)=\left\{
\begin{array}{cl}
k+1&\text{if }k=0\text{ or }(k=1\text{ with }m=2),\\
k&\text{otherwise.}
\end{array}
\right.
\]
Indeed, the inequality $\frac{k-(m-2)}{m}+1\leq k$ is equivalent to $2\leq k(m-1)$. It follows that if $k\geq 1$ with $m\geq 3$ or $k\geq 2$ then $n<k$, so $k\vee(n+1)=k$.
\end{remark}
\begin{proof}
Let us decompose $f=f_{< k}+f_{k}+e_f$ with
\begin{gather*}
f_{< k}(\omega,\overline \omega)=\sum_{0\leq j< k}\frac{1}{j!}\sum_{0\leq i\leq j}\binom{j}{i}\frac{\partial^j f}{\partial\omega^i\partial\overline{\omega}^{j-i}}(0,0)\omega^i\overline{\omega}^{j-i},\\
f_{k}(\omega,\overline \omega)=\frac{1}{k!}\sum_{0\leq i\leq k}\binom{k}{i}\frac{\partial^{k} f}{\partial\omega^i\partial\overline{\omega}^{k-i}}(0,0)\omega^i\overline{\omega}^{k-i}.
\end{gather*}
By Lemma~\ref{lem:analf},
\begin{gather}\label{eq:anfk}
\varphi_{f_{< k},l}\in C^\omega([-1,0))\cap C^{\omega}((0,1])\text{ for }0\leq l<m,\\
\label{eq:anfk1}
\|\varphi_{f_{< k},l}\|_{C^{n+1}([-1,0)\cup(0,1])}\leq C^{n+1}_k\|f\|_{C^{k}(\mathcal{D})}.
\end{gather}
Since $D^j(f_{k}+e_f)=0$ for every $0\leq j< k$, in view of \eqref{eqn;F(z)bound1}, if $(j_1,j_2)\in\Z^2_{\geq 0}$ is such that $j_1+j_2=j
\leq n+1$ then
\begin{align*}
\left|\frac{\partial^{j} {\mathscr{F}_{f_{k}+e_f,l}}(z,{\overline z})}{\partial z^{j_1}\partial {\overline z}^{j_2}}\right| &=
O\big(\|f_{k}+e_f\|_{C^{k\vee(n+1)}(\mathcal{D})}\langle|\Im z|\rangle^{-(\frac{(m-2)-k}{m}+j)}\big)\\
&=O\big(\|f\|_{C^{k\vee(n+1)}(\mathcal{D})}\langle|\Im z|\rangle^{(n+1-j)-(b+1)}\big).
\end{align*}
As $\frac{d^{j}}{d s^{j}} = \left(\iota\left(\frac{\partial}{\partial {z}} - \frac{\partial}{\partial \overline {z}}\right)\right)^{j}$, this gives
\begin{align}
\label{neq:varnor1}
|D^{j}\varphi_{f_{k}+e_f,l}(s)|&=O\big(\|f\|_{C^{k\vee(n+1)}(\mathcal{D})}\big)\text{ if }0\leq j\leq n-1,\\
\label{neq:varnor2}
|D^{n}\varphi_{f_{k}+e_f,l}(s)|&=O\big(\|f\|_{C^{k\vee(n+1)}(\mathcal{D})}\langle|s|\rangle^{-b}\big),\\
\label{neq:varnor3}
|D^{n+1}\varphi_{f_{k}+e_f,l}(s)|&=O\big(\|f\|_{C^{k\vee(n+1)}(\mathcal{D})}|s|^{b+1}\big).
\end{align}
By \eqref{neq:varnor2},
\[\|D^{n}\varphi_{f_{k}+e_f,l}\|_{L^1}=O\big(\|f\|_{C^{k\vee(n+1)}(\mathcal{D})}\big).\]
In view of \eqref{eq:anfk1}, \eqref{neq:varnor1} and \eqref{neq:varnor3}, this gives
\[\|\varphi_{f,l}\|_{C^{n+\pb}}\leq \|\varphi_{f_{< k},l}\|_{C^{n+\pb}}+\|\varphi_{f_{k}+e_f,l}\|_{C^{n+\pb}}=O\big(\|f\|_{C^{k\vee(n+1)}(\mathcal{D})}\big).\]
Since $D^j(e_f)=0$ for every $0\leq j\leq k$, we also have
\begin{equation}
\label{neq:varnor4}
|D^{n+1}\varphi_{e_f,l}(s)|=o(|s|^{-(b+1)}).
\end{equation}
Indeed, if $k\vee(n+1)=k+1$, i.e.\ $k=n$ then, again by  \eqref{eqn;F(z)bound1},
\begin{align*}
\left\|D^{n+1} {\mathscr{F}_{e_f,l}}(z,{\overline z})\right\| =
O(\langle|\Im z|\rangle^{-(\frac{(m-2)-(k+1)}{m}+n+1)})= O(|\Im z|^{-(b+1)+\frac{1}{m}}).
\end{align*}
If $k\vee(n+1)=k$, i.e.\ $n+1\leq k$ then, by  \eqref{eqn;F(z)bound3},
\begin{align*}
\left\|D^{n+1} {\mathscr{F}_{e_f,l}}(z,{\overline z})\right\| =
o(|\Im z|^{-(\frac{(m-2)-k}{m}+n+1)})= o(|\Im z|^{-(b+1)}).
\end{align*}
Both yield \eqref{neq:varnor4}.

Therefore, by \eqref{eq:anfk} and \eqref{neq:varnor4},
\begin{equation}\label{eq:Dn+1}
 |s|^{b+1}D^{n+1}\varphi_{f,l}(s)=|s|^{b+1}D^{n+1}\varphi_{f_{k},l}(s)+o(1).
\end{equation}
By \eqref{eq:anxi1} (see the proof of Lemma~\ref{lem:analf}),
\[\mathscr{F}_{f_{k},l}(1,s)\!=\!\frac{1}{k!}\!\sum_{\substack{0\leq j\leq k\wedge(m-2)\\j\neq k-(m-1)\operatorname{mod} m}}\!\partial_j^{k}(f)\mathfrak G^l_{(m-1)-j,(m-1)-(k-j)}(1,s)+C^{\omega}((0,1])\cap C^{\omega}([-1,0)).\]
As $\frac{m-((m-1)-j+(m-1)-(k-j))}{m}=\frac{k-(m-2)}{m}=n-b$, by {the definition of ${\mathscr{C}}^k_{l}(f)$ and \eqref{eq:B1s},}
\begin{equation}\label{eq:B1s1}
\varphi_{f_k,l}(s)=\frac{{\mathscr{C}}^k_{2l}(f)}{k!}|s|^{n-b}+C^\omega((0,1]),\
\varphi_{f_k,l}(s)=\frac{{\mathscr{C}}^k_{2l+1}(f)}{k!}|s|^{n-b}+C^\omega([-1,0))
\end{equation}
if $0<b<1$ and, by \eqref{eq:B1sln},
\begin{align}\label{eq:B1sln1}
\begin{split}
\varphi_{f_k,l}(s)&=-\frac{\mathscr{C}^k_{2l}(f)}{k!}|s|^{n}\log |s|+C^\omega((0,1]),\\
\varphi_{f_k,l}(s)&=-\frac{{\mathscr{C}}^k_{2l+1}(f)}{k!}|s|^{n}\log|s|+C^\omega([-1,0))
\end{split}
\end{align}
if $b=0$. After $n+1$ times differentiation, it follows that
%
%
\begin{align*}\label{eq:fkl}
\begin{split}
&D^{n+1}\varphi_{f_k,l}(s)=|s|^{-(b+1)}(-1)^{n+1}\frac{(n+1)!}{k!}\binom{b}{n+1}\mathscr{C}^k_{2l}(f)+C^{\omega}((0,1]),\\
&D^{n+1}\varphi_{f_k,l}(s)=|s|^{-(b+1)}\frac{(n+1)!}{k!}\binom{b}{n+1}\mathscr{C}^k_{2l+1}(f)+C^{\omega}([-1,0)).
\end{split}
\end{align*}
Finally, by \eqref{eq:anfk} and \eqref{eq:Dn+1}, this yields \eqref{eq:b+1+} and \eqref{eq:b+1-}.
\end{proof}

{We also verify a similar result when we restrict our attention only to an individual $l$-th sector. Then we assume only the vanishing of functionals $\mathscr{C}^j_{l}$.}
\begin{theorem}\label{cor:Hold}
Let $k\geq 0$, $0\leq l<m$ and $\epsilon\in\{0,1\}$.
Suppose that $f\in C^{k\vee(n+1)}(\mathcal D)$ and
$\mathscr{C}^j_{2l+\epsilon}(f)=0$ for all $0\leq j<k$. Then $\varphi_{f,l}\in C^{n+\pb}((0,(-1)^\epsilon])$ with
\begin{equation}\label{eq:limcc}
\lim_{\substack{s\to 0\\ s\in (0,(-1)^\epsilon]}}|s|^{b+1}D^{n+1}\varphi_{f,l}(s)=(-1)^{(1-\epsilon)(n+1)}\frac{(n+1)!}{k!}\binom{b}{n+1}\mathscr{C}^k_{2l+\epsilon}(f)
\end{equation}
and
\begin{equation}\label{eq:norcc}
\text{there exists $C>0$ such that }\|\varphi_{f,l}\|_{C^{n+\pb}((0,(-1)^\epsilon])}\leq C\|f\|_{C^{k\vee(n+1)}(\mathcal D)}.
\end{equation}
In particular, if $k\geq m-1$  then $\varphi_{f,l}\in C^{\mathfrak{e}(\sigma,k)}((0,(-1)^\epsilon])$ and there exists $C>0$ such that
$\|\varphi_{f,l}\|_{C^{\mathfrak{e}(\sigma,k)}((0,(-1)^\epsilon])}\leq C\|f\|_{C^{k\vee (n+1)}(\mathcal D)}$.

On the other hand, if $f\in C^{k\vee(n+1)}(\mathcal D)$ is  such that $\varphi_{f,l}\in C^{r}((0,(-1)^\epsilon])$ for some $r\in \R_\eta$ with $0<v(r)\leq \mathfrak{o}(\sigma,k)$ then
$\mathscr{C}^j_{2l+\epsilon}(f)=0$ for all $j\geq 0$ such that $\mathfrak{o}(\sigma,j)<v(r)$.
\end{theorem}

\begin{proof}
We will focus only on the even case, when $\epsilon=0$. The proof in the odd case proceeds in the same way.
Let us decompose $f=f_{<k}+f_k+e_f$, where $f_{< k}=\sum_{0\leq j< k}f_{j}$ with
\begin{gather*}
f_{j}(\omega,\overline \omega)=\frac{1}{j!}\sum_{0\leq i\leq j}\binom{j}{i}\frac{\partial^j f}{\partial\omega^i\partial\overline{\omega}^{j-i}}(0,0)\omega^i\overline{\omega}^{j-i}.
\end{gather*}
By \eqref{eq:B1s1}, \eqref{eq:B1sln1},
\begin{align}\label{eq:phifl}
\begin{split}
\varphi_{f_j,l}(s)&=\frac{{\mathscr{C}}^j_{2l}(f)}{j!}s^{\frac{j-(m-2)}{m}}+C^\omega((0,1])\text{ if }j\neq m-2\ \operatorname{mod}m,\\
\varphi_{f_j,l}(s)&=-\frac{\mathscr{C}^j_{2l}(f)}{j!}s^{\frac{j-(m-2)}{m}}\log s+C^\omega((0,1]))\text{ if }j= m-2\ \operatorname{mod}m.
\end{split}
\end{align}
Since the operator $f\mapsto f_j$ takes values in the finite-dimensional space of homogenous polynomials of degree $j$, for every $0\leq j<k$ there exists $C_j>0$ such that
\begin{align*}
\|\varphi_{f_j,l}(s)-\frac{{\mathscr{C}}^j_{2l}(f)}{j!}s^{\frac{j-(m-2)}{m}}\|_{ C^{n+\pb}((0,1])}&\leq C_j\|f\|_{C^{k}(\mathcal D)}\text{ or } \\
\|\varphi_{f_j,l}(s)+\frac{\mathscr{C}^j_{2l}(f)}{j!}s^{\frac{j-(m-2)}{m}}\log s\|_{ C^{n+\pb}((0,1])}&\leq C_j\|f\|_{C^{k}(\mathcal D)}.
\end{align*}
If $\mathscr{C}^j_{2l}(f)=0$ for all $0\leq j<k$, then
\begin{equation}\label{eq:e<k}
\varphi_{f_{<k},l}\in C^\omega((0,1])\text{ and }\|\varphi_{f_{<k},l}\|_{ C^{n+\pb}((0,1])}\leq \sum_{0\leq j<k}C_j\|f\|_{C^{k}(\mathcal D)}.
\end{equation}
Again, by Theorem~\ref{thm:thmC} applied to $f_k+e_f$, we have $\varphi_{f_{k}+e_f,l}\in C^{n+\pb}((0,1])$,
\[\|\varphi_{f_{k}+e_f,l}\|_{ C^{n+\pb}((0,1])}\leq C\|f_k+e_f\|_{C^{k\vee (n+1)}(\mathcal D)}\leq  C'\|f\|_{C^{k\vee (n+1)}(\mathcal D)}\] and
\[\lim_{s\to 0^+}s^{b+1}D^{n+1}\varphi_{f_k+e_f,l}(s)=(-1)^{(n+1)}\frac{(n+1)!}{k!}\binom{b}{n+1}\mathscr{C}^k_{2l}(f).
\]
Since $\varphi_{f,l}=\varphi_{f_{<k},l}+\varphi_{f_k+e_f,l}$, in view of \eqref{eq:e<k}, this yields \eqref{eq:limcc} and \eqref{eq:norcc}.
As $n-b=\mathfrak{o}(\sigma,k)$,  by  Remark~\ref{rmk:Hold}, this gives $\varphi_{f,l}\in C^{\mathfrak{e}(\sigma,k)}((0,1])$.


%

\medskip

Now suppose that $f\in C^{k\vee(n+1)}(\mathcal D)$ is  such that $\varphi_{f,l}\in C^{r}((0,1])$ for some $r\in \R_\eta$ with $0<v(r)\leq \mathfrak{o}(\sigma,k)$. Choose $m-2<j_0\leq k$ such that
$\mathfrak{o}(\sigma,j_0-1)<v(r)\leq \mathfrak{o}(\sigma,j_0)$. By the first part of the theorem, $\varphi_{f-f_{< j_0},l}\in C^{\mathfrak{e}(\sigma,j_0)}((0,1])$.
As $\varphi_{f,l}\in C^{r}((0,1])$ and $v(r)\leq \mathfrak{o}(\sigma,j_0)$, it follows that $\varphi_{f_{< j_0},l}\in C^{r}((0,1])$. In view of \eqref{eq:phifl},
\begin{align*}
\varphi_{f_{<j_0},l}(s)&=\!\!\!\sum_{\substack{0\leq j<j_0\\j\neq m-2\operatorname{mod} m
}}\!\!\!\frac{{\mathscr{C}}^j_{2l}(f)}{j!}s^{\frac{j-(m-2)}{m}}+\!\!\!\sum_{\substack{0\leq j<j_0\\j= m-2\operatorname{mod} m
}}\!\!\!\frac{\mathscr{C}^j_{2l}(f)}{j!}s^{\frac{j-(m-2)}{m}}(-\log s)+C^\omega((0,1]).
\end{align*}
Therefore,
\[\sum_{\substack{0\leq j<j_0\\j\neq m-2\operatorname{mod} m
}}\frac{{\mathscr{C}}^j_{2l}(f)}{j!}s^{\frac{j-(m-2)}{m}}-\sum_{\substack{0\leq j<j_0\\j= m-2\operatorname{mod} m
}}\frac{\mathscr{C}^j_{2l}(f)}{j!}s^{\frac{j-(m-2)}{m}}\log s\in C^r((0,1])\]
with $\frac{j-(m-2)}{m}\leq \mathfrak{o}(\sigma,j_0-1)<v(r)$ for $0\leq j<j_0$. It follows that ${\mathscr{C}}^j_{2l}(f)=0$ for $0\leq j<j_0$.
%
\end{proof}

By the proof of Theorem~\ref{cor:Hold}, we also have the following.
\begin{corollary}\label{cor:Hold1}
Let $k\geq 0$, $0\leq l<m$ and $\epsilon\in\{0,1\}$.
Suppose that $f\in C^{k\vee(n+1)}(\mathcal D)$.
Then
\begin{align}\label{eq:phiflcor}
\begin{split}
\varphi_{f,l}(s)&=-\sum_{\substack{0\leq j<k\\j= m-2\operatorname{mod} m
}}\frac{\mathscr{C}^j_{2l+\epsilon}(f)}{j!}|s|^{\frac{j-(m-2)}{m}}\log |s|\\
&\quad+\sum_{\substack{0\leq j<k\\j\neq m-2\operatorname{mod} m
}}\frac{{\mathscr{C}}^j_{2l+\epsilon}(f)}{j!}|s|^{\frac{j-(m-2)}{m}}+C^{n+\pb}((0,(-1)^\epsilon]).
\end{split}
\end{align}
\end{corollary}

\subsection{Basic properties of $\mathscr{C}^k_l$}
Recall that $\mathscr{C}^k_l:C^{k}(\mathcal{D})\to\C$ for $k\geq 0$ and $0\leq l<2m$  are given by
\begin{align*}
\mathscr{C}^k_{l}(f)=\sum_{\substack{0\leq j\leq k\wedge(m-2)\\j\neq k-(m-1)\operatorname{mod} m}}
\theta_0^{l(2j-k)}\mathfrak{B}(\tfrac{(m-1)-j}{m},\tfrac{(m-1)-(k-j)}{m})\partial_j^{k}(f).
\end{align*}
%
The functionals $\mathscr{C}^k_l$, $0\leq l<2m$ are not independent. By definition,
\begin{align}\label{eq:relCpm1}
\mathscr{C}^k_{l+m}=(-1)^k\mathscr{C}^k_l\text{ for any }0\leq l<m.
\end{align}
Moreover, we can also get back the value of $\partial ^{k}_j$ from $\mathscr{C}^k_l$. Indeed,
for every $0\leq j\leq k\wedge(m-2)$ with $j\neq k-(m-1)\operatorname{mod} m$,
\begin{gather*}
\mathfrak{B}(\tfrac{(m-1)-j}{m},\tfrac{(m-1)-(k+1)+j}{m})\partial_j^{k}
=\frac{1}{2m}\sum_{0\leq l<2m}\theta_0^{l(k-2j)}\mathscr{C}^k_l=\frac{1}{m}\sum_{0\leq l<m}\theta_0^{l(k-2j)}\mathscr{C}^k_l.
\end{gather*}
Similarly, if $k\wedge(m-2)<j\leq m-2$ or $j=m-1$ or $j= k-(m-1)\operatorname{mod} m$, then
\begin{gather*}
\sum_{0\leq l<2m}\theta_0^{l(k-2j)}\mathscr{C}^k_l=2\sum_{0\leq l<m}\theta_0^{l(k-2j)}\mathscr{C}^k_l=0.
\end{gather*}
Together with \eqref{eq:relCpm1}  this gives all  linear relations involving the functionals  $\mathscr{C}^k_l$.

Moreover, using \eqref{def:parl}, we obtain an elegant formula for $\mathscr{C}^k_l$ depending on the partial derivatives of the function $f$.
Indeed, if $0\leq j\leq m-2$, $j\neq k-(m-1)\operatorname{mod}m$ and $0\leq n\leq\frac{j-i}{m}$ then, by \eqref{eq:beta1},
\begin{align*}
 \mathfrak{B}&(\tfrac{(m-1)-j}{m},\tfrac{(m-1)-(k-j)}{m})\frac{\binom{\frac{(m-1)-j}{m}-1}{n}}{\binom{\frac{(k-j)-(m-1)}{m}}{n}}
 \\
&=\mathfrak{B}(\tfrac{(m-1)-j-nm}{m}+n,\tfrac{(m-1)-(k-j)+nm}{m}-n)(-1)^{n}
\frac{\binom{-\frac{(m-1)-j-nm}{m}}{n}}{\binom{-\frac{(m-1)-(k-j)+nm}{m}+n}{n}}\\
&=\mathfrak{B}(\tfrac{(m-1)-j-nm}{m},\tfrac{(m-1)-(k-j)+nm}{m}).
\end{align*}
By the definition of $\partial^k_j$, it follows that
\begin{align*}
\mathscr{C}^k_l(f)&=\sum_{\substack{0\leq j\leq k\wedge(m-2)\\j\neq k-(m-2)\operatorname{mod}m}}
\Big(\theta_0^{l(2j-k)}\mathfrak{B}(\tfrac{(m-1)-j}{m},\tfrac{(m-1)-(k-j)}{m})\\
&\qquad
\sum_{0\leq n\leq \frac{k-j}{m}}
\frac{\binom{k}{j+nm}\binom{\frac{(m-1)-j}{m}-1}{n}}{\binom{\frac{(k-j)-(m-1)}{m}}{n}}\frac{\partial^kf}{\partial \omega^{j+nm} \partial \overline{\omega}^{k-j-nm}}(0,0) \Big)\\
&=\sum_{\substack{0\leq i\leq k\\i\neq m-1\operatorname{mod} m\\i\neq k-(m-1)\operatorname{mod} m}}
\theta_0^{l(2i-k)}\binom{k}{i}\mathfrak{B}(\tfrac{(m-1)-i}{m},\tfrac{(m-1)-(k-i)}{m})\frac{\partial^{k}f}{\partial \omega^i\partial\overline{\omega}^{k-i}}(0,0).
\end{align*}
According to \eqref{def:Bint},
\begin{align}\label{def2:C}
\mathscr{C}^k_l(f)=\sum_{0\leq i\leq k}
\theta_0^{l(2i-k)}\binom{k}{i}\mathfrak{B}(\tfrac{(m-1)-i}{m},\tfrac{(m-1)-(k-i)}{m})\frac{\partial^{k}f}{\partial \omega^i\partial\overline{\omega}^{k-i}}(0,0).
\end{align}

%
%
\begin{remark}
This formula generalizes the one for $C_\alpha^{\pm}(\varphi_{f,l}), \alpha \in \mathcal{A}$ in \cite[Theorem 9.1]{Fr-Ki} by replacing new functionals for higher order derivatives.
\end{remark}


We now strengthen Theorem~\ref{thm:Cketa} by proving that $F_f$ is also smooth (with some drop of regularity) on the closed sectors $\overline{\mathcal{D}}(\frac{l}{2m}, \frac{l+1}{2m})$.

\begin{theorem}\label{thm;ext}
Fix $k\geq m-1$ and $0\leq l<2m$. Let $m-1\leq \underline{k}\leq k$ be the natural number given by $\widehat{\mathfrak{o}}(\sigma,\underline{k})=\underline{k}-(m-2)=\lceil\frac{k-(m-2)}{m}\rceil=\lceil\mathfrak{o}(\sigma,k)\rceil=:n$.
Suppose that $f\in C^{k\vee(n+1)}(\mathcal{D})$ is  such that $\partial^j_i(f)=0$ for all $0\leq j<\underline{k}$ and $0\leq i\leq j\wedge(m-2)$ with $i\neq j-(m-1)\operatorname{mod} m$
and $\mathscr{C}^j_{l}(f)=0$ for all $0\leq j< k$. Then the map $F_f:\mathcal{D}(\frac{l}{2m}, \frac{l+1}{2m})\to\C$ has a
$C^{\mathfrak{e}(\sigma,k)}$-extension on $\overline{\mathcal{D}}(\frac{l}{2m}, \frac{l+1}{2m})$ and there exists $C>0$ such that
$\|F_f\|_{C^{\mathfrak{e}(\sigma,k)}(\overline{\mathcal{D}}(\frac{l}{2m}, \frac{l+1}{2m}))}\leq C\|f\|_{C^{k\vee (n+1)}(\mathcal{D})}$.
\end{theorem}
\begin{proof}
We focus only on the even sectors ${\mathcal{D}}(\frac{2l}{2m}, \frac{2l+1}{2m})$. The proof in the odd case proceeds in the same way.
By Theorem~\ref{thm:Cketa}, for every $0<\vep<1/2$ the map $F_f$ has a $C^{\widehat{\mathfrak{e}}(\sigma,\underline{k})}$-extension on
$\overline {\mathcal D}(\frac{2l+\vep}{2m}, \frac{2l+1}{2m})\subset \overline {\mathcal D}(\frac{2l+\vep}{2m}, \frac{2l+2-\vep}{2m})$ and there exists $C_\vep>0$ so that
$\|F_f\|_{C^{\widehat{\mathfrak{e}}(\sigma,\underline{k})}(\overline {\mathcal D}(\frac{2l+\vep}{2m}, \frac{2l+1}{2m}))}\leq C_\vep\|f\|_{C^{\underline{k}}(\mathcal{D})}$.
Moreover,
\begin{align*}
\mathscr{F}_{f,l}(u,s)&=\int_{-1}^u\frac{f(G_l(v,s))}{(v^2+s^2)^{\frac{m-1}{m}}}dv=\varphi_{f,l}(s)-\int_{u}^{1}\frac{f(G_l(v,s))}{(v^2+s^2)^{\frac{m-1}{m}}}dv\\
&=\varphi_{f,l}(s)-\int_{-1}^{-u}\frac{f(G_l(-v,s))}{(v^2+s^2)^{\frac{m-1}{m}}}dv.
\end{align*}
As $G_l(-v,s)=\theta_0^{-1}G_l(v,-s)$ for $s>0$, this gives
\[\mathscr{F}_{f,l}(z,\overline{z})=\varphi_{f,l}(\Im z)-\mathscr{F}_{f\circ \theta_0^{-1},l}(-z,-\overline{z})\text{ if }\Im z>0.\]
It follows that
\begin{equation}\label{eq:Ffsym}
{F}_f(\omega,\overline{\omega})=
\varphi_{f,l}(\Im \omega^m)-{F}_{f\circ \theta_0^{- 1}}(\theta_0\omega,\theta_0^{-1}\overline{\omega})\text{ on } \mathcal{D}(\tfrac{2l}{2m}, \tfrac{2l+1}{2m}).
\end{equation}
Note that
\[\frac{\partial^{j}(f\circ \theta_0^{-1})}{\partial \omega^i\partial\overline{\omega}^{j-i}}(0,0)=\theta_0^{-(2i-j)}\frac{\partial^{j}f}{\partial \omega^i\partial\overline{\omega}^{j-i}}(0,0).\]
By \eqref{def:parl}, it follows that $\partial_{i}^j(f\circ \theta_0^{-1})=\theta_0^{-(2i-j)}\partial_{i}^j(f)$. Therefore, by assumption, $\partial^j_i(f\circ \theta_0^{-1})=0$ for all $0\leq j<\underline{k}$ and $0\leq i\leq j\wedge(m-2)$ with $i\neq j-(m-1)\operatorname{mod} m$. Using Theorem~\ref{thm:Cketa} again, we obtain the map $F_{f\circ \theta_0^{-1}}$ has a $C^{\widehat{\mathfrak{e}}(\sigma,\underline{k})}$-extension on
$\overline {\mathcal D}(\frac{2l+1}{2m}, \frac{2l+2-\vep}{2m})\subset \overline {\mathcal D}(\frac{2l+\vep}{2m}, \frac{2l+2-\vep}{2m})$ and
$\|F_{f\circ \theta_0^{-1}}\|_{C^{\widehat{\mathfrak{e}}(\sigma,\underline{k})}(\overline {\mathcal D}(\frac{2l+1}{2m}, \frac{2l+2-\vep}{2m}))}\leq C_\vep\|f\circ \theta_0^{-1}\|_{C^{\underline{k}}(\mathcal{D})}$.
In particular,
\begin{align}\label{eq:Ffsym1}
\begin{split}
&\text{${F}_{f\circ \theta_0^{-1}}(\theta_0\omega,\theta_0^{-1}\overline{\omega})$ is of the class  $C^{\widehat{\mathfrak{e}}(\sigma,\underline{k})}$ on $\overline {\mathcal D}(\tfrac{2l}{2m}, \tfrac{2l+1-\vep}{2m})$ and}\\
&\|F_{f\circ \theta_0^{-1}}(\theta_0\omega,\theta_0^{-1}\overline{\omega})\|_{C^{\widehat{\mathfrak{e}}(\sigma,\underline{k})}(\overline {\mathcal D}(\frac{2l}{2m}, \frac{2l+1-\vep}{2m}))}\leq C_{\vep}\|f\|_{C^{\underline{k}}(\mathcal{D})}.
\end{split}
\end{align}
By Theorem~\ref{cor:Hold}, $\varphi_{f,l}$ has a $C^{\mathfrak{e}(\sigma,k)}$-extension on $[0,1]$ with
\[
\|\varphi_{f,l}\|_{C^{\mathfrak{e}(\sigma,k)}([0,1])}\leq C\|f\|_{C^{k\vee(n+1)}(\mathcal D)}.
\]
Therefore, $\omega\to \varphi_{f,l}(\Im \omega^m)$ has a $C^{\mathfrak{e}(\sigma,k)}$-extension on $\overline {\mathcal D}(\tfrac{2l}{2m}, \tfrac{2l+1}{2m})$ and
\begin{gather*}
\|\varphi_{f,l}(\Im \omega^m)\|_{C^{\mathfrak{e}(\sigma,k)}(\overline {\mathcal D}(\frac{2l}{2m}, \frac{2l+1}{2m}))}\leq C'\|f\|_{C^{k\vee(n+1)}(\mathcal D)}.
\end{gather*}
As $F_f$ is a $C^{\widehat{\mathfrak{e}}(\sigma,\underline{k})}$-map on
$\overline {\mathcal D}(\frac{2l+\vep}{2m}, \frac{2l+1}{2m})$ with
$\|F_{f}\|_{C^{\widehat{\mathfrak{e}}(\sigma,\underline{k})}(\overline {\mathcal D}(\frac{2l+\vep}{2m}, \frac{2l+1}{2m}))}\leq C_{\vep}\|f\|_{C^{\underline{k}}(\mathcal{D})}$,
$\mathfrak{o}(\sigma,k)\leq \lceil\mathfrak{o}(\sigma,k)\rceil= \widehat{\mathfrak{o}}(\sigma,\underline{k})$ and $\underline{k}\leq k$,
in view of \eqref{eq:Ffsym} and \eqref{eq:Ffsym1}, this gives our claim.
\end{proof}

We now show that {the regularity of $F_f$ obtained in} Theorem~\ref{thm;ext} is optimal.

\begin{theorem}\label{thm;extinv}
Fix $k\geq m-1$ and $0\leq l<2m$.
If $f\in C^{k\vee(n+1)}(\mathcal D)$ is  such that $F_{f}\in C^{r}(\overline {\mathcal D}(\frac{l}{2m}, \frac{l+1}{2m}))$ for some $r\in \R_\eta$ with $0<v(r)\leq \mathfrak{o}(\sigma,k)$ then
$\mathscr{C}^j_{l}(f)=0$ for all $j\geq 0$ such that $\mathfrak{o}(\sigma,j)<v(r)$ and $\partial^j_i(f)=0$ for all $j\geq 0$ with $\widehat{\mathfrak{o}}(\sigma,j)<v(r)$ and $0\leq i\leq j\wedge(m-2)$ with $i\neq j-(m-1)\operatorname{mod} m$.
\end{theorem}

\begin{proof}
We will focus only on the even sectors ${\mathcal{D}}(\frac{2l}{2m}, \frac{2l+1}{2m})$. The proof in the odd case proceeds in the same way.

By definition, $\varphi_{f,l}(s)=\mathscr{F}_{f,l}(1,s)=F_f(G_l(1+\iota s),\overline{G_l(1+\iota s)})$ on $(0,1]$. As $F_{f}\in C^{r}(\overline {\mathcal D}(\frac{2l}{2m}, \frac{2l+1}{2m}))$, it follows that
$\varphi_{f,l}\in C^r((0,1])$. In view of Theorem~\ref{cor:Hold}, $\mathscr{C}^j_{2l}(f)=0$ for all $j\geq 0$ such that $\mathfrak{o}(\sigma,j)<v(r)$.

The proof of the vanishing of $\partial^j_i$ is much more involved. Choose $m-1\leq\underline{k}\leq \overline{k}< k$ such that
$\mathfrak{o}(\sigma,\overline{k}-1)<v(r)\leq \mathfrak{o}(\sigma,\overline{k})$ and $\widehat{\mathfrak{o}}(\sigma,\underline{k}-1)<v(r)\leq \widehat{\mathfrak{o}}(\sigma,\underline{k})$.
By the first part of the theorem, $\mathscr{C}^j_{2l}(f)=0$ for all $0\leq j<\overline{k}$.
Let us decompose $f=f_{< \underline{k}}+e_f$, where $f_{< \underline{k}}=\sum_{0\leq j< \underline{k}}f_{j}$ with
\begin{gather*}
f_{j}(\omega,\overline \omega)=\frac{1}{j!}\sum_{0\leq i\leq j}\binom{j}{i}\frac{\partial^j f}{\partial\omega^i\partial\overline{\omega}^{j-i}}(0,0)\omega^i\overline{\omega}^{j-i}.
\end{gather*}
Then for every $0\leq j< \underline{k}$ we have $D^je_f(0,0)=0$ and $\mathscr{C}^j_{2l}(e_f)=0$ and for $\underline{k}\leq j<\overline{k}$ we have $\mathscr{C}^j_{2l}(e_f)=\mathscr{C}^j_{2l}(f)=0$.
Since $\widehat{\mathfrak{o}}(\sigma,\underline{k})=\lceil\mathfrak{o}(\sigma,\overline{k})\rceil$,
in view of Theorem~\ref{thm;ext}, this gives $F_{e_f}\in C^{\mathfrak{e}(\sigma, \overline{k})}(\overline{\mathcal{D}}(\frac{2l}{2m}, \frac{2l+1}{2m}))$. As $F_{f}\in C^{r}(\overline{\mathcal{D}}(\frac{2l}{2m}, \frac{2l+1}{2m}))$ and $v(r)\leq \mathfrak{o}(\sigma,\overline{k})$, this yields $F_{f_{< \underline{k}}}=F_{f}-F_{e_f}\in C^{r}(\overline{\mathcal{D}}(\frac{2l}{2m}, \frac{2l+1}{2m}))$.

For every $0<a<1$ let $\Delta_a=\{(u,s):0< u\leq 1,0<s\leq au\}$.
By Lemmas~\ref{lem:Bs^}, \ref{lem:Blog}, \ref{lem:Bmix} and \eqref{def2:C},
for every $0<a<1$, there exist ${\varrho_{j,l}}\in C^{\omega}([0,a])$ and ${c_{j,l}}\in \C$ for $0\leq j<\underline{k}$ and ${\rho_l}\in C^{\omega}([0,a])$ such that for any $(u,s)\in\Delta_a$,
\begin{align*}
\mathscr{F}_{f_{<\underline{k}},l}(u,s)&=\sum_{0\leq j<\underline{k}}\frac{1}{j}\sum_{0\leq i\leq j}\binom{j}{i}
\frac{\partial^j f(0,0)}{\partial\omega^i\partial\overline{\omega}^{j-i}}\mathfrak{G}^l_{(m-1)-i,(m-1)-(j-i)}(u,s)\\
&=\sum_{0\leq j<\underline{k}}u^{\frac{j-(m-2)}{m}}{\varrho_{j,l}}(s/u)+\log u\sum_{0\leq j<\underline{k}}{c_{j,l}}s^{\frac{j-(m-2)}{m}}+{\rho_l}(s).
\end{align*}
Let $\alpha\in(0,1/4)$ so that $\tan(\pi \alpha)=a$. Fix any $0<\beta<\alpha$ and let $\omega_0=e^{2\pi \iota \frac{2l+\beta}{2m}}$. Then for any $t\in(0,a]$,
\begin{align*}
\mathscr{F}_{f_{<\underline{k}},l}((t\omega_0)^m,\overline{t\omega_0}^m)&=\mathscr{F}_{f_{<\underline{k}},l}(t^m\cos(\pi\beta),t^m\sin(\pi\beta))\\
&=
\sum_{0\leq j<\underline{k}}\cos(\pi\beta)^{\frac{j-(m-2)}{m}}{\varrho_{j,l}}(\tan(\pi\beta))t^{j-(m-2)}+{\rho_l}(t^m\sin(\pi\beta))\\
&\quad+\sum_{0\leq j<\underline{k}}{c_{j,l}}\sin(\pi\beta)^{\frac{j-(m-2)}{m}}t^{j-(m-2)}\log (t^m\cos(\pi\beta)).
\end{align*}
Since $(0,a]\ni t\mapsto \mathscr{F}_{f_{<\underline{k}},l}((t\omega_0)^m,\overline{t\omega_0}^m)\in\C$ is of class $C^r$, $[0,a]\ni t\mapsto {\rho_l}(t^m\sin(\pi\beta))\in\C$ is analytic and $v(r)>\widehat{\mathfrak o}(\sigma,\underline{k}-1)=\underline{k}-1-(m-2)\geq j-(m-2)$ for every $0\leq j<\underline{k}$,
it follows that ${c_{j,l}}=0$ for all $0\leq j<\underline{k}$, so
$\mathscr{F}_{f_{<\underline{k}},l}(u,s)=\sum_{0\leq j<\underline{k}}u^{\frac{j-(m-2)}{m}}{\varrho_{j,l}}(s/u)+{\rho_l}(s)$.

For every $0\leq j<\underline{k}$ let ${\Upsilon_{j,l}}:\Delta_a\to\C$ be a real analytic homogenous map of degree  $\frac{j-(m-2)}{m}$ given by ${\Upsilon_{j,l}}(u,s)=u^{\frac{j-(m-2)}{m}}{\varrho_{j,l}}(s/u)$. Then
\[\mathscr{F}_{f_{<\underline{k}},l}(z,\overline{z})=\sum_{0\leq j<\underline{k}}{\Upsilon_{j,l}}(z,\overline{z})+{\rho_l}(\Im z)\text{ on }\Delta_a\]
and
\[{F}_{f_{<\underline{k}},l}(\omega,\overline{\omega})=\mathscr{F}_{f_{<\underline{k}},l}(\omega^m,\overline{\omega}^m)=\sum_{0\leq j<\underline{k}}{\Upsilon_{j,l}}(\omega^m,\overline{\omega}^m)+{\rho_l}(\Im\omega^m)\text{ on }{\mathcal{D}}(\tfrac{2l}{2m}, \tfrac{2l+\alpha}{2m}).\]
Since $F_{f_{< \underline{k}}}\in C^{r}(\overline{\mathcal{D}}(\frac{2l}{2m}, \frac{2l+1}{2m}))$  and ${\rho_l}(\Im\omega^m)\in C^{\omega}(\overline{\mathcal{D}}(\frac{2l}{2m}, \frac{2l+\alpha}{2m}))$, we have
\[\sum_{0\leq j<\underline{k}}{\Upsilon_{j,l}}(\omega^m,\overline{\omega}^m)\in C^{r}(\overline{\mathcal{D}}(\tfrac{2l}{2m}, \tfrac{2l+\alpha}{2m}))\]
and ${\Upsilon_{j,l}}(\omega^m,\overline{\omega}^m)$ is a homogenous map of degree  $j-(m-2)<v(r)$ for $0\leq j<\underline{k}$. Then standard arguments for smooth homogenous maps {(see e.g.\ Lemma~\ref{lem:hom} in Appendix)} show that ${\Upsilon_{j,l}}=0$ for $0\leq j<m-2$
and ${\Upsilon_{j,l}}(\omega^m,\overline{\omega}^m)$ is a homogenous polynomial of degree  $j-(m-2)$ for $m-2\leq j<\underline{k}$. Suppose that
\[{\Upsilon_{j,l}}(\omega^m,\overline{\omega}^m)=\sum_{0\leq i\leq j-(m-2)}a_{j,i}\omega^i\overline{\omega}^{(j-i)-(m-2)}\text{ for }m-2\leq j<\underline{k}.\]
Then
\begin{gather*}
\sum_{0\leq j<\underline{k}}\frac{1}{j!}\sum_{0\leq i\leq j}\binom{j}{i}\frac{\partial^j f(0,0)}{\partial\omega^i\partial\overline{\omega}^{j-i}}\mathfrak{G}^l_{(m-1)-i,(m-1)-(j-i)}(u,s)=\mathscr{F}_{f_{<\underline{k}},l}(u,s)\\=\!\!\sum_{0\leq j<\underline{k}}\!\!{\Upsilon_{j,l}}(u,s)\!+\!{\rho_l}(s)=\!\!\sum_{m-2\leq j<\underline{k}}\sum_{0\leq i\leq j-(m-2)}\!\!{a_{j,i}}G_l(u,s)^i\overline{G_l(u,s)}^{(j-i)-(m-2)}\!+\!{\rho_l}(s).
\end{gather*}
Differentiating with respect $u$, we get
\begin{align*}
\sum_{0\leq j<\underline{k}}&\frac{1}{j!}\sum_{0\leq i\leq j}\binom{j}{i}\frac{\partial^j f(0,0)}{\partial\omega^i\partial\overline{\omega}^{j-i}}G_l^{i-(m-1)}\overline{G_l}^{(j-i)-(m-1)}\\
&=
\sum_{m-2\leq j<\underline{k}}\sum_{0\leq i\leq j-(m-2)}{a_{j,i}}\big(\tfrac{i}{m}G_l^{i-m}\overline{G_l}^{(j-i)-(m-2)}+\tfrac{(j-i)-(m-2)}{m}G_l^{i}\overline{G_l}^{(j-i)-2(m-1)}\big)\\
&=
\sum_{m-1\leq j<\underline{k}}\Big(\sum_{0\leq i\leq j-(m-1)}{a_{j,i+1}}\tfrac{i+1}{m}G_l^{i-(m-1)}\overline{G_l}^{(j-i)-(m-1)}\\
&\qquad\qquad+
\sum_{m-1\leq i\leq j}{a_{j,i-(m-1)}}\tfrac{(j-i)+1}{m}G_l^{i-(m-1)}\overline{G_l}^{(j-i)-(m-1)}\Big).
\end{align*}
It follows that $D^jf(0,0)=0$ for $0\leq j\leq m-2$ and for every $m-1\leq j<\underline{k}$ and $0\leq i\leq j$,
\[\frac{1}{j!}\binom{j}{i}\frac{\partial^j f(0,0)}{\partial\omega^i\partial\overline{\omega}^{j-i}}=
{a_{j,i+1}}\tfrac{i+1}{m}+{a_{j,i-(m-1)}}\tfrac{(j-i)+1}{m},
\]
here we adhere to the convention that ${a_{j,i}}=0$ if $i<0$ or $i>j-(m-2)$.
It follows that for any $m-1\leq j<\underline{k}$ and $0\leq i\leq m-2$ with $i\neq j-(m-1)\operatorname{mod}m$,
\begin{align*}
\frac{\partial^j_{i}(f)}{j!}&=\sum_{0\leq n\leq \frac{j-i}{m}}\frac{\binom{\frac{(m-1)-i}{m}-1}{n}}{\binom{\frac{(j-i)-(m-1)}{m}}{n}}\frac{1}{j!}\binom{j}{mn+i}\frac{\partial^jf}{\partial \omega^{mn+i} \partial \overline{\omega}^{j-(mn+i)}}(0,0)\\
&=\sum_{n\geq 0}\frac{\binom{\frac{i-(m-1)}{m}+n}{n}}{\binom{-\frac{(j-i)-(m-1)}{m}+(n-1)}{n}}\Big(a_{j,i-(m-1)+m(n+1)}(\tfrac{i-(m-1)}{m}+n+1)\\
&\qquad\qquad+
a_{j,i-(m-1)+mn}(\tfrac{(j-i)-(m-1)}{m}-(n-1))\Big)\\
&=\sum_{n\geq 0}\frac{\binom{\frac{i-(m-1)}{m}+n+1}{n+1}(n+1)}{\binom{-\frac{(j-i)-(m-1)}{m}+(n-1)}{n}}a_{j,i-(m-1)+m(n+1)}\\
&\quad -\sum_{n\geq 1}\frac{\binom{\frac{i-(m-1)}{m}+n}{n}n}{\binom{-\frac{(j-i)-(m-1)}{m}+(n-2)}{n-1}}{a_{j,i-(m-1)+mn}}-{a_{j,i-(m-1)}}(\tfrac{(j-i)-(m-1)}{m}+1).
\end{align*}
As $i-(m-1)<0$, we have ${a_{j,i-(m-1)}}=0$. Hence $\partial^j_{i}(f)=0$ for every $0\leq j<\underline{k}$ and $0\leq i\leq j\wedge(m-2)$ with  $i\neq j-(m-1)\operatorname{mod}m$.
\end{proof}

\section{Global properties}\label{sec;GP}
In this section,
by combining previous results for local analysis near singularity, we finally obtain solutions for cohomological equations with optimal loss of regularity.

\subsection{Transition from local to global results}\label{sec;GP1}
Let $M$ be a compact connected orientable $C^\infty$-surface. Let $\psi_\R$ be a locally Hamiltonian $C^\infty$-flow on $M$ with isolated fixed points such that all its saddles are perfect and all saddle connections are loops. Let $M'\subset M$ be a minimal component of the flow and let $I\subset M'$ be a transversal curve. The corresponding IET $T:I\to I$ exchanges the intervals $\{I_\alpha:\alpha \in \mathcal{A}\}$.
There exists $0<\vep\leq 1$ such that for every $\sigma\in \mathrm{Sd}(\psi_\R)$ we have $\mathcal{D}_{\sigma,\vep}\subset U_\sigma$, where $\mathcal{D}_{\sigma,\vep}$ is the pre-image of the square  $[-\vep,\vep]\times[-\vep,\vep]$ via the map $z\mapsto z^{m_\sigma}$ in local singular coordinates, {see Figure~\ref{fig:Dep}}. Moreover, we can assume that every orbit starting from $I$ meets at most one set $\mathcal{D}_{\sigma,\vep}$ (maybe many times) before return to $I$.
For every $0\leq l<2m_\sigma$ let $\mathcal{D}^l_{\sigma,\vep}=\overline{\mathcal{D}}_{\sigma,\vep}(\frac{l}{2m_\sigma}, \frac{l+1}{2m_\sigma})$ be the $l$-th closed angular sector of $\mathcal{D}_{\sigma,\vep}$.

\begin{figure}[h!]
 \includegraphics[width=0.4\textwidth]{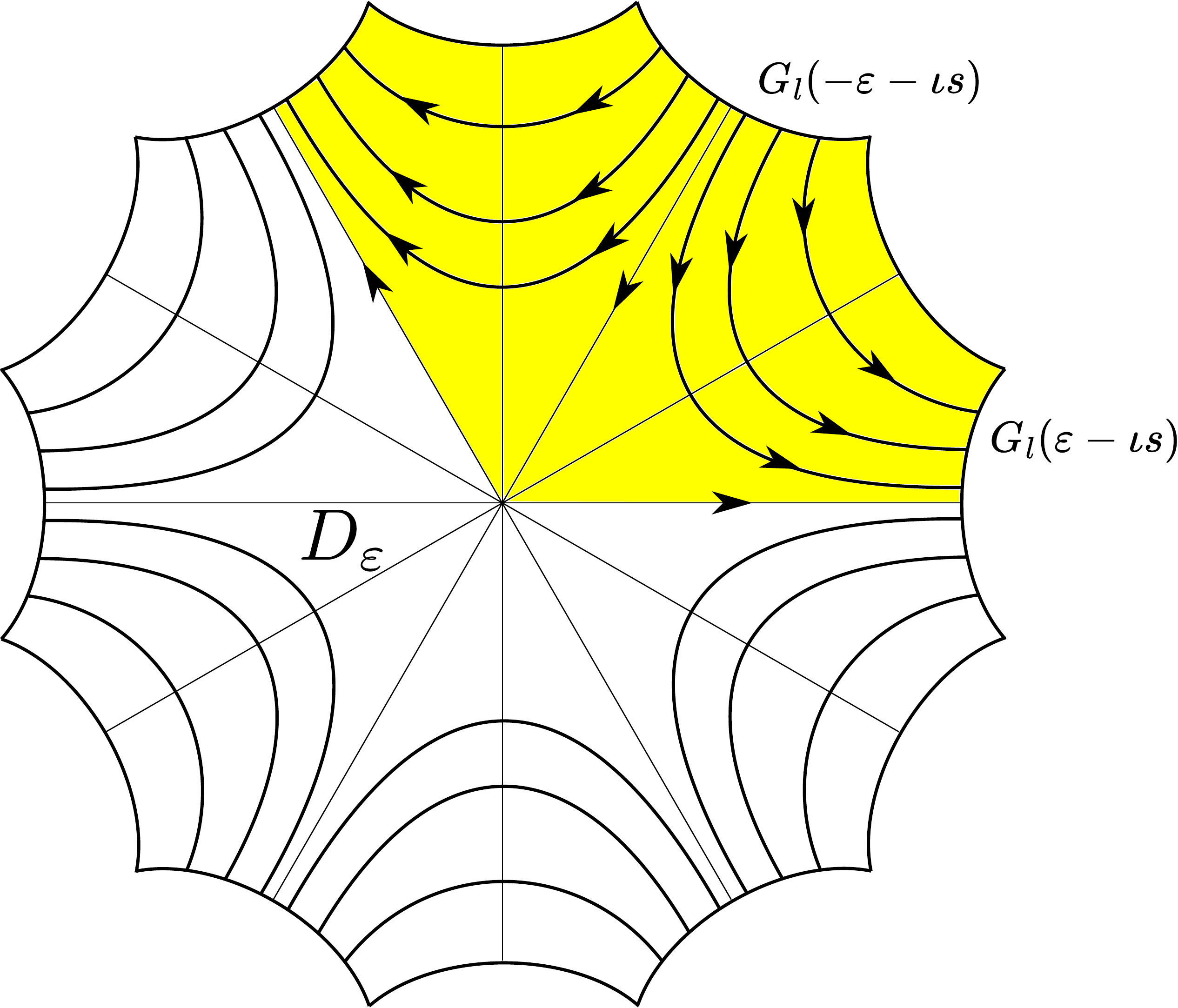}
 \caption{{The set $D_\varepsilon$ and the incoming and outgoing segments of its boundary in the case of $m_\sigma = 3$.} \label{fig:Dep}
 }
\end{figure}

\begin{remark}\label{rmk:Feps}
By Lemma~8.2 in \cite{Fr-Ki}, the enter and exit sets of ${\mathcal{D}}_{\sigma,\vep}(\frac{l}{m_\sigma}, \frac{l+1}{m_\sigma})$ are $C^\infty$-curves with standard parametrization
\[[-\vep,\vep]\ni s\mapsto G_l(-\vep-\iota s)\in {\mathcal{D}}_{\sigma,\vep}\text{ and }[-\vep,0)\cup(0,\vep]\ni s\mapsto G_l(\vep-\iota s)\in {\mathcal{D}}_{\sigma,\vep}\text{  resp.}\]
Every $\omega\in {\mathcal{D}}_{\sigma,\vep}(\frac{l}{m_\sigma}, \frac{l+1}{m_\sigma})$ lies on the positive semi-orbit of  $G_l(-\vep-\iota s)$, $s\in[-\vep,\vep]$ so that $\psi_{\xi_l(\omega)}G_l(-\vep-\iota s)=\omega$ for some $\xi_l(\omega)>0$.
By the proof of Lemma~8.2 in \cite{Fr-Ki}, $z=\omega^{m_\sigma}=u-\iota s$ for some $u\in[-\vep,\vep]$ and for any $f\in C(M)$,
\begin{align}\label{eq:intpass}
\begin{split}
\int_0^{\xi_l(\omega)}&f(\psi_tG_l(-\vep-\iota s))dt=\frac{1}{m_\sigma^2}\int_{-\vep}^{u}\frac{(f\cdot V)(G_l(v-\iota s))}{(v^2+s^2)^{\frac{m_\sigma-1}{m_\sigma}}}dv\\
&=\frac{\vep^{-\frac{m_\sigma-2}{m_\sigma}}}{m_\sigma^2}\int_{-1}^{u/\vep}\frac{(f\cdot V)(\vep^{\frac{1}{m_\sigma}}G_l(v-\iota (s/\vep)))}{(v^2+(s/\vep)^2)^{\frac{m_\sigma-1}{m_\sigma}}}dv\\
&=\frac{\vep^{-\frac{m_\sigma-2}{m_\sigma}}}{m_\sigma^2}\mathscr{F}_{(f\cdot V)\circ \vep^{1/m_{\sigma}},l}(z/\vep)
=\frac{\vep^{-\frac{m_\sigma-2}{m_\sigma}}}{m_\sigma^2}{F}_{(f\cdot V)\circ \vep^{1/m_{\sigma}}}(\vep^{-1/m_\sigma}\omega).
\end{split}
\end{align}
In particular, if $\omega=G_l(\vep-\iota s)$ for $s\in [-\vep,\vep]\setminus\{0\}$ then $\tau_l(s):=\xi_l(G_l(-\vep-\iota s))$ is the transit time of $G_l(-\vep-\iota s)$ through the set ${\mathcal{D}}_{\sigma,\vep}(\frac{l}{m_\sigma}, \frac{l+1}{m_\sigma})$, $u-\iota s=G_l(\vep-\iota s)^{m_\sigma}=\vep-\iota s$ and
\begin{align}\label{eq:intpass1}
\begin{split}
\int_0^{\tau_l(s)}f(\psi_tG_l(-\vep-\iota s))dt&=\frac{\vep^{-\frac{m_\sigma-2}{m_\sigma}}}{m_\sigma^2}\mathscr{F}_{(f\cdot V)\circ \vep^{1/m_{\sigma}},l}((\vep-\iota s)/\vep)\\
&=\frac{\vep^{-\frac{m_\sigma-2}{m_\sigma}}}{m_\sigma^2}\varphi_{(f\cdot V)\circ \vep^{1/m_{\sigma}},l}(-s/\vep).
\end{split}
\end{align}
\end{remark}

\begin{remark}\label{rmk:rkr}
Recall that $m\geq 2$ is the maximal multiplicity of saddles in $\mathrm{Sd}(\psi_\R)\cap M'$ {and for any $r\geq -\frac{m-2}{m}$,
\[
k_r=\left\{
\begin{array}{cl}
\lceil mr+(m-1)\rceil &\text{if }-\frac{m-2}{m}\leq r\leq -\frac{m-3}{m},\\
\lceil mr+(m-2)\rceil &\text{if }-\frac{m-3}{m}< r.
\end{array}
\right.
\]
Note that} for any
$r\geq -\frac{m-2}{m}$ we have $\lceil r\rceil+1\leq k_r$.
 Indeed, if $-\frac{m-2}{m}\leq r\leq -\frac{m-3}{m}$ then  $-\frac{m-2}{m-1}\leq r$. Hence $r+1\leq mr+(m-1)$, which yields $\lceil r\rceil+1\leq \lceil mr+(m-1)\rceil= k_r$.

 If $-\frac{m-3}{m}< r$ with $m\geq 3$ or $1\leq r$ with $m=2$ then  $-\frac{m-3}{m-1}\leq r$. Hence $r+1\leq mr+(m-2)$, which yields $\lceil r\rceil+1\leq \lceil mr+(m-2)\rceil= k_r$.
  Suppose that $m=2$ and $\frac{1}{2}=-\frac{m-3}{m}<r<1$. Then  $\lceil r\rceil+1=2\leq \lceil 2r\rceil=\lceil mr+(m-2)\rceil= k_r$.
\end{remark}
\begin{remark}\label{rmk:rkr1}
For any $r\geq -\frac{m-2}{m}$ and $\sigma\in \mathrm{Sd}(\psi_\R)\cap M'$ let $k\geq 0$ be such that $\mathfrak{o}(\sigma,k-1)<r\leq \mathfrak{o}(\sigma,k)$. It follows that $n:=\lceil \mathfrak{o}(\sigma,k)\rceil=\lceil r\rceil$.
 In view of \eqref{eq:expkr}, $k\leq \lceil mr+(m-2)\rceil\leq k_r$. Moreover, by Remark~\ref{rmk:rkr}, $n+1=\lceil r\rceil+1\leq k_r$. Therefore, $k\vee(n+1)\leq k_r$.
\end{remark}

\begin{proof}[Proof of Theorem~\ref{thm1}]
Let $\tau:I\to\R_{>0}\cup\{+\infty\}$ be the first return time map for the flow $\psi_\R$ restricted to $M'$. For any interval (set) $J\subset I$
avoiding the set $disc(T)$ of discontinuities of $T$ let $J^\tau=\{\psi_t s:s\in J, 0\leq t\leq \tau(s)\}$.
If an interval $J$ contains some elements of $disc(T)$ then $J^\tau$ is the closure of $(J\setminus disc(T))^\tau$.
\medskip

\textbf{Case 1.} Suppose that $J\subset I_\alpha$ is a closed interval such that $\sup\tau(J)<\infty$ and $\max\tau(J)<2\min\tau(J)$. Choose any $t_J<\min\tau(J)$ so that $2t_J>\max\tau(J)$. Let us consider the set
$J^\tau$ and its two subsets
\begin{equation}\label{eq:defJtau}
J_+^\tau=\{\psi_t s:s\in J, 0\leq t\leq t_J\},\quad J_-^\tau=\{\psi_{-t} (Ts):s\in J, 0\leq t\leq t_J\}.
\end{equation}
By assumption, $J_+^\tau\cup J_-^\tau=J^\tau$. Let $\rho_+,\rho_-:J^\tau\to[0,1]$ be the corresponding $C^\infty$-partition of unity, i.e.\ $\rho_\pm$ are $C^\infty$-maps such that $\rho_++\rho_-=1$ and $\rho_{\pm}=0$ on $J^\tau\setminus J_\pm^\tau$. Let $v_\pm:J\times [0,t_J]\to J^\tau_\pm$ be given by $v_+(s,t)=\psi_t s$ and $v_-(s,t)=\psi_{-t} (Ts)$. Then
\[\varphi_f(s)=\int_{0}^{t_J}(\rho_+\cdot f)\circ v_+(s,t)dt+\int_{0}^{t_J}(\rho_-\cdot f)\circ v_-(s,t)dt.\]
Since $v_\pm$ are of class $C^\infty$, it follows that for every $q>0$ if $f\in C^q(M)$ then $\varphi_f\in C^q(J)$ and there exists $C_J^q>0$ such that $\|\varphi_f\|_{C^q(J)}\leq C_J^q\|f\|_{C^q(M)}$ for any $f\in C^q(M)$. Suppose that $f\in C^{k_r}(M)$. In view of Remark~\ref{rmk:rkr1}, $n+1\leq k_r$, and hence $\varphi_f\in C^{n+1}(J)$ with
\begin{equation}\label{eq:estnpa0}
\|\varphi_f\|_{C^{n+1}(J)}\leq C_J^{n+1}\|f\|_{C^{k_r}(M)}\text{ for any }f\in C^{k_r}(M).
\end{equation}

\textbf{Case 2.} Suppose that $J\subset I_\alpha$ is of the form $J=[l_\alpha,l_\alpha+\vep]$. Suppose that $l_\alpha$ is the first backward meeting point of a separatrix incoming to $\sigma\in\mathrm{Sd}(\psi_\R)\cap M'$. It follows that
the orbits starting from $J$ meet the set $\mathcal{D}_{\sigma,\vep}$ before return to $I$. Suppose that each such orbit meets $\mathcal{D}_{\sigma,\vep}$ only once and it meets a sector $\mathcal{D}_{\sigma,\vep}^{2l+1}$ for some $0\leq l< m_\sigma$.
In general, the orbits of $J$ can meet $\mathcal{D}_{\sigma,\vep}$ several times in different sectors. This case arises when the saddle $\sigma$ has saddle loops, but this situation is discussed later.

For every $s\in J$ denote by $\tau_+(s)$ the first forward entrance time of the orbit of $s$ to $\mathcal{D}_{\sigma,\vep}$ and  by $\tau_-(s)$ the first backward entrance time of the orbit of $Ts$ to $\mathcal{D}_{\sigma,\vep}$.
Then $\psi_{\tau_+(s)}(s)=G_l(-\vep-\iota(s-l_\alpha))$.
Since $\tau(s)\to+\infty$ as $s\to l_\alpha$ and $\tau_\pm$ are bounded, decreasing $\vep$, if necessary, we can assume that $\min\tau(J)>\max\tau_\pm(J)$. Choose $\max\tau_\pm(J)<t_J<\min\tau(J)$ and let us consider two subsets $J_\pm^\tau\subset J^\tau$ given by \eqref{eq:defJtau}.
Then $J^\tau=J_+^\tau\cup \mathcal{D}_{\sigma,\vep}^{2l+1}\cup J_-^\tau$. Let us consider the corresponding $C^\infty$-partition of unity $\rho_+,\rho_\sigma, \rho_-:J^\tau\to[0,1]$, i.e.\
$\rho_+$, $\rho_\sigma$, $\rho_-$ are $C^\infty$-maps such that $\rho_++\rho_\sigma+\rho_-=1$, $\rho_{\pm}=0$ on $J^\tau\setminus J_\pm^\tau$ and $\rho_{\sigma}=0$ on $J^\tau\setminus \mathcal{D}_{\sigma,\vep}$. Then
\[\varphi_f(s)=\int_{0}^{t_J}(\rho_+\cdot f)\circ v_+(s,t)dt+\int_{0}^{t_J}(\rho_-\cdot f)\circ v_-(s,t)dt+\int_{\tau_+(s)}^{\tau(s)-\tau_-(s)}(\rho_\sigma\cdot f)(\psi_ts)dt.\]
Repeating the arguments used in Case 1, for any $q>0$ we get $C^q_J>0$ such that
\begin{equation}\label{eq:phipm}
\Big\|\int_{0}^{t_J}(\rho_+\cdot f)(\psi_t\,\cdot\,) dt+\int_{0}^{t_J}(\rho_-\cdot f)(\psi_{-t}(T\,\cdot\,) dt\Big\|_{C^q(J)}\leq C_J^q\|f\|_{C^q(M)}
\end{equation}
for any $f\in C^q(M)$.

Note that for every
$s\in(0,\vep]$,
\begin{align*}
\varphi^\sigma_f(l_\alpha+s):&=\int_{\tau_+(l_\alpha+s)}^{\tau(l_\alpha+s)-\tau_-(l_\alpha+s)}(\rho_\sigma\cdot f)(\psi_t(l_\alpha+s))dt\\
&=
\int_0^{\tau_l(s)}(\rho_\sigma\cdot f)(\psi_tG_l(-\vep-\iota s))dt.
\end{align*}
By \eqref{eq:intpass1}, it follows that  for any  $s\in(0,1]$,
$\varphi^\sigma_f(l_\alpha+\vep s)=\frac{\vep^{-\frac{m_\sigma-2}{m_\sigma}}}{m_\sigma^2}\varphi_{\tilde{f},l}(-s)$,
where $\tilde{f}(\omega,\overline{\omega})=(\rho_\sigma\cdot f\cdot V)(\vep^{\frac{1}{m_\sigma}}\omega,\vep^{\frac{1}{m_\sigma}}\overline{\omega})$.

\medskip

Suppose that  $f\in C^{k_r}(M)$ for some $r\geq -\frac{m-2}{m}$. Choose $k\geq 0$ such that $\mathfrak{o}(\sigma,k-1)<r\leq \mathfrak{o}(\sigma,k)$.  By Remark~\ref{rmk:rkr1}, we have  $\lceil \mathfrak{o}(\sigma,k)\rceil=\lceil r\rceil=n$ and $k\vee(n+1)\leq k_r$. Assume that $\mathfrak{C}^j_{\sigma,2l+1}(f)=0$ for all $0\leq j<k$, or equivalently for all $j\geq 0$ such that $\mathfrak{o}(\mathfrak{C}^j_{\sigma,2l+1})<r$.
Since $\rho_\sigma=1$ in a neighborhood of $\sigma$, it follows that
\[
\mathscr{C}^{j}_{2l+1}(\tilde f)=\vep^{\frac{j}{m_\sigma}}\mathscr{C}^{j}_{2l+1}(f\cdot V)=\vep^{\frac{j}{m_\sigma}}\mathfrak{C}^j_{\sigma,2l+1}(f)=0 \text{ for all } 0\leq j<k.
\]
Let $a_0:=\lceil\mathfrak{o}(\sigma,k)\rceil-\mathfrak{o}(\sigma,k)=n-\mathfrak{o}(\sigma,k)$.
Then  $n-a_0=\mathfrak{o}(\sigma,k)\geq r=n-a$.

As $k\vee(n+1)\leq k_r$ and both $f$ and $\tilde{f}$ are of class $C^{k_r}$,
in view of Theorem~\ref{cor:Hold}, $\varphi_{\tilde f,l}\in C^{n+\mathrm{P}_{a_0}}([-1,0))$ and there exists $C^r_{\sigma,l}>0$ such that
\[\|\varphi_{\tilde f,l}\|_{C^{n+\mathrm{P}_{a_0}}([-1,0))}\leq C^r_{\sigma,l}\|\tilde f\|_{C^{k\vee(n+1)}(\mathcal D)}\leq C^r_{\sigma,l}\|\rho_\sigma\cdot V\|_{C^{k_r}(\mathcal D)}\|f\|_{C^{k_r}(\mathcal D)}.\]
As $\varphi^\sigma_f(l_\alpha+\vep s)=\frac{\vep^{-\frac{m_\sigma-2}{m_\sigma}}}{m_\sigma^2}\varphi_{\tilde{f},l}(-s)$ and $n-a\leq n-a_0$, in view of Remark~\ref{rmk:filpa},
for any $f\in C^{k_r}(M)$ with $\mathfrak{C}^j_{\sigma,2l+1}(f)=0$ for all $0\leq j<k$,
\[\varphi^\sigma_f\in C^{n+\pa}(J)\text{ and }\|\varphi^\sigma_f\|_{C^{n+\pa}(J)}\leq \widetilde{C}^r_{\sigma,l}\|f\|_{C^{k_r}(\mathcal D)}.\]
In view of \eqref{eq:phipm} and  Remark~\ref{rmk:filpa}, it follows that for any $f\in C^{k_r}(M)\cap \displaystyle\bigcap_{0\leq j< k}\ker(\mathfrak{C}^j_{\sigma,2l+1})$,
\begin{align*}
\|\varphi_f\|_{C^{n+\pa}(J)}&\leq \widetilde{C}^r_{\sigma,l}\|f\|_{C^{k_r}(M)}+C_J^{n+1}\|f\|_{C^{n+1}(M)}\\
&\leq (\widetilde{C}^r_{\sigma,l}+C_J^{n+1})\|f\|_{C^{k_r}(M)}.
\end{align*}


\begin{figure}[h!]
 \includegraphics[width=0.6\textwidth]{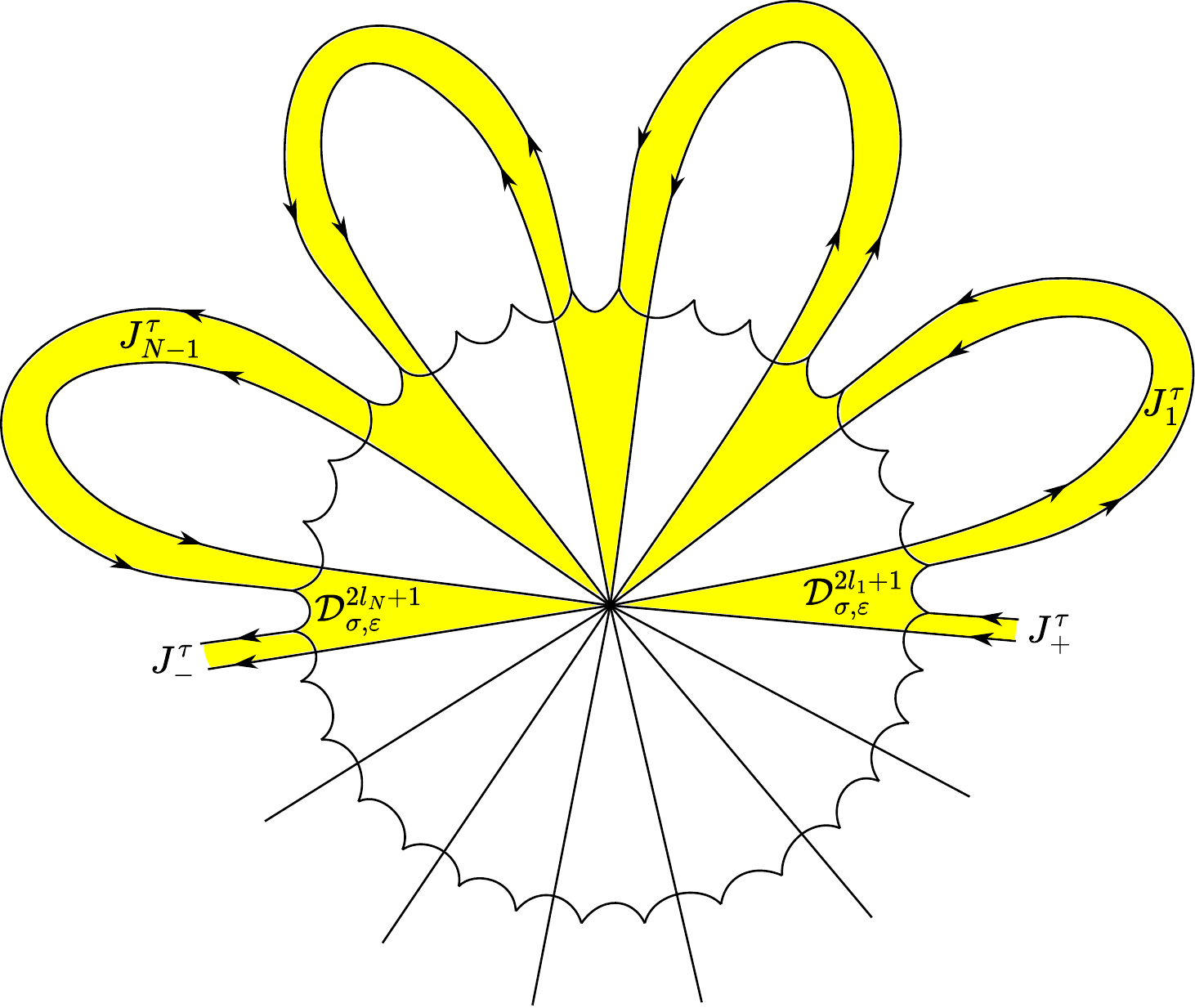}
 \caption{{The decomposition of $J^\tau$ along the bouquet of $N-1$ saddle loops.} \label{fig:krolik}}
\end{figure}

\textbf{Case 3.} Suppose that $J\subset I_\alpha$ is of the form $J=[l_\alpha,l_\alpha+\vep]$, where $l_\alpha$ is the first backward meeting point of a separatrix incoming to $\sigma\in\mathrm{Sd}(\psi_\R)\cap M'$.
Suppose that $\sigma$ has some saddle loops and $J^\tau$ meets $\sigma$ $N$-times ($1<N=N_J<m_\sigma$). Then all orbits starting from $\Int J$ meet the set $\mathcal{D}_{\sigma,\vep}$ $N$-times  before return to $I$. Assume that each such orbit meets its sectors $\mathcal{D}_{\sigma,\vep}^{2l_i+1}$ for $1\leq i\leq N$ consecutively. Then $\sigma$ has $N-1$ saddle loops $sl_i$ connecting the sector $\mathcal{D}_{\sigma,\vep}^{2l_i+1}$ with $\mathcal{D}_{\sigma,\vep}^{2l_{i+1}+1}$ for $1\leq i\leq N-1$. In particular,
\begin{equation}\label{eq:coverN}
J^\tau=J_+^\tau\cup J_-^\tau\cup \bigcup_{i=1}^N \mathcal{D}_{\sigma,\vep}^{2l_i+1}\cup\bigcup_{i=1}^{N-1}J_i^\tau,
\end{equation}
where $J^\tau_i=\{\psi_t\gamma_i(s):s\in[0,\vep],t\in[0,t_i]\}$ is a rectangle whose base is a $C^\infty$-curve $\gamma_i([0,\vep])\subset\mathcal{D}_{\sigma,\vep}^{2l_i+1}$ with a standard parametrization while its left side $\{\psi_t\gamma_i(0):t\in[0,t_i]\}$ is a part of the loop $sl_i$, {see Figure~\ref{fig:krolik}}. Using a partition of unity associated to the cover \eqref{eq:coverN} and repeating the arguments used in Case 1 and 2, for every $r>0$ we get $C_J^{r}>0$ such that
\begin{equation}\label{eq:estnpa1}
\|\varphi_f\|_{C^{n+\pa}(J)}\leq C_J^{r}\|f\|_{C^{k_r}(M)}\text{ for }f\in C^{k_r}(M)\cap\bigcap_{1\leq i\leq N_J}\bigcap_{0\leq j< k}\ker(\mathfrak{C}^j_{\sigma,2l_i+1}).
\end{equation}
\medskip

\textbf{Case 4.} Suppose that $J\subset I_\alpha$ is of the form $J=[r_\alpha-\vep,r_\alpha]$, where $r_\alpha$ is the first backward meeting point of a separatrix incoming to $\sigma\in\mathrm{Sd}(\psi_\R)$.
Suppose that $J^\tau$ meets $\sigma$ $N$-times ($N=N_J$) and the orbits starting from $\Int J$ meet the set $\mathcal{D}_{\sigma,\vep}^{2l_i}$ for $1\leq i\leq N$ consecutively  before return to $I$.
Then repeating the arguments used in Case 1, 2 and 3, for every $r>0$ we get $C_J^{r}>0$ such that
\begin{equation}\label{eq:estnpa2}
\|\varphi_f\|_{C^{n+\pa}(J)}\leq C_J^{r}\|f\|_{C^{k_r}(M)}\text{ for }f\in C^{k_r}(M)\cap\bigcap_{1\leq i\leq N_J}\bigcap_{0\leq j< k}\ker(\mathfrak{C}^j_{\sigma,2l_i}).
\end{equation}
\medskip

\textbf{Final step.} We can find a finite family of closed subintervals  $\{J_q\}_{q=1}^Q$ of $I$ which covers the whole interval $I$ and such that every $J_q$ is of the form $[l_\alpha,l_\alpha+\vep]$ (or $[r_\alpha-\vep,r_\alpha]$) with $\min\tau_\pm(J_q)>\max\tau(J_q)$, or $J_q\subset \Int I_\alpha$ with $2\min\tau(J_q)>\max\tau(J_q)$.

If $J_q$ is an interval of the form $[l_\alpha,l_\alpha+\vep]$ or $[r_\alpha-\vep,r_\alpha]$ then by \eqref{eq:estnpa1} and \eqref{eq:estnpa2},
\begin{align*}
\|\varphi_f\|_{C^{n+\pa}(J_q)}\leq C_{J_q}^{r}\|f\|_{C^{k_r}(M)}\text{ for }f\in C^{k_r}(M)\cap\bigcap_{\substack{(\sigma,j,l)\in \mathscr{TC}\\ \mathfrak{o}(\sigma,j)<r}}\ker(\mathfrak{C}^j_{\sigma,l}).
\end{align*}
If $J_q\subset \Int I_\alpha$ then, by Remark~\ref{rmk:filpa} and \eqref{eq:estnpa0},
\begin{align*}
\|\varphi_f\|_{C^{n+\mathrm{P_{a}}}(J_q)}&\leq\|\varphi_f\|_{C^{n+1}(J_q)}\leq C_{J_q}^{n+1}\|f\|_{C^{n+1}(M)}\leq C_{J_q}^{n+1}\|f\|_{C^{k_r}(M)}
\end{align*}
 for all $f\in C^{k_r}(M)$.
This yields $\varphi_f\in C^{n+\pa}(\sqcup_{\alpha \in \mathcal{A}}I_\alpha)$ and
\[\|\varphi_f\|_{C^{n+\mathrm{P_{a}}}}\leq \sum_{q=1}^Q\|\varphi_f\|_{C^{n+\mathrm{P_{a}}}(J_q)}\leq C \|f\|_{C^{k_r}(M)}\]
for all $f\in C^{k_r}(M)$ such that $\mathfrak{C}^j_{\sigma,l}(f)=0$ for  $(\sigma,j,l)\in \mathscr{TC}$ with $\mathfrak{o}(\sigma,j)<r$.
\medskip

Recall that, by assumption, the right end of $I$ is the first meeting point of a separatrix (that is not a saddle connection) emanating by a fixed point $\sigma$ (incoming or outgoing) with the
interval $I$. Suppose that the right end is the first backward meeting point of a separatrix incoming to $\sigma$.

Let $\alpha=\pi_1^{-1}(d)$, i.e.\ the interval $I_\alpha=[l_\alpha,r_\alpha)$ is the latest after the exchange. It follows that for every $0<\vep<|I_\alpha|$ the strip $[r_\alpha-\vep,r_\alpha]^\tau$ avoids all fixed points, so $\sup\tau([r_\alpha-\vep,r_\alpha])<\infty$. By the continuity of $\tau$, we can choose $\vep>0$ so that  $\max\tau([r_\alpha-\vep,r_\alpha])<2\min\tau([r_\alpha-\vep,r_\alpha])$. In view of Case~1, $\varphi_f\in C^{n+1}([r_\alpha-\vep,r_\alpha])$. Hence,
\[
C_{\alpha,n}^{a,-}(\varphi_f) =\lim_{x\nearrow r_\alpha}D^{n+1} \varphi_f(x)(r_\alpha-x)^{1+a}=0.\]
The same argument shows that if the right end is the first forward meeting point of a separatrix outgoing from $\sigma$,
then $C_{\alpha,n}^{a,-}(\varphi_f)=0$ for $\alpha=\pi_0^{-1}(d)$. Finally we have $C_{\pi_0^{-1}(d),n}^{a,-}(\varphi_f) \cdot C_{\pi_1^{-1}(d),n}^{a,-}(\varphi_f)=0$. Analyzing the orbit of the left end in the same way, we also get $C_{\pi_0^{-1}(1),n}^{a,+}(\varphi_f) \cdot C_{\pi_1^{-1}(1),n}^{a,+}(\varphi_f)=0$, which shows that  $\varphi_f\in C^{n+\pag}(\sqcup_{\alpha \in \mathcal{A}}I_\alpha)$.
\end{proof}

For all $(\sigma,k,j)\in \mathscr{TD}$ let ${\chi}^k_{\sigma,j}:M\to\C$ be a $C^\infty$-map such that ${\chi}^k_{\sigma,j}(\omega,\overline{\omega})=\omega^j\overline{\omega}^{k-j}/(k!V(\omega,\overline{\omega}))$ on $U_\sigma$ and it is equal to zero on all  $U_{\sigma'}$ for $\sigma'\neq \sigma$. By definition,
$\mathfrak{d}^k_{\sigma,j}({\chi}^k_{\sigma,j})=1$ and $\mathfrak{d}^{k'}_{\sigma',j'}({\chi}^k_{\sigma,j})=0$ if $(\sigma',k',j')\neq (\sigma,k,j)$.

In view of Theorem~\ref{thm1}, we get the following result.

\begin{corollary}\label{cor;fdecomp}
For every $r\geq -\frac{m-2}{m}$ and any   $f\in C^{k_r}(M)$ we have a decomposition
\begin{equation}\label{eq:Rr}
f=\sum_{\substack{(\sigma,k,j)\in \mathscr{TD}\\ \mathfrak{o}(\sigma,k)<r}}
\mathfrak{d}^k_{\sigma,j}(f)\chi^k_{\sigma,j}+\mathfrak{R}_{r}(f)
\end{equation}
such that $\varphi_{\mathfrak{R}_{r}(f)}\in C^{n+\pag}(\sqcup_{\alpha \in \mathcal{A}}I_\alpha)$ with $n=\lceil r\rceil$ and $a=n-r$. Moreover, the operators $\mathfrak{R}_{r}:C^{k_r}(M)\to C^{k_r}(M)$ and
$C^{k_r}(M)\ni f\mapsto \varphi_{\mathfrak{R}_{r}(f)}\in C^{n+\pag}(\sqcup_{\alpha \in \mathcal{A}}I_\alpha)$
are bounded.
\end{corollary}

Let us consider an equivalence relation $\sim$ on $\mathscr{TC}$ as follows:
 $(\sigma,k,l)\sim(\sigma,k,l')$ if the angular sectors $U_{\sigma,l}$ and $U_{\sigma,l'}$ are connected through a chain of saddle loops emanating from the saddle $\sigma$. For every equivalence class $[(\sigma,k,l)]\in \mathscr{TC}/\sim$, let
\[\mathfrak{C}_{[(\sigma,k,l)]}(f):=\sum_{(\sigma,k,l')\sim (\sigma,k,l)}\mathfrak{C}^k_{\sigma,l}(f).\]
For any $[(\sigma,k,l)]\in \mathscr{TC}/\sim$ there exists $\alpha\in\mathcal{A}$ and an interval $J$ of the form $[l_\alpha,l_\alpha+\vep]$ or $[r_\alpha-\vep,r_\alpha]$ such that $l_\alpha$ or $r_\alpha$ is
the first backward meeting point of a separatrix incoming to $\sigma\in\mathrm{Sd}(\psi_\R)$ and $J^\tau$ contains all angular sectors $U_{\sigma,l'}$ for which $(\sigma,k,l')\sim(\sigma,k,l)$. Let $\xi_{[(\sigma,k,l)]}:I\to\R$ be given as follows:
\begin{itemize}
\item $\xi_{[(\sigma,k,l)]}$ is zero on any interval $I_\beta$ with $\beta\neq \alpha$;
\item if $J=[l_\alpha,l_\alpha+\vep]$ then for any $s\in I_\alpha$,
\begin{align*}
\xi_{[(\sigma,k,l)]}(s)&=\frac{(s-l_\alpha)^{\frac{k-(m_\sigma-2)}{m_\sigma}}}{m_\sigma^2 k!}\text{ if }k\neq m_\sigma-2\ \operatorname{mod}m_\sigma,\\
\xi_{[(\sigma,k,l)]}(s)&=-\frac{(s-l_\alpha)^{\frac{k-(m_\sigma-2)}{m_\sigma}}\log(s-l_\alpha)}{m_\sigma^2 k!} \text{ if }k= m_\sigma-2\ \operatorname{mod}m_\sigma;
\end{align*}
\item if $J=[r_\alpha-\vep,r_\alpha]$ then for any $s\in I_\alpha$,
\begin{align*}
\xi_{[(\sigma,k,l)]}(s)&=\frac{(r_\alpha-s)^{\frac{k-(m_\sigma-2)}{m_\sigma}}}{m_\sigma^2 k!}\text{ if }k\neq m_\sigma-2\ \operatorname{mod}m_\sigma,\\
\xi_{[(\sigma,k,l)]}(s)&=-\frac{(r_\alpha-s)^{\frac{k-(m_\sigma-2)}{m_\sigma}}\log(r_\alpha-s)}{m_\sigma^2 k!} \text{ if }k= m_\sigma-2\ \operatorname{mod}m_\sigma.
\end{align*}
\end{itemize}
Of course, $\xi_{[(\sigma,k,l)]}\in C^{n+\pag}(\sqcup_{\alpha \in \mathcal{A}}I_\alpha)$ with $n:=\lceil\mathfrak{o}(\sigma,k)\rceil$ and $a:=n-\mathfrak{o}(\sigma,k)$.

In view of the proof of  Theorem~\ref{thm1} we also have the following.
\begin{corollary}\label{cor:vfzan}
Fix $\sigma\in\mathrm{Sd}(\psi_\R)\cap M'$, $k\geq 0$ and let $n:=\lceil\mathfrak{o}(\sigma,k)\rceil$ and $a:=n-\mathfrak{o}(\sigma,k)$. Suppose that $f\in C^{k\vee (n+1)}(M)$ is such that it is equal to zero
on  $U_{\sigma'}$ for $\sigma'\neq \sigma$. Then
\begin{equation}\label{eq:sumspe}
\varphi_f=\sum_{\substack{[(\sigma,j,l)]\in \mathscr{TC}/\sim\\0\leq j<k}}\mathfrak{C}_{[(\sigma,j,l)]}(f)\xi_{[(\sigma,j,l)]}+ C^{n+\pag}(\sqcup_{\alpha \in \mathcal{A}}I_\alpha).
\end{equation}
\end{corollary}

\begin{proof}
The proof proceeds in the same way as the proof of Theorem~\ref{thm1}, except that we use Corollary~\ref{cor:Hold1} instead of Theorem~\ref{cor:Hold} in the key reasoning. For example, using the notations introduced in the proof of the Theorem \ref{thm1}, for any $s \in (0,1]$
\[\varphi^\sigma_f(l_\alpha+\vep s)=\frac{\vep^{-\frac{m_\sigma-2}{m_\sigma}}}{m_\sigma^2}\varphi_{\tilde{f},l}(-s)\text{ with }
\tilde{f}(\omega,\overline{\omega})=(\rho_\sigma\cdot f\cdot V)(\vep^{\frac{1}{m_\sigma}}\omega,\vep^{\frac{1}{m_\sigma}}\overline{\omega})\]
and $\mathscr{C}^{j}_{2l+1}(\tilde f)=\vep^{\frac{j}{m_\sigma}}\mathscr{C}^{j}_{2l+1}(f\cdot V)=\vep^{\frac{j}{m_\sigma}}\mathfrak{C}^j_{\sigma,2l+1}(f)$. In view of Corollary~\ref{cor:Hold1}, for $s\in[-1,0)$,
\begin{align*}
\varphi_{\tilde{f},l}(s)&=-\sum_{\substack{0\leq j<k\\j= m_\sigma-2\operatorname{mod} m_\sigma
}}\frac{\mathscr{C}^j_{2l+1}(\tilde{f})}{j!}(-s)^{\frac{j-(m_\sigma-2)}{m_\sigma}}\log (-s)\\
&\quad+\sum_{\substack{0\leq j<k\\j\neq m_\sigma-2\operatorname{mod} m_\sigma
}}\frac{{\mathscr{C}}^j_{2l+1}(\tilde{f})}{j!}(-s)^{\frac{j-(m_\sigma-2)}{m_\sigma}}+C^{n+\pa}([-1,0)).
\end{align*}
It follows that for $s\in(l_\alpha,l_\alpha+\vep]$,
\begin{align*}
\varphi^\sigma_f(s)&=\frac{\vep^{-\frac{m_\sigma-2}{m_\sigma}}}{m_\sigma^2}\varphi_{\tilde{f},l}((l_\alpha-s)/\vep)\\
&=-\frac{\vep^{-\frac{m_\sigma-2}{m_\sigma}}}{m_\sigma^2}\sum_{\substack{0\leq j<k\\j= m_\sigma-2\operatorname{mod} m_\sigma
}}\frac{\vep^{\frac{j}{m_\sigma}}\mathfrak{C}^j_{\sigma,2l+1}(f)}{j!}\Big(\frac{s-l_\alpha}{\vep}\Big)^{\frac{j-(m_\sigma-2)}{m_\sigma}}\log \Big(\frac{s-l_\alpha}{\vep}\Big)\\
&\quad+\frac{\vep^{-\frac{m_\sigma-2}{m_\sigma}}}{m_\sigma^2}\!\!\sum_{\substack{0\leq j<k\\j\neq m_\sigma-2\operatorname{mod} m_\sigma
}}\frac{\vep^{\frac{j}{m_\sigma}}\mathfrak{C}^j_{\sigma,2l+1}(f)}{j!}\Big(\frac{s-l_\alpha}{\vep}\Big)^{\frac{j-(m_\sigma-2)}{m_\sigma}}+C^{n+\pa}((l_\alpha,l_\alpha+\vep])\\
&=-\sum_{\substack{0\leq j<k\\j= m_\sigma-2\operatorname{mod} m_\sigma
}}\!\!\mathfrak{C}^j_{\sigma,2l+1}(f)\frac{(s-l_\alpha)^{\frac{j-(m_\sigma-2)}{m_\sigma}}\log (s-l_\alpha)}{m_\sigma^2j!}\\
&\quad+\sum_{\substack{0\leq j<k\\j\neq m_\sigma-2\operatorname{mod} m_\sigma
}}\mathfrak{C}^j_{\sigma,2l+1}(f)\frac{(s-l_\alpha)^{\frac{j-(m_\sigma-2)}{m_\sigma}}}{m_\sigma^2j!}+C^{n+\pa}((l_\alpha,l_\alpha+\vep]).
\end{align*}
This key observation makes it possible to get \eqref{eq:sumspe} proceeding further as in the proof of Theorem~\ref{thm1}.
\end{proof}

\begin{theorem}
For any $r\geq-\frac{m-2}{m}$ let $n=\lceil r\rceil$ and $a=n-r$. Then for any $f\in C^{k_r}(M)$ we have
\begin{align}
\label{def:rsk}
\mathfrak{s}_{r}(f)&=\varphi_{f}-\sum_{\substack{[(\sigma,k,l)]\in \mathscr{TC}/\sim\\ \mathfrak{o}(\sigma,k)<r}}
\mathfrak{C}_{[(\sigma,k,l)]}(f)\xi_{[(\sigma,k,l)]}\in C^{n+\pa}(\sqcup_{\alpha \in \mathcal{A}}I_\alpha)
\end{align}
and the operator $\mathfrak{s}_{r}:C^{k_r}(M)\to C^{n+\pa}(\sqcup_{\alpha \in \mathcal{A}}I_\alpha)$ is bounded.
\end{theorem}

\begin{proof}
Let $\{\rho_\sigma:\sigma\in \mathrm{Sd}(\psi_\R)\cap M'\}$ be a $C^\infty$-partition of unity of $M$ such that $\rho_\sigma=1$ on $U_\sigma$.
For any $\sigma\in \mathrm{Sd}(\psi_\R)\cap M'$ choose $k_\sigma\geq 1$ so that $\mathfrak{o}(\sigma,k_\sigma-1)<r\leq \mathfrak{o}(\sigma,k_\sigma)$. Let $n_{\sigma,k_\sigma}=\lceil \mathfrak{o}(\sigma,k_\sigma)\rceil$ and $a_{\sigma,k_\sigma}=n_{\sigma,k_\sigma}-\mathfrak{o}(\sigma,k_\sigma)$. By Remark~\ref{rmk:rkr1}, $k_\sigma\vee(n_{\sigma,k_\sigma}+1)\leq k_r$. Therefore, Corollary~\ref{cor:vfzan} applied to $f\cdot \rho_\sigma$, shows that
\[\varphi_{f\cdot \rho_\sigma}=\sum_{\substack{[(\sigma,j,l)]\in \mathscr{TC}/\sim\\0\leq j<k_\sigma}}\mathfrak{C}_{[(\sigma,j,l)]}(f\cdot \rho_\sigma){\xi}_{[(\sigma,j,l)]}+ C^{n_{\sigma,k_\sigma}+\mathrm{P}_{a_{\sigma,k_\sigma}}\mathrm{G}}(\sqcup_{\alpha \in \mathcal{A}}I_\alpha).\]
As $r\leq \mathfrak{o}(\sigma,k_\sigma)$, by Remark~\ref{rmk:filpa}, $C^{n_{\sigma,k_\sigma}+\mathrm{P}_{a_{\sigma,k_\sigma}}\mathrm{G}}\subset C^{n+\pag}$. Since $\mathfrak{C}_{[(\sigma,j,l)]}(f\cdot \rho_\sigma)=\mathfrak{C}_{[(\sigma,j,l)]}(f)$, this gives
\[\varphi_{f\cdot \rho_\sigma}=\sum_{\substack{[(\sigma,j,l)]\in \mathscr{TC}/\sim\\0\leq j<k_\sigma}}\mathfrak{C}_{[(\sigma,j,l)]}(f){\xi}_{[(\sigma,j,l)]}+ C^{n+\pag}(\sqcup_{\alpha \in \mathcal{A}}I_\alpha).\]
When summed against $\sigma$, this yields \eqref{def:rsk}.

To prove that the operator $\mathfrak{s}_{r}$ is bounded, we use the decomposition \eqref{eq:Rr}. Indeed,
\begin{align*}
\mathfrak{s}_{r}(f)&=\sum_{\substack{(\sigma,k,j)\in \mathscr{TD}\\ \mathfrak{o}(\sigma,k)<r}}
\mathfrak{d}^k_{\sigma,j}(f)\mathfrak{s}_{r}(\chi^k_{\sigma,j})+\mathfrak{s}_{r}(\mathfrak{R}_{r}(f))\\
&=\sum_{\substack{(\sigma,k,j)\in \mathscr{TD}\\ \mathfrak{o}(\sigma,k)<r}}
\mathfrak{d}^k_{\sigma,j}(f)\mathfrak{s}_{r}(\chi^k_{\sigma,j})+\varphi_{\mathfrak{R}_{r}(f)}.
\end{align*}
Since the functionals $\mathfrak{d}^k_{\sigma,j}$ and the operator $f\mapsto \varphi_{\mathfrak{R}_{r}(f)}$ (by Corollary~\ref{cor;fdecomp}) are bounded, this gives that  $\mathfrak{s}_{r}$ is bounded.
\end{proof}

\begin{proof}[Proof of Theorem~\ref{thm2}]
Arguments presented in Section~\ref{sec:twoop} show that if $g\in C^{r}(I)$ is a solution of the cohomological equation $g\circ T-g=\varphi_f$,
then the corresponding function $u=u_{g,f}:M'\setminus (\mathrm{Sd}(\psi_\R)\cup \SL)\to\C$ given by
\[u(x):=g(\psi_tx)-\int_{0}^tf(\psi_sx)\,ds\]
whenever  $\psi_tx\in I$ for some $t\in\R$, is of class $C^{r}$ on $M'\setminus (\mathrm{Sd}(\psi_\R)\cup \SL)$. We need to show that if $\mathfrak{d}^k_{\sigma,j}(f)=0$ for all $(\sigma,k,j)\in\mathscr{TD}$ such that $\widehat{\mathfrak{o}}(\mathfrak{d}^k_{\sigma,j})<v(r)$ and $\mathfrak{C}^k_{\sigma,l}(f)=0$ for all $(\sigma,k,l)\in\mathscr{TC}$ such that ${\mathfrak{o}}(\mathfrak{C}^k_{\sigma,l})<v(r)$ then $u$ has a $C^r$-extension to $M'_e$ and
\begin{equation}\label{maineq}
\|u\|_{C^r(M'_e)}\leq C(\|g\|_{C^{r}(I)}+\|f\|_{C^{k_{v(r)}}(M)}).
\end{equation}
We split the proof of our claim into several steps. In fact, we split $M'_e$ into subsets of two kinds: subsets which are far from saddles and saddle loops, and sets surrounding saddles or saddle loops.
\medskip

\textbf{Step 1. Sets far from saddles and saddle loops.}
We will show that for any compact subset $A\subset M'\setminus (\mathrm{Sd}(\psi_\R)\cup \SL)$ there exists $C_{A}>0$ such that
\begin{equation}\label{eq:Acomp}
\|u\|_{C^r(A)}\leq  C_{A}(\|g\|_{C^{r}(I)}+\|f\|_{C^{k_{v(r)}}(M)}).
\end{equation}
Recall that, by arguments from Section~\ref{sec:twoop}, for any $x_0\in M'\setminus (\mathrm{Sd}(\psi_\R)\cup \SL)$ there exist closed intervals $[\tau_1,\tau_2]$ and $J\subset \Int I$ such that the set $R(x_0)=\{\psi_tx:x\in J,t\in[\tau_1,\tau_2]\}$ is a rectangle in $M'$, i.e.\
the map
\[J\times [\tau_1,\tau_2]\ni(x,t)\mapsto \nu(x,t)=\psi_tx\in R(x_0)\]
is a $C^\infty$-diffeomorphism and $x_0\in \Int R(x_0)$. Moreover,
\[u\circ\nu(x,t)=g(x)+\int_0^tf\circ\nu(x,s)ds\text{ on }J\times [\tau_1,\tau_2].\]
By Remark~\ref{rmk:rkr}, it follows that there exists $C_{x_0}>0$ such that
\[\|u\|_{C^r(R(x_0))}\leq C_{x_0}(\|g\|_{C^{r}(I)}+\|f\|_{C^{r}(M)})\leq C_{x_0}(\|g\|_{C^{r}(I)}+\|f\|_{C^{k_{v(r)}}(M)}).\]
Covering $A$ by a finite number of rectangles, this yields \eqref{eq:Acomp}.
\medskip

\textbf{Step 2. Some sets far from saddles.}
Suppose that $\gamma:[a,b]\to M\setminus\rm{Fix}(\psi_\R)$ is a standard $C^\infty$-parametrization of a curve and $\xi:[a,b]\to\R_{>0}$ is a $C^\infty$ map such that
\[[a,b]^\xi\ni (x,t)\mapsto \nu(x,t)=\psi_tx\in \nu([a,b]^\xi)=:(\gamma[a,b])^\xi\]
is a $C^\infty$-diffeomorphism, where $[a,b]^\xi=\{(x,t):x\in[a,b],0\leq t\leq \xi(x)\}$. Then the arguments used in Step~1 show that if $u\circ \gamma\in C^r([a,b])$ then $u\in C^r([a,b]^\gamma)$ and there exists $C_{\gamma,\xi}>0$ such that
\begin{equation}\label{eq:gxi}
\|u\|_{C^r(\gamma([a,b])^\xi)}\leq  C_{\gamma,\xi}(\|u\circ\gamma\|_{C^{r}([a,b])}+\|f\|_{C^{k_{v(r)}}(M)}).
\end{equation}

\textbf{Step 3. Strips touching saddles and saddle loops and their decomposition.}
From now on we will use a notation introduced in the proof of Theorem~\ref{thm1}. Let $\tau:I\to\R_{>0}\cup\{+\infty\}$ be the first return time map.
Suppose that $J\subset I_\alpha$ is of the form $J=[l_\alpha,l_\alpha+\vep]$, where $l_\alpha$ is the first backward meeting point of a separatrix incoming to $\sigma\in\mathrm{Sd}(\psi_\R)\cap M'$.
Suppose that $J^\tau$ meets $\sigma$ exactly $N$-times ($1\leq N=N_J<m_\sigma$) and the orbits starting from $\Int J$ meet $\mathcal{D}_{\sigma,\vep}$ in  its sectors $\mathcal{D}_{\sigma,\vep}^{2l_i+1}$ for $1\leq i\leq N$ consecutively   before return to $I$. Then $\sigma$ has $N-1$ saddle loops $sl_i$ connecting the sector $\mathcal{D}_{\sigma,\vep}^{2l_i+1}$ with $\mathcal{D}_{\sigma,\vep}^{2l_{i+1}+1}$ for $1\leq i\leq N-1$.
Recall that $J^\tau$ is the closure of $(\Int J)^\tau$.
Then
\begin{equation}\label{eq:coverN2}
J^\tau= \bigcup_{i=1}^N \mathcal{D}_{\sigma,\vep}^{2l_i+1}\cup\bigcup_{i=0}^{N}E_i,
\end{equation}
where each $E_i$ is of the form $\gamma_i([0,\vep])^{\xi_i}$ with
\begin{itemize}
\item $\gamma_0(s)=\gamma(l_\alpha+s)$ (here $\gamma$ is the parametrization of $I$)
and $\xi_0(s)$ is the time spent to go from $J$ to  $\mathcal{D}_{\sigma,\vep}^{2l_1+1}$;
\item  for $0\leq i\leq N-1$, $\gamma_i(s)=G_{l_i}(\vep-\iota s)$
and $\xi_i(s)$ is the time spent to go from $\mathcal{D}_{\sigma,\vep}^{2l_i+1}$ to  $\mathcal{D}_{\sigma,\vep}^{2l_{i+1}+1}$;
\item   $\gamma_N(s)=G_{l_N}(\vep-\iota s)$
and $\xi_N(s)$ is the time spent to go from $\mathcal{D}_{\sigma,\vep}^{2l_N+1}$ to  $I$.
\end{itemize}

\textbf{Step 4.0. The set $E_0$.}
In view of \eqref{eq:gxi} in Step~2,
\begin{equation}\label{eq:gxi1}
u\in C^r(E_0)\text{ and }\|u\|_{C^r(E_0)}\leq  C_{\gamma_0,\xi_0}(\|g\|_{C^{r}(I)}+\|f\|_{C^{k_{v(r)}}(M)}).
\end{equation}

\textbf{Step 4.1. The sets $\mathcal{D}_{\sigma,\vep}^{2l_i+1}$ surrounding the saddle $\sigma$.} We will show that for every $1\leq i\leq N$ there exist $C_i,C'_i>0$ such that  if $u$ has a $C^r$-extension on $E_{i-1}$ then
it has $C^r$-extension on $\mathcal{D}_{\sigma,\vep}^{2l_i+1}$ and
\begin{equation}\label{eq:aaa3}
\|u\|_{C^r({\mathcal{D}}_{\sigma,\vep}^{2l_i+1})}\leq C_{i}\|u\|_{C^r(E_{i-1})}+C'_i\|f\|_{C^{k_{v(r)}}(M)}.
\end{equation}
This is the main inductive step running to the proof of \eqref{maineq} restricted to $J^\tau$.

By Remark~\ref{rmk:Feps}, for every $\omega\in \mathcal{D}_{\sigma,\vep}(\frac{l_i}{m_\sigma},\frac{l_i+1}{m_\sigma})$ we have $\psi_{\xi_l(\omega)}G_{l_i}(-\vep-\iota s)=\omega$ for $s=-\Im \omega^{m_\sigma}\in[0,\vep]$ and
\[u(\omega)-u( G_{l_i}(-\vep-\iota s))=\int_{0}^{\xi_l(\omega)}f(\psi_t G_{l_i}(-\vep-\iota s))dt.\]
In view of \eqref{eq:intpass}, for  $\omega\in \mathcal{D}_{\sigma,\vep}(\frac{2l_i+1}{2m_\sigma},\frac{2l_i+2}{2m_\sigma})$,
\begin{equation}\label{eq:aaa0}
u(\omega)-u( G_{l_i}(-\vep+\iota\Im \omega^{m_\sigma}))=\frac{\vep^{-\frac{m_\sigma-2}{m_\sigma}}}{m_\sigma^2}{F}_{(f\cdot V)\circ \vep^{1/m_{\sigma}}}(\vep^{-1/m_\sigma}\omega).
\end{equation}
Choose $m-1\leq\underline{k}\leq k\leq k_{v(r)}$ such that
\[
\mathfrak{o}(\sigma,{k}-1)<v(r)\leq \mathfrak{o}(\sigma,{k}) \text{ and } \widehat{\mathfrak{o}}(\sigma,\underline{k}-1)<v(r)\leq \widehat{\mathfrak{o}}(\sigma,\underline{k}).
\]
Then $\widehat{\mathfrak{o}}(\sigma,\underline{k})=\lceil\mathfrak{o}(\sigma,{k})\rceil$. Moreover, by Remark~\ref{rmk:rkr1}, $n:=\lceil v(r)\rceil=\lceil\mathfrak{o}(\sigma,{k})\rceil$ and $k\vee(n+1)\leq k_{v(r)}$.

By assumption, for every $0\leq j<k$ and $0\leq i\leq j\wedge(m_\sigma-2)$ with $i\neq j-(m_\sigma-1)\operatorname{mod} m_\sigma$ we have $\partial_i^j(f\cdot V)=\mathfrak{d}_{\sigma,i}^j(f)=0$ and $\mathscr{C}^j_l(f\cdot V)=\mathfrak{C}_{\sigma,l}^j(f)=0$ for all $0\leq j<k$ and $l=2l_i+1$, $1\leq i\leq N$.

In view of Theorem~\ref{thm;ext},
 the map $F_{(f\cdot V)\circ \vep^{1/m_\sigma}}\circ\vep^{-1/m_\sigma}:\mathcal{D}_{\sigma,\vep}(\frac{2l_i+1}{2m_\sigma}, \frac{2l_i+2}{2m_\sigma})\to\C$ has a
$C^{\mathfrak{e}(\sigma,k)}$-extension on ${\mathcal{D}}_{\sigma,\vep}^{2l_i+1}=\overline{\mathcal{D}}_{\sigma,\vep}(\frac{2l_i+1}{2m_\sigma}, \frac{2l_i+2}{2m_\sigma})$ and there exists $C'_i>0$ such that
\begin{align}\label{eq:aaa1}
\Big\|\frac{\vep^{-\frac{m_\sigma-2}{m_\sigma}}}{m_\sigma^2}{F}_{(f\cdot V)\circ \vep^{1/m_{\sigma}}}\circ \vep^{-1/m_\sigma}\Big\|_{C^{\mathfrak{e}(\sigma,k)}({\mathcal{D}}_{\sigma,\vep}^{2l_i+1})}\leq C'_i\|f\|_{C^{k\vee(n+1)}(M)}.
\end{align}
Moreover, the map ${\mathcal{D}}_{\sigma,\vep}(\frac{2l_i+1}{2m_\sigma}, \frac{2l_i+2}{2m_\sigma})\ni\omega \mapsto G_{l_i}(-\vep+\iota \Im \omega^{m_\sigma})\in {\mathcal{D}}_{\sigma,\vep}^{2l_i+1}\cap E_{i-1}$ has an obvious analytic extension on ${\mathcal{D}}_{\sigma,\vep}^{2l_i+1}$.
It follows that there exists $C_i>0$ such that if $u$ is of class $C^r$ on $E_{i-1}$ then  $u\circ G_{l_i}(-\vep+\iota\Im \omega^{m_\sigma})$
has a $C^r$-extension to ${\mathcal{D}}_{\sigma,\vep}^{2l_i+1}$ and
\begin{equation}\label{eq:aaa2}
\|u\circ G_{l_i}(-\vep+\iota\Im\omega^{m_\sigma})\|_{C^r({\mathcal{D}}_{\sigma,\vep}^{2l_i+1})}\leq C_i\|u\|_{C^r(E_{i-1})}.
\end{equation}
As $v(r)\leq \mathfrak{o}(\sigma,k)$ and $k\vee(n+1)\leq k_{v(r)}$, by \eqref{eq:aaa0},  \eqref{eq:aaa1} and \eqref{eq:aaa2},
$u$ has a $C^r$-extension on  ${\mathcal{D}}_{\sigma,\vep}^{2l_i+1}$ and \eqref{eq:aaa3} holds.
\medskip

\textbf{Step 4.2. The sets $E_i$ surrounding the saddle loops.}
We will show that for every $1\leq i\leq N$ there exist $C''_i,C'''_i>0$ such that  if $u$ has a $C^r$-extension on $\mathcal{D}_{\sigma,\vep}^{2l_i+1}$ then it has $C^r$-extension on $E_i$ and
\begin{equation*}
\|u\|_{C^r(E_i)}\leq C''_{i}\|u\|_{C^r({\mathcal{D}}_{\sigma,\vep}^{2l_i+1})}+C'''_i\|f\|_{C^{k_{v(r)}}(M)}.
\end{equation*}
This is an easy inductive step leading to the proof of \eqref{maineq} restricted to $J^\tau$, which follows directly from \eqref{eq:gxi}.
Indeed, as  $\gamma_i:[0,\vep]\to {\mathcal{D}}_{\sigma,\vep}^{2l_i+1}$ is an analytic curve,
there exists $C>0$  such that if $u$ is of class $C^r$ on ${\mathcal{D}}_{\sigma,\vep}^{2l_i+1}$ then $\|u\circ \gamma_i\|_{C^r([0,\vep])}\leq C\|u\|_{C^r({\mathcal{D}}_{\sigma,\vep}^{2l_i+1})}$. As $E_i=\gamma_i([0,\vep])^{\xi_i}$, in view of \eqref{eq:gxi}, $u$ has $C^r$-extension on $E_i$ and
\begin{align*}
\|u\|_{C^r(E_i)}&\leq C_{\gamma_i,\xi_i}(\|u\circ \gamma_i\|_{C^r([0,\vep])}+\|f\|_{C^{k_{v(r)}}(M)})\\
&\leq C_{\gamma_i,\xi_i}(C\|u\|_{C^r({\mathcal{D}}_{\sigma,\vep}^{2l_i+1})}+\|f\|_{C^{k_{v(r)}}(M)}).
\end{align*}

\textbf{Step 4.3. Induction.} Starting from Step~4.0 (as the initial inductive step) and then repeating alternately Steps~4.1 and 4.2 $N$-times, we have
that there exists $C_J>0$ such that $u$ has a $C^r$-extension on $J^\tau$ and
\begin{equation}\label{eq:Jcomp}
\|u\|_{C^r(J^\tau)}\leq  C_{J}(\|g\|_{C^{r}(J)}+\|f\|_{C^{{k}_{v(r)}}(M)}).
\end{equation}

\textbf{Step 5. Summary.} Using the arguments from Step~4, we obtain \eqref{eq:Jcomp} also in the case where $J=[r_\alpha-\vep,r_\alpha]$. Then the strip $J^\tau$ touches a saddle on right side. Let $A\subset M'\setminus (\mathrm{Sd}(\psi_\R)\cup \SL)$ be the closure of
\[M'\setminus \bigcup_{\alpha\in\mathcal{A}}([l_\alpha,l_\alpha+\vep]^\tau\cup[r_\alpha-\vep,r_\alpha]^\tau).\]
Then by Step~1 applied to $A$ and Step~4 applied to the intervals  $[l_\alpha,l_\alpha+\vep]$ and $[r_\alpha-\vep,r_\alpha]$ for all $\alpha\in\mathcal{A}$, we have that $u$ has a $C^r$-extension on $M'_e$ and \eqref{maineq} holds with $C=C_A+\sum_{\alpha\in\mathcal{A}}(C_{[l_\alpha,l_\alpha+\vep]}+C_{[r_\alpha-\vep,r_\alpha]})$.
\end{proof}

{Finally, we prove the optimality of the regularity of solutions for cohomological equation. The proof of Theorem \ref{thm3}  reduces our problem to a local framework and then applies Theorem~\ref{thm;extinv} (proved in Section~\ref{sec;laphi}) knowing that $\mathscr{C}^j_{l}$ corresponds to $\mathfrak{C}^{j}_{\sigma,l}$ and $\partial^{j}_{i}$ corresponds  to $\mathfrak{d}^{j}_{\sigma,i}$ in the local context. If $u$ is a $C^r$-solution of the cohomological equation on an angular sector around a saddle, then we compare the value of $u$ at any point of the sector with the value of $u$ at the entry point of its orbit into the sector. As both are of class $C^r$, their difference is also of class $C^r$. Moreover, this difference is related to $F_f$ on the sector. This is what the mentioned reduction is all about.}

\begin{proof}[Proof of Theorem~\ref{thm3}]
Suppose that there exists $u\in C^{r}(M'_e)$ such that $Xu=f$ for some $r\in\R_\eta$ with $v(r)>0$.
Choose $\sigma\in \mathrm{Sd}(\psi_\R)\cap M'$ and  $0\leq l<m_\sigma$ such that $U_{\sigma,2l+1}\cap M'\neq \emptyset$.
We will show that $\mathfrak{C}^{j}_{\sigma,2l+1}(f) = 0$ for all $j\geq 0$ such that $\mathfrak{o}(\sigma,j)<v(r)$ and
$\mathfrak{d}^{j}_{\sigma,i}(f) = 0$ for all $j\geq 0$ such that $\widehat{\mathfrak{o}}(\sigma,j)<v(r)$ and $0\leq i\leq j\wedge(m_\sigma-2)$ with
$i\neq j-(m_\sigma-1)\operatorname{mod}m_\sigma$.
The proof for even sectors follows the same way as for odd sectors, so we will only focus on the latter.

In view of \eqref{eq:aaa0}, for  $\omega\in \mathcal{D}_{\sigma,\vep}(\frac{2l+1}{2m_\sigma},\frac{2l+2}{2m_\sigma})$,
\[u(\omega)-u( G_{l}(-\vep+\iota\Im \omega^{m_\sigma}))=\frac{\vep^{-\frac{m_\sigma-2}{m_\sigma}}}{m_\sigma^2}{F}_{(f\cdot V)\circ \vep^{1/m_{\sigma}}}(\vep^{-1/m_\sigma}\omega).\]
By assumption, $u$ is of class $C^r$ on $\overline{\mathcal{D}}_{\sigma,\vep}(\frac{2l+1}{2m_\sigma},\frac{2l+2}{2m_\sigma})$, and hence
$u( G_{l}(-\vep+\iota\Im \omega^{m_\sigma}))$ is of class $C^r$ on $\overline{\mathcal{D}}_{\sigma,\vep}(\frac{2l+1}{2m_\sigma},\frac{2l+2}{2m_\sigma})$. Therefore, the map ${F}_{(f\cdot V)\circ \vep^{1/m_{\sigma}}}\circ \vep^{-1/m_\sigma}$ has a $C^r$-extension on  $\overline {\mathcal D}(\frac{2l+1}{2m}, \frac{2l+2}{2m})$.

Choose $k\geq m_\sigma-1$ such that $\mathfrak{o}(\sigma,k-1)<v(r)\leq\mathfrak{o}(\sigma,k)$. By Remark~\ref{rmk:rkr1}, we have $n:=\lceil v(r)\rceil=\lceil\mathfrak{o}(\sigma,k)\rceil$ and $k\vee(n+1)\leq k_{v(r)}$.
Therefore, by Theorem~\ref{thm;extinv},
\[\vep^{j/m_{\sigma}}\mathfrak{C}^{j}_{\sigma,2l+1}(f)=\mathscr{C}^j_{2l+1}((f\cdot V)\circ \vep^{1/m_{\sigma}})=0\]
for all  $j\geq 0$ such that $\mathfrak{o}(\sigma,j)<v(r)$, and
\[\vep^{j/m_{\sigma}}\mathfrak{d}^{j}_{\sigma,i}(f)=\partial^j_i((f\cdot V)\circ \vep^{1/m_{\sigma}})=0\]
for all $j\geq 0$ with $\widehat{\mathfrak{o}}(\sigma,j)<v(r)$ and $0\leq i\leq j\wedge(m_\sigma-2)$ with $i\neq j-(m_\sigma-1)\operatorname{mod} m_\sigma$.
\end{proof}

\section*{Acknowledgements}
The authors would like to thank Alexander Gomilko for his help in understanding some analytical issues used in Section~\ref{sec;prelim} and Giovanni Forni for an interesting discussion of potentially alternative approaches to solving cohomological equations. We are grateful to the referees for their helpful comments and suggestions for improving in the presentation of this work.

The authors acknowledge the Center of Excellence ``Dynamics, mathematical analysis and artificial intelligence'' at the Nicolaus Copernicus University in Toru\'n and  Centro di Ricerca Matematica Ennio De Giorgi - Scuola Normale Superiore, Pisa for hospitality during their visits.
The Research was partially supported by the Narodowe Centrum Nauki Grant 2022/45/B/ST1/00179.

{
\appendix
\section{}
The following technical result is crucial to prove an optimality of regularity in Theorem~\ref{thm;extinv}. Its proof is included for reader's convenience.
\begin{lemma}\label{lem:hom}
Let $0\leq \alpha<\beta<1$, $N\in\N$ and $r>0$. Suppose that $f:\overline{\mathcal{D}}(\alpha,\beta)\to\C$ is a $C^r$-maps such that
\[f(x,y)=\sum_{-N\leq n<r}f_n(x,y)\text{ for }(x,y)\in \overline{\mathcal{D}}(\alpha,\beta)\setminus\{(0,0)\},\]
where $f_n:\overline{\mathcal{D}}(\alpha,\beta)\setminus\{(0,0)\}\to\C$ is a  $C^\infty$-map homogenous of degree $n$ for all $-N\leq n<r$.
Then $f_{-n}\equiv 0$ for $1\leq n\leq N$ and $f_n$ is a homogenous polynomial of degree $n$ for all $0\leq n<r$.
\end{lemma}

\begin{proof}
Take any pair $(i,j)\in\Z_{\geq 0}^2$ with $i+j\leq r$. Then
\[\frac{\partial^{i+j}}{\partial x^i\partial y^j}f(x,y)=\sum_{-N\leq n<r}\frac{\partial^{i+j}}{\partial x^i\partial y^j}f_n(x,y)\text{ for }(x,y)\in \overline{\mathcal{D}}(\alpha,\beta)\setminus\{(0,0)\},\]
where $\frac{\partial^{i+j}}{\partial x^i\partial y^j}f_n:\overline{\mathcal{D}}(\alpha,\beta)\setminus\{(0,0)\}\to\C$ is a  $C^\infty$-map homogenous of degree $n-i-j$ for all $-N\leq n<r$.

Since $\frac{\partial^{i+j}}{\partial x^i\partial y^j}f$ is continuous on $\overline{\mathcal{D}}(\alpha,\beta)$, for any $(x,y)\in \overline{\mathcal{D}}(\alpha,\beta)\setminus\{(0,0)\}$ we have
$\frac{\partial^{i+j}}{\partial x^i\partial y^j}f(sx,sy)\to \frac{\partial^{i+j}}{\partial x^i\partial y^j}f(0,0)$ as $s\searrow 0$.
On the other hand,
\[\frac{\partial^{i+j}}{\partial x^i\partial y^j}f(sx,sy)=\sum_{-N\leq n<r}s^{n-i-j}\frac{\partial^{i+j}}{\partial x^i\partial y^j}f_n(x,y).\]
It follows that for any $(x,y)\in \overline{\mathcal{D}}(\alpha,\beta)\setminus\{(0,0)\}$ we have
\begin{gather}
\frac{\partial^{i+j}}{\partial x^i\partial y^j}f_n(x,y)=0\text{ if }-N\leq n<i+j\text{ and  }\nonumber\\
\frac{\partial^{i+j}}{\partial x^i\partial y^j}f_n(x,y)=\frac{\partial^{i+j}}{\partial x^i\partial y^j}f(0,0)\text{ if } n=i+j.\label{eq:hom}
\end{gather}
Letting $i=j=0$ we get $f_{-k}\equiv 0$ for $1\leq k\leq N$.

Let us consider any $0\leq k<r$. Then for any $(x,y)\in \overline{\mathcal{D}}(\alpha,\beta)\setminus\{(0,0)\}$ and $0<s\leq 1$ we have $f_k(sx,sy)=s^kf_k(x,y)$. Applying the $k$th-order derivative with respect to $s$, we get
\[\sum_{0\leq i\leq k}\binom{k}{i}\frac{\partial^{k}}{\partial x^i\partial y^{k-i}}f_k(sx,sy)x^iy^{k-i}=k!f_k(x,y).\]
In view of \eqref{eq:hom}, we obtain
\[f_k(x,y)=\frac{1}{k!}\sum_{0\leq i\leq k}\binom{k}{i}\frac{\partial^{k}}{\partial x^i\partial y^{k-i}}f(0,0)x^iy^{k-i}.\]
\end{proof}}

\end{document}